\documentclass[a4paper, reqno]{amsart} 
\usepackage[margin=3.5cm]{geometry}

\usepackage{newlfont, color}   
\usepackage{amsthm, amssymb, amsmath, latexsym}
\usepackage{graphicx}
\usepackage{amsfonts, mathrsfs}
\usepackage{esint}
\usepackage{comment}

\theoremstyle{plain}                     
\begingroup
\newtheorem{teo}{Theorem}[section]       
     
\newtheorem{cor}[teo]{Corollary}        
\newtheorem{lem}[teo]{Lemma}      
\endgroup
      
\theoremstyle{definition}                
\begingroup
\newtheorem{defin}[teo]{Definition}

\newtheorem{oss}[teo]{Remark}
\endgroup

\newcommand{\ep}{\varepsilon}
\newcommand{\Oe}{\Omega_\ep}

\newcommand{\R}{\mathbb R}

\newcommand{\C}{\mathbb C}
\newcommand{\M}{\mathbb M}
\newcommand{\ac}{r_\C}
\newcommand{\bc}{R_\C}
\newcommand{\rk}{r_K}
\newcommand{\Rk}{R_K}
\newcommand{\B}[1]{B(#1)}
\newcommand{\A}[0]{\mathbb{A}}

\newcommand{\dist}{\mathrm{dist}}
\newcommand{\supp}{\mathrm{supp}}

\newcommand{\sym}{\mathrm{sym}}
\newcommand{\ms}{\mathbb M^{3\times 3}_{sym}}
\newcommand{\md}{\mathbb M^{3\times 3}_D}

\newcommand{\mthree}{\mathbb{M}^{3\times 3}}

\newcommand{\deb}{\rightharpoonup}

\newcommand{\intom}[1]{\int_{\Omega}{#1\, dx}}
\newcommand{\intome}[1]{\int_{\Omega_{\ep}}{#1\, dx}}
\newcommand{\intomm}[1]{\int_{\omega}{#1\, dx'}}
\newcommand{\intt}[1]{\int_{-\frac{1}{2}}^{\frac{1}{2}}{#1\, dx_3}}

\newcommand{\tep}[0]{\theta^{\ep}}
\newcommand{\etp}[0]{\eta^{\ep}}
\newcommand{\pep}[0]{\phi^{\ep}}
\newcommand{\zep}[0]{z^{\ep}}
\newcommand{\nep}[0]{\nabla_{\ep}}

\newcommand{\eepp}[0]{{\tilde{E}}^{\ep}}

\newcommand{\eep}[0]{E^{\ep}}

\newcommand{\vep}[0]{\varphi^{\ep}}
\newcommand{\bvep}[0]{({\varphi}^{\ep})'}
\newcommand{\bzep}[0]{({z}^{\ep})'}
\newcommand{\vepj}[0]{\varphi^{\ep_j}}
\newcommand{\twh}[0]{\widetilde{W}_{hard}}

\newcommand{\tr}[1]{\text{tr }#1}
\newcommand{\hye}[0]{\hat{y}^{\ep}}
\newcommand{\hpe}[0]{\hat{P}^{\ep}}
\newcommand{\hppe}[0]{\hat{p}^{\ep}}
\newcommand{\sep}[0]{S_{\ep}}
\newcommand{\tepp}[0]{\tep\Big(\frac{y^{\ep}_3}{\ep}\Big)}

\newcommand{\teppp}[0]{\dot{\tep}\Big(\frac{y^{\ep}_3}{\ep}\Big)}
\newcommand{\teppm}[0]{\tep\big(\frac{y^{\ep}_3}{\ep}\big)}

\newcommand{\tepppm}[0]{\dot{\tep}\big(\frac{y^{\ep}_3}{\ep}\big)}
\newcommand{\m}[1]{(#1)'}

\newcommand{\epjt}{\ep_{j}}
\newcommand{\pepjt}[0]{\phi^{\epjt}}
\newcommand{\zepjt}[0]{z^{\epjt}}

\newcommand{\vepjt}[0]{\varphi^{\epjt}}

\newcommand{\bzepjt}[0]{({z}^{\epjt})'}
\newcommand{\tepjt}[0]{\theta^{\epjt}}

\numberwithin{equation}{section}

\begin{document}
\title[Quasistatic evolution models for thin plates in finite plasticity]{Quasistatic evolution models for thin plates arising as low energy $\Gamma$-limits of finite plasticity} 
 
\author[E. Davoli]{Elisa Davoli} 

\address[E. Davoli]{Scuola Internazionale Superiore di Studi Avanzati, via Bonomea 265, 34136 Trieste (Italy); current address: Department of Mathematical Sciences, Carnegie-Mellon University, Pittsburgh, PA, USA}
\email{davoli@sissa.it; edavoli@andrew.cmu.edu}


\maketitle
\begin{abstract}
In this paper we deduce by $\Gamma$-convergence some partially and fully linearized quasistatic evolution models for thin plates, in the framework of finite plasticity. Denoting by $\ep$ the thickness of the plate, we study the case where the scaling factor of the elasto-plastic energy is of order $\ep^{2\alpha-2}$, with $\alpha\geq 3$. We show that solutions to the three-dimensional quasistatic evolution problems converge, as the thickness of the plate tends to zero, to a quasistatic evolution associated to a suitable reduced model depending on $\alpha$.
\end{abstract}

\section{Introduction}
The subject of this paper is the rigorous derivation of quasistatic evolution models for nonlinearly elastic - finitely plastic plates. The problem of deriving lower dimensional models for thin structures has been intensively studied since the early 90's by means of a rigorous approach based on $\Gamma$-convergence \cite{A-B-P, L-R}. Starting from the seminal paper \cite{FJM}, this approach has led to establish a hierarchy of limit models for plates \cite{FJM, FJM2}, rods \cite{M-M3, M-M, S0, S1}, and shells \cite{F-J-M-M, L-M-P, L-M-P2}, in the stationary framework and in the context of nonlinear elasticity.
More recently, the $\Gamma$-convergence approach to dimension reduction has gained attention also in the evolutionary framework: in nonlinear elasticity \cite{AMM}, crack propagation \cite {B, FPZ}, linearized elastoplasticity \cite{DM, LM, LR}, and delamination problems \cite{MRT}. 

In this paper we justify via $\Gamma$-convergence some linearized quasistatic evolution models for a thin plate, whose elastic behaviour is nonlinear and whose plastic response is governed by finite plasticity with hardening. {{We remark that different schools in finite plasticity are still competing and a generally accepted model is still lacking (see e.g. \cite{Ber}). We shall adopt here a mathematical model introduced in \cite{CHM, Men, M}. }}We assume that the reference configuration of the plate is the set
$$\Omega_{\ep}:=\omega\times\big(-\tfrac{\ep}{2},\tfrac{\ep}{2}\big),$$
where $\omega$ is a domain in $\R^2$ and $\ep>0$ represents the thickness of the plate. {{Following the lines of  \cite{Lee} and \cite{Man},}} we consider deformations of the plate $\eta\in W^{1,2}(\Omega_{\ep};\R^3)$ satisfying the multiplicative decomposition
$$\nabla \eta(x)=F_{el}(x)F_{pl}(x)\quad\text{for a.e. }x\in\Oe,$$
where $F_{el}\in L^2(\Omega;\mthree)$ is the elastic strain, $F_{pl}\in L^2(\Omega_{\ep};SL(3))$ is the plastic strain and $SL(3):=\{F\in\mthree:\, \det F=1 \}$. { To guarantee coercivity in the plastic strain variable, we suppose to be in a hardening regime. More precisely,} the stored energy  associated to a deformation $\eta$ and to its elastic and plastic strains is expressed as follows:
\begin{eqnarray*}
\nonumber \cal{E}(\eta,F_{pl})&:=&\intome{W_{el}(\nabla \eta(x) F_{pl}^{-1}(x))}+\intome{W_{hard}(F_{pl}(x))}\\
&=&\intome{W_{el}(F_{el}(x))}+\intome{W_{hard}(F_{pl}(x))},
\end{eqnarray*}
where $W_{el}$ is a frame-indifferent elastic energy density satisfying the standard assumptions of nonlinear elasticity, and $W_{hard}$ describes hardening.
The plastic dissipation is given in terms of a dissipation distance $D:\mthree\times\mthree\to [0,+\infty]$, which is defined via a positively 1-homogeneous potential $H_D$ (see Section \ref{prel}).

We consider a subset $\gamma_d$ of $\partial\omega$ and for every $t\in [0,T]$ we prescribe on $\gamma_d\times \big(-\tfrac{\ep}{2},\tfrac{\ep}{2}\big)$ a boundary datum for the deformations, of the form 
$$\pep(t,x):=\Big(\begin{array}{c}x'\\ x_3\end{array}\Big)+\ep^{\alpha-1}\Big(\begin{array}{c}u^0(t,x')\\0\end{array}\Big)+\ep^{\alpha-2}\Big(\begin{array}{c}-x_3\nabla' v^0(t,x')\\v^0(t,x')\end{array}\Big)\quad\text{for every }x=(x',\ep x_3)\in\overline{\Omega}_{\ep},$$
where $\alpha\geq 3$, $u^0\in C^1([0,T];C^1(\overline{\omega};\R^2))$, $v^0\in C^1([0,T];C^2(\overline{\omega}))$ and $\nabla'$ denotes the gradient with respect to $x'$.

  As usual in dimension reduction, we perform a change of variable to state the problem on a fixed domain independent of $\ep$. We consider the set $\Omega:=\omega\times \big(-\tfrac 12,\tfrac 12\big)$ and the map $\psi^{\ep}:\overline{\Omega}\to\overline{\Omega}_{\ep}$, given by
 $$\psi^{\ep}(x):=(x',\ep x_3)\quad\text{for every }x=(x',x_3)\in\overline{\Omega}.$$ 
 To deal with the nonlinear structure of the energy, we follow the approach of \cite{FM}: we assume $\pep(t)$ to be a $C^1$ diffeomorphism on $\R^3$ and we write deformations $\eta\in W^{1,2}(\Omega_{\ep};\R^3)$ as
 $$\eta \circ \psi^{\ep}=\pep(t) \circ z,$$
 where $z\in W^{1,2}(\Omega;\R^3)$ satisfies 
 $$z(x)=\psi^{\ep}(x)=(x',\ep x_3)\quad\cal{H}^2\text{ - a.e. on }\gamma_d\times \big(-\tfrac 12,\tfrac 12\big).$$
 To any plastic strain $F_{pl}\in L^2(\Omega_{\ep};SL(3))$ we associate a scaled plastic strain $P\in L^2(\Omega;SL(3))$ defined as
 $$P:=F_{pl}\circ \psi^{\ep}$$
 and we rewrite the stored energy as 
 $$\cal{F}_{\ep}(t,z,P):=\intom{W_{el}(\nabla \pep(t,z(x))\nep z(x))}+\intom{W_{hard}(P(x))}=\frac{1}{\ep}\cal{E}(\eta,F_{pl}),$$
 where $\nep z:=(\nabla' z|\tfrac{1}{\ep}\partial_3 z)$.
 
 In this setting, according to the variational theory for rate-independent processes developed in \cite{MM}, a quasistatic evolution for the boundary datum $\pep$ is a function $t\mapsto(z(t),P(t))\in W^{1,2}(\Omega;\R^3)\times L^2(\Omega;SL(3))$ such that for every $t\in [0,T]$ the following two conditions are satisfied: 
 \begin{enumerate}
	\item[(qs1)] \emph{global stability:} there holds
	$$z(t)=\psi^{\ep}\quad\cal{H}^2\text{ - a.e. on }\gamma_d\times \big(-\tfrac 12,\tfrac 12\big)$$ and $(z(t), P(t))$ minimizes
	\begin{eqnarray}
\nonumber \cal{F}_{\ep}(t,\tilde{z},\tilde{P})+{\ep^{\alpha-1}}\intom{D(P(t),\tilde{P})},
\end{eqnarray}
among all $(\tilde{z},\tilde{P})\in W^{1,2}(\Omega;\R^3)\times L^2(\Omega;SL(3))$ such that $\tilde{z}=\psi^{\ep}$ $\cal{H}^2$ - a.e. on $\gamma_d\times \big(-\tfrac 12,\tfrac 12\big)$;
	\item[(qs2)] \emph{energy balance:}
	\begin{eqnarray}
\nonumber &&\cal{F}_{\ep}(t,z(t),P(t))+{\ep^{\alpha-1}}\cal{D}(P;0,t)\\
\nonumber &&=\cal{F}_{\ep}(0,z(0),P(0))+{\ep^{\alpha-1}}\int_0^t{\intom{E^{\ep}(s):\Big(\nabla \dot{\pep}(s,z(s))(\nabla \pep)^{-1}(s,z(s))\Big)}\,ds}.
\end{eqnarray}
\end{enumerate}
In the previous formula, $\cal{D}(P;0,t)$ is the plastic dissipation in the interval $[0,t]$ (see Section \ref{quas}), $E^{\ep}(t)$ is the stress tensor, defined as
$$E^{\ep}(t):=\frac{1}{\ep^{\alpha-1}}DW_{el}\big(\nabla \pep(t,z(t))\nep z(t)(P)^{-1}(t)\big)\big(\nabla \pep(t,z(t))\nep z(t)(P)^{-1}(t)\big)^T,$$
and $\alpha\geq 3$ is the same exponent as in the expression of the boundary datum.

Our main result is the characterization of the asymptotic behaviour of quasistatic evolutions as $\ep\to 0$. More precisely, in Theorem \ref{cvstress} and Corollaries \ref{sp} and \ref{su} we show that, given a sequence of initial data $(\zep_0,P^{\ep}_0)$ which is compact in a suitable sense, if $t\mapsto (\zep(t),P^{\ep}(t))$ is a quasistatic evolution for the boundary datum $\pep$ (according to (qs1)--(qs2)), satisfying $\zep(0)=\zep_0$ and $P^{\ep}(0)=P^{\ep}_0$, then defining the in-plane displacement
$$u^{\ep}(t):=\frac{1}{\ep^{\alpha-1}}\intt{\Big(\Big(\begin{array}{c}\pep_1(t,\zep(t))\\\pep_2(t,\zep(t))\end{array}\Big)-x'\Big)},$$
the out-of-plane displacement
$$v^{\ep}(t):=\frac{1}{\ep^{\alpha-2}}\intt{\pep_3(t,\zep(t))}$$
and the scaled linearized plastic strain
$$p^{\ep}(t):=\frac{P^{\ep}(t)-Id}{\ep^{\alpha-1}},$$
for every $t\in [0,T]$ we have
\begin{equation}
\nonumber
 p^{\ep}(t)\to p(t)\quad\text{strongly in }L^2(\Omega;\mthree),
 \end{equation}
 where $p(t)\in L^2(\Omega;\mthree)$ with $\tr{p(t)}=0$ a.e. in $\Omega$. If $\alpha>3$ there hold
\begin{eqnarray}
&&\label{1ci}u^{\ep}(t)\to u(t)\quad\text{strongly in }W^{1,2}(\omega;\R^2),\\
&&\label{2ci}v^{\ep}(t)\to v(t)\quad\text{strongly in }W^{1,2}(\omega),
\end{eqnarray}
for every $t\in [0,T]$, where $u(t)\in W^{1,2}(\omega;\R^2)$ and $v(t)\in W^{2,2}(\omega)$. If $\alpha=3$, the convergence of the in-plane and the out-of-plane displacements holds only on a $t$-dependent subsequence. Moreover, $t\mapsto(u(t),v(t),p(t))$ is a solution of the following reduced quasistatic evolution problem: for every $t\in [0,T]$ 
\begin{enumerate}
\item[(qs1)$_{r\alpha}$]\emph{reduced global stability:} 
$$u(t)=u^0(t),\quad v(t)=v^0(t),\quad \nabla' v(t)=\nabla' v^0(t)\quad\cal{H}^1\text { - a.e. on }\gamma_d$$
 and $(u(t), v(t), p(t))$ minimizes
 $$\intom{Q_2\big(\sym \nabla' \tilde{u}-x_3(\nabla')^2 \tilde{v}+\tfrac {L_{\alpha}}{2} \nabla' \tilde{v}\otimes \nabla' \tilde{v}-\tilde{p}'\big)}+\intom{\B{\tilde{p}}}+\intom{H_D(\tilde{p}-p(t))}$$
among all triples $(\tilde{u},\tilde{v},\tilde{p})\in W^{1,2}(\omega;\R^2)\times W^{2,2}(\omega)\times L^2(\Omega;\mthree)$, such that $\tr{\tilde{p}}=0$ a.e. in $\Omega$, and
$$\tilde{u}=u^0(t),\quad \tilde{v}=v^0(t)\quad\text{and}\quad \nabla' \tilde{v}=\nabla' v^0(t)\quad\cal{H}^1\text{ - a.e. on }\gamma_d;$$
\item[(qs2)$_{r\alpha}$]\emph{reduced energy balance:}
\begin{eqnarray*}
&&\intom{Q_2(e_{\alpha}(t))}+\intom{\B{p(t)}}+\cal{D}_{H_D}(p;0,t)=\intom{Q_2(e_{\alpha}(0))}+\intom{\B{p(0)}}\\
&&+\int_0^t{\intom{\C_2e_{\alpha}(s):\Big(\begin{array}{cc}\nabla' \dot{u}^0(s)+L_{\alpha}\nabla' \dot{v}^0(s)\otimes\nabla' v(s)-x_3(\nabla')^2 \dot{v}^0(s)&0\\0&0\end{array}\Big)}\,ds}
\end{eqnarray*}
\end{enumerate}
where
$$e_{\alpha}(t):=\sym \nabla' u(t)-x_3(\nabla')^2 v(t)+\tfrac{L_{\alpha}} {2} \nabla' v(t)\otimes \nabla' v(t)-p'(t)\quad\text{for every }t\in [0,T]$$
and $$L_{\alpha}:=\begin{cases}0&\text{if }\alpha>3,\\
1&\text{if }\alpha=3.\end{cases}$$
	In the above formulas, $\tilde{p}'$ and $p'(t)$ are the $2\times 2$ minors of $\tilde{p}$ and $p(t)$ given by the first two rows and columns, $\nabla'$ denotes the gradient with respect to $x'$, $Q_2$ and $B$ are two symmetric, positive definite quadratic forms on $\M^{2\times 2}$ and $\mthree$, respectively, for which an explicit formula is provided (see Sections \ref{prel} and \ref{comp}), $\C_2$ is the tensor associated to $Q_2$ and $\cal{D}_{H_D}$ is the plastic dissipation in the interval $[0,t]$ for the reduced model (see \eqref{NUMpag13}). \\
	
	We remark that Theorem \ref{cvstress} is only a convergence result. In fact, the issue of the existence of a quasistatic evolution in finite plasticity according to (qs1)--(qs2), is quite delicate, and it has only recently been solved in \cite{MM09} by adding to the stored-energy functional some further regularizing terms in the plastic component. We shall not add these further terms here, we rather show, in the last section, that our convergence result can be extended to sequences of approximate discrete-time quasistatic evolutions, whose existence is always guaranteed (see Theorem \ref{cvapp}). The limit quasistatic evolution problem identified in (qs1)$_{r\alpha}$--(qs2)$_{r\alpha}$, on the other hand, has always a solution (see Remark \ref{csd}).\\

The constant $L_{\alpha}$ in the limit problem encodes the main differences between the cases $\alpha>3$ and $\alpha=3$. Indeed, for $\alpha=3$, the limit energy contains the nonlinear term $\tfrac12 \nabla' v\otimes \nabla' v$, which is related to the stretching due to the out-of-plane displacement. For $\alpha>3$ the limit problem is completely linearized and, in the absence of hardening, coincides with that identified in \cite{DM} starting from three-dimensional linearized plasticity. However, we point out that the role of the hardening term in the present formulation is fundamental to deduce compactness of the three-dimensional evolutions (see Step 1, Proof of Theorem \ref{cvstress}). 
	
The limit stored energy and the limit plastic dissipation potential have both been deduced in the static case by $\Gamma$-convergence arguments. Indeed, in the absence of plastic deformations $(p=0)$ the stored energy reduces to the Von K\'arm\'an functional for $\alpha=3$ and to the linear plate functional for $\alpha>3$, which have been rigorously justified via $\Gamma$-convergence in \cite{FJM2} as low energy limits of three-dimensional nonlinear elasticity. In the case where plastic deformation is allowed, the energy in (qs1)$_{r\alpha}$ has been obtained in \cite{D1} as $\Gamma$-limit of a suitable scaling of the three-dimensional energy in (qs1). Our particular choice of the boundary datum and the scaling of the displacements are motivated by these results.\\

The setting of the problem and some arguments in the proofs are close to those of \cite{MS}. In particular, the proof of Theorem \ref{cvstress}  follows along the general lines of \cite{MRS}, where an abstract criterion for convergence of quasistatic evolutions is provided. 

A major difficulty in the proof of the reduced energy balance is related to the compactness of the stress tensors $E^{\ep}(t)$. In fact, due to the physical growth assumptions on $W_{el}$, weak $L^2$ compactness of $E^{\ep}(t)$ is in general not guaranteed. However, the sequence of stress tensors satisfies the following properties: there exists a sequence of sets $O_{\ep}(t)$, which converges in measure to $\Omega$, such that on $O_{\ep}(t)$ the stresses $E^{\ep}(t)$ are weakly compact in $L^2$, while in the complement of $O_{\ep}(t)$ their contribution is negligible in the $L^1$ norm. This mixed-type convergence is enough to pass to the limit in the three-dimensional energy balance. This argument of proof is similar to the one used in \cite{M-S} by Mora and Scardia, to prove convergence of critical points for thin plates under physical growth conditions for the energy density.

A further difficulty arises because of the physical growth conditions on $W_{el}$: the global stability (qs1) does not secure that $\zep(t)$ fulfills the usual Euler-Lagrange equations. This is crucial to identify the limit stress tensor. This issue is overcome by proving that $\zep(t)$ satisfies the analogue of an alternative first order condition introduced by Ball in \cite[Theorem 2.4]{B} in the context of nonlinear elasticity, and by adapting some techniques in \cite{M-S}. 

Finally, to obtain the reduced global stability condition, we need an approximation result for triples $(u,v,p)\in W^{1,2}(\omega;\R^2)\times W^{2,2}(\omega)\times L^2(\Omega;\mthree)$ such that 
\begin{equation}
\label{diri}u=0,\quad v=0,\quad\nabla' v=0\quad\cal{H}^1\text{ - a.e. on }\gamma_d
\end{equation}
in terms of smooth triples. {Arguing as in \cite[Section 3]{DM}, such a density result is proved under additional regularity assumptions on $\partial \omega$ and on $\gamma_d$ (see Lemma \ref{bdc}).\\}
		
	The paper is organized as follows: in Section \ref{prel} we set the static problem and we describe the limit functional. In Section \ref{comp} {we recall the compactness results proved in \cite{D1} and we prove an approximation result for triples $(u,v,p)$ satisfying \eqref{diri}.} Section \ref{quas} concerns the formulation of the quasistatic evolution problems, the statement of the main results of the paper and the construction of the mutual recovery sequence, whereas Section \ref{pquas} is entirely devoted to the proofs of the convergence of quasistatic evolutions. Finally, in Section \ref{appr} we discuss convergence of approximate discrete-time quasistatic evolutions. 

\smallskip

\noindent{\bf{Notation}}
We shall write any point $x\in\R^3$ as a pair $(x',x_3)$, where $x'\in \R^2$ and $x_3\in\R$.
 We shall use the following notation: given $\varphi:\Omega\to \R^3$, we denote by $\varphi':\Omega\to \R^2$ the map 
$$\varphi':=\Big(\begin{array}{c}\varphi_1\\\varphi_2\end{array}\Big)$$
and for every $\eta\in W^{1,2}(\Omega)$ we denote by $\nabla'\eta$ the vector $\Big(\begin{array}{c}\partial_1 \eta\\\partial_2\eta\end{array}\Big)$.
Analogously, given a matrix $M\in \mthree$, we use the notation $M'$ to represent the minor
$$M':=\Big(\begin{array}{cc}M_{11}&M_{12}\\M_{21}&M_{22}\end{array}\Big).$$

\section{Preliminaries and setting of the problem}
\label{prel}
Let $\omega\subset \R^2$ be a connected, bounded open set with { $C^2$} boundary. Let $\ep>0$. We assume that the set $\Omega_{\ep}:=\omega\times\big(-\tfrac \ep2,\tfrac \ep2\big)$ is the reference configuration of a finite-strain elastoplastic plate, and every deformation $\eta\in W^{1,2}(\Oe;\R^3)$ fulfills the multiplicative decomposition
$$\nabla \eta(x)=F_{el}(x)F_{pl}(x)\quad\text{for a.e. }x\in\Oe,$$
where $F_{el}\in L^2(\Oe;\mthree)$ represents the elastic strain, $F_{pl}\in L^2(\Oe;SL(3))$ is the plastic strain and $SL(3):=\{F\in\mthree: \det F=1\}.$ The stored energy (per unit thickness) associated to a deformation $\eta$ and to its elastic and plastic strains can be expressed as follows:
\begin{eqnarray}
\nonumber \cal{E}(\eta,F_{pl})&:=&\intome{W_{el}(\nabla \eta(x) F_{pl}^{-1}(x))}+\intome{W_{hard}(F_{pl}(x))},\\
\label{nsenergy}&=&\intome{W_{el}(F_{el}(x))}+\intome{W_{hard}(F_{pl}(x))}
\end{eqnarray}
where $W_{el}$ is the elastic energy density and $W_{hard}$ describes hardening.\\
{\bf Properties of the elastic energy}\\
We assume that $W_{el}:\mthree\to [0,+\infty]$ satisfies
\begin{itemize}
\item[(H1)] $W_{el}\in C^1(\mthree_+),\quad W_{el}\equiv +\infty \text{ on }\mthree\setminus \mthree_+$,
\item [(H2)] $W_{el}(Id)=0,$
\item  [(H3)] $W_{el}(RF)=W_{el}(F)\quad\text{for every }R\in SO(3),\, F\in \mthree_+,$
\item [(H4)] $W_{el}(F)\geq c_1 \dist^2(F;SO(3))\quad\text{for every }F\in \mthree_+,$
\item [(H5)] $|F^T DW_{el}(F)|\leq c_2 (W_{el}(F)+1) \quad\text{for every }F\in \mthree_+.$
\end{itemize}
Here $c_1,c_2$ are positive constants, $\mthree_+:=\{F\in\mthree:\, \det F>0\}$ and $SO(3):=\{F\in\mthree_+:\, F^T F=Id\}$.
We also assume that there exists a symmetric, positive semi-definite tensor $\C:\mthree\to \ms$ such that, setting
$$Q(F):=\frac{1}{2}\C F:F\quad\text{for every }F\in\mthree,$$
the quadratic form $Q$ encodes the local behaviour of $W_{el}$ around the identity, namely
\begin{equation}
\label{quadrwel}
 \forall \delta>0\,\, \exists c_{el}(\delta)>0 \text{ such that }\forall F\in B_{c_{el}(\delta)}(0)\text{ there holds }|W_{el}(Id+F)-Q(F)|\leq \delta |F|^2.
\end{equation}
\begin{oss}
By \cite[Proposition 1.5]{DL} and by (H3) and (H5), there holds
\begin{equation}
\label{mandel2}
|DW_{el}(F)F^T|\leq c_3(W_{el}(F)+1)\quad\text{for every }F\in\mthree_+,
\end{equation}
where $c_3$ is a positive constant.
Moreover, by (H1) and (H5), there exist  $ c_4,\, c_5,\,\gamma > 0$ such that, for every $G_1,\, G_2\in B_{\gamma}(Id)$ and for every $F \in \mthree_+$ the following estimate holds true
 \begin{equation}
 \label{lemmams}
  |W_{el}(G_1FG_2) - W_{el}(F)| \leq c_4(W_{el}(F) + c_5)(|G_1-Id| + |G_2-Id|)
  \end{equation}
 (see \cite[Lemma 4.1]{MS}).\\
\end{oss}
\begin{oss}
As remarked in \cite[Section 2]{MS}, the frame-indifference condition (H3) yields
$$\C_{ijkl}=\C_{jikl}=\C_{ijlk}\text{ for every }i,j,k,l\in\{1,2,3\}$$
and
$$\C F=\C\, (\sym\, F) \quad\text{for every }F\in\mthree.$$
Hence, the quadratic form $Q$ satisfies:
$$Q(F)=Q(\sym\,F)\quad\text{for every }F\in\mthree$$
and by (H4) it is positive definite on symmetric matrices. This, in turn, implies that there exist two constants $\ac$ and $\bc$ such that 
\begin{equation}
\label{growthcondQ}
\ac |F|^2\leq Q(F)\leq \bc |F|^2\quad\text{for every }F\in\mthree_{\sym},
\end{equation}
and
\begin{equation}
\label{growthcondC}
|\C F|\leq 2\bc |F| \quad\text{for every }F\in\mthree_{\sym}.
\end{equation}
\end{oss}
\begin{oss}
We note that \eqref{quadrwel} entails, in particular,
$$W_{el}(Id)=0,\quad DW_{el}(Id)=0$$
and
$$\C =D^2W_{el}(Id),\quad \C_{ijkl}=\frac{\partial^2 W}{\partial F_{ij}\partial F_{kl}}(Id)\text{ for every }i,j,k,l\in\{1,2,3\}.$$
By combining \eqref{quadrwel} with \eqref{growthcondC} we deduce also that there exists a constant $c_{el_2}$ such that
\begin{equation}
\label{locquad} 
|DW_{el}(Id+F)|\leq (2\bc+1)|F|
\end{equation}
for every $F\in\mthree$, $|F|<c_{el_2}$.
\end{oss}
\noindent{\bf Properties of the hardening functional}\\
We assume that the hardening map $W_{hard}:\mthree\to [0,+\infty]$ is of the form
\begin{equation}
\nonumber
W_{hard}(F):=\begin{cases}\twh(F)&\text{for every }F\in K,\\
+\infty&\text{otherwise}.
\end{cases}
\end{equation}
Here $K$ is a compact set in $SL(3)$ that contains the identity as a relative interior point, and the map $\twh:\mthree\to [0,+\infty)$ fulfills
\begin{eqnarray}
\nonumber &&\twh\text{ is locally Lipschitz continuous},\\
\label{prh3} && \twh(Id+F)\geq c_6 |F|^2\quad\text{for every }F\in\mthree,
\end{eqnarray}
where $c_6$ is a positive constant. 
We also assume that there exists a symmetric, positive definite tensor $\mathbb{B}:\mthree\to\mthree$ such that, setting
$$B(F):=\frac{1}{2}\mathbb{B}F:F\quad\text{ for every }F\in\mthree,$$
the quadratic form $B$ satisfies
\begin{eqnarray}
\nonumber && \forall \delta>0\, \exists c_h(\delta)>0 \text{ such that }\forall F\in B_{c_h(\delta)}(0)\text{ there holds }|\twh(Id+F)-\B{F}|\leq \delta \B{F}.\\
 \label{prh4}
\end{eqnarray} 
In particular, by the hypotheses on $K$ there exists a constant $c_k$ such that
\begin{eqnarray}
\label{prk1}&& |F|+|F^{-1}|\leq c_k\quad\text{for every }F\in K,\\
\label{prk2}&& |F-Id|\geq \frac{1}{c_k}\quad\text{for every }F\in SL(3)\setminus K.
\end{eqnarray}
Combining \eqref{prh3} and \eqref{prh4} we deduce also
\begin{equation}
\label{grbelowh}
\frac{c_6}{2} |F|^2\leq \B{F}\quad\text{for every }F\in\mthree.
\end{equation}
{\bf Dissipation functional}\\
Denote by $\md$ the set of symmetric trace-free matrices, namely
$$\md:=\{F\in\mthree_{\sym}: \tr F=0\}.$$
Let $H_{D}:\md \to [0,+\infty)$ be a convex, positively one-homogeneous function such that
\begin{equation}
\label{growthh}
\rk |F|\leq H_{D}(F)\leq \Rk |F|\quad\text{for every }F\in\md.
\end{equation}
We define the dissipation potential $H:\mthree\to [0,+\infty]$ as
\begin{equation}
\nonumber
H(F):=\begin{cases}H_{D}(F)&\text{if }F\in \md,\\
+\infty &\text{otherwise.}\end{cases}
\end{equation}
For every $F\in\mthree$ consider the quantity
\begin{equation}
\label{distid}
D(Id,F):=\inf \Big\{\int_0^1{H(\dot{c}(t)c^{-1}(t))\,dt}: c\in C^1([0,1];\mthree_+),\, c(0)=Id,\, c(1)=F \Big\}.
\end{equation}
Note that, by the Jacobi's formula for the derivative of the determinant of a differentiable matrix-valued map, if $D(Id, F)<+\infty$, then $F\in SL(3)$.

We define the dissipation distance as the map $D:\mthree\times \mthree\to [0,+\infty]$, given by 
$$D(F_1,F_2):=\begin{cases}D(Id, {F_2}F_1^{-1})& \text{if }F_1\in\mthree_{+}, F_2\in\mthree\\ +\infty& \text{if }F_1\notin \mthree_{+}, F_2\in\mthree. 
\end{cases}$$
We note that the map $D$ satisfies the triangle inequality
\begin{equation}
\label{triang}
D(F_1,F_2)\leq D(F_1,F_3)+D(F_3,F_2)
\end{equation}
for every $F_1,F_2, F_3\in\mthree$. 

Given a preexistent plastic strain $F_{pl}^0\in L^2(\Omega_{\ep};SL(3))$, we define the plastic dissipation potential associated to a plastic configuration $F\in L^2(\Omega_{\ep};SL(3))$ as
\begin{equation}
\label{dissord}
\ep^{\alpha-1}\intome{D(F_{pl}^0;F)},
\end{equation}
where $\alpha\geq 3$ is a given parameter.
\begin{oss}
 We remark that there exists a positive constant $c_7$ such that
\begin{eqnarray}
\label{prd1} &&D(F_1,F_2)\leq c_7\quad\text{for every }F_1,F_2\in K,\\
\label{prd2} && D(Id,F)\leq c_7|F-Id|\quad\text{for every }F\in K.
\end{eqnarray}
Indeed, by the compactness of $K$ and the continuity of the map $D$ on $SL(3)\times SL(3)$ (see \cite{M}), there exists a constant $\tilde{c}_7$ such that 
\begin{equation}
\label{quasiprd1}
D(F_1,F_2)\leq \tilde{c}_7\quad\text{for every }F_1,F_2\in K.
\end{equation} 
By the previous estimate, \eqref{prd2} needs only to be proved in a neighbourhood of the identity. More precisely, let $\delta>0$ be such that $\log F$ is well defined for $F\in K$ and $|F-Id|<\delta$. If $F\in K$ is such that $|F-Id|\geq \delta$, by \eqref{quasiprd1} we have
$$D(Id,F)\leq \frac{\tilde{c}_7}{\delta}|F-Id|.$$
If $|F-Id|<\delta$, taking $c(t)=\exp({t\log F})$ in \eqref{distid}, inequality \eqref{growthh} yields
$$D(Id,F)\leq H_{D}(\log F)\leq \Rk |\log F|\leq C|F-Id|$$
for every $F\in K$. Collecting the previous estimates we deduce \eqref{prd1} and \eqref{prd2}.
\end{oss}
\subsection{Change of variable and formulation of the problem}
{
 We suppose that the boundary $\partial\omega$ is partitioned into two disjoint open subsets $\gamma_d$ and $\gamma_n$, and their common boundary $\partial\lfloor_{\partial\omega}\gamma_d = \partial\lfloor_{\partial\omega}\gamma_n$(topological notions refer here to the relative topology of $\partial\omega$). We assume that $\gamma_d$ is nonempty and that $\partial\lfloor_{\partial\omega}\gamma_d= \{P_1, P_2\}$, where $P_1, P_2$ are two points in $\partial\omega$. We denote by $\Gamma_{\ep}$ the portion of the lateral surface of the plate given by $\Gamma_{\ep}:=\gamma_d\times\big(-\tfrac \ep2,\tfrac \ep2\big)$. On $\Gamma_{\ep}$ we prescribe a boundary datum of the form
\begin{equation}
\label{defbddat}
\phi^{\ep}(x):=\Big(\begin{array}{c}x'\\x_3\end{array}\Big)+\Big(\begin{array}{c}\ep^{\alpha-1}u^0(x')\\0\end{array}\Big)+\ep^{\alpha-2}\Big(\begin{array}{c}-x_3\nabla' v^0(x')\\v^0(x')\end{array}\Big)\end{equation}
for every $x=(x',\ep x_3)\in\Omega_{\ep}$, where $u^0\in C^{1}(\overline{\omega};\R^2)$, $v^0\in C^{2}(\overline{\omega})$ and $\alpha\geq 3$ is the same parameter as in \eqref{dissord}.

We consider deformations $\eta\in W^{1,2}(\Omega_{\ep};\R^3)$ satisfying \begin{equation}
\label{bddatunsc}
\eta=\pep\quad\cal{H}^2\text{ - a.e. on }\Gamma_{\ep}.
\end{equation}}
As usual in dimension reduction, we perform a change of variable to formulate the problem on a domain independent of $\ep$. We consider the set $\Omega:=\omega \times \big(-\tfrac 12, \tfrac 12\big)$ and the map $\psi^{\ep}:\overline{\Omega}\to \overline{\Omega}_{\ep}$ given by 
\begin{equation}
\label{cov}
\psi^{\ep}(x):=(x',\ep x_3)\quad\text{for every }x\in\overline{\Omega}.\end{equation}
To every deformation $\eta \in W^{1,2}(\Oe;\R^3)$ satisfying 
\eqref{bddatunsc}
 and to every plastic strain $F_{pl}\in L^2(\Oe;SL(3))$, we associate the scaled deformation $y:=\eta\circ \psi^{\ep}$ and the scaled plastic strain $P:=F_{pl}\circ \psi^{\ep}$. Denoting by $\Gamma_d$ the set $\gamma_d\times \big(-\tfrac 12, \tfrac 12\big),$ the scaled deformation satisfies the boundary condition 
{\begin{equation}
\label{bddatum}
y=\phi^{\ep}\circ \psi^{\ep}\quad\cal{H}^2\text{ - a.e. on }\Gamma_d. 
\end{equation}}

Denote by $\cal{A}_{\ep}(\phi^{\ep})$ the class of pairs $(y^{\ep},P^{\ep})\in W^{1,2}(\Omega;\R^3)\times L^2(\Omega;SL(3))$ such that \eqref{bddatum} is satisfied.
Applying the change of variable \eqref{cov} to \eqref{nsenergy} and \eqref{dissord}, the energy functional is now given by
\begin{equation}
\label{encov}
\cal{I}(y,P):=\frac{1}{\ep}\cal{E}(\eta,F_{pl})=\intom{W_{el}(\nep y(x) P^{-1}(x))}+\intom{W_{hard}(P(x))},
\end{equation}
where $\nep y(x):=\big(\partial_1 y(x)\big|\partial_2 y(x)\big|\frac{1}{\ep} \partial_3 y(x)\big)$ for a.e. $x\in\Omega$. The plastic dissipation potential is given by 
\begin{equation}
\label{disscov}
\ep^{\alpha-1}\intom{D(P^{\ep,0},P)}
\end{equation}
where $P^{\ep,0}:=F_{pl}^0\circ \psi^{\ep}$ is a preexistent plastic strain. We remark here that the asymptotic behaviour of sequences of pairs $(y^{\ep}, P^{\ep})\in\cal{A}_{\ep}(\pep)$ such that 
$$\cal{I}(y^{\ep},P^{\ep})+\ep^{\alpha-1}\intom{D(P^{\ep,0},P^{\ep})}$$
is of order $\ep^{2\alpha-2}$ has been studied in \cite{D1} under suitable assumptions on the maps $P^{\ep,0}$.\\

\section{Compactness results}
\label{comp}
In this section we collect two compactness results that were {obtained} in \cite{D1} { and we state an approximation result which will be crucial in the proof of the reduced global stability condition}. In the first theorem, the rigidity estimate proved by Friesecke, James and M\"uller in \cite[Theorem 3.1]{FJM} allow us to 
approximate sequences of deformations whose distance of the gradient from $SO(3)$ is uniformly bounded, by means of rotations (see \cite[Theorem 3.3]{D1}). 
\begin{teo}
\label{compactbd1}
Assume that $\alpha\geq 3$. Let $(y^{\ep})$ be a sequence of deformations in $W^{1,2}(\Omega;\R^3)$ satisfying \eqref{bddatum} and such that
\begin{equation}
\label{elasticbd}
\|\dist(\nep y^{\ep}, SO(3))\|_{L^2(\Omega; \mthree)}\leq C\ep^{\alpha-1}.
\end{equation}
Then, there exists a sequence $(R^{\ep})\subset W^{1,\infty}(\omega; \mthree)$ such that for every $\ep>0$
\begin{eqnarray}
&&\label{rt1} R^{\ep}(x')\in SO(3)\quad\text{for every }x'\in \omega,\\
&&\label{rt2} \|\nep y^{\ep}-R^{\ep}\|_{L^2(\Omega;\mthree)}\leq C\ep^{\alpha-1},\\
&&\label{rt3} \|\partial_i R^{\ep}\|_{L^2(\omega;\mthree)}\leq C\ep^{\alpha-2},\,i=1,2\\
&&\label{rt4} \|R^{\ep}-Id\|_{L^2(\omega;\mthree)}\leq C\ep^{\alpha-2}.
\end{eqnarray} 
\end{teo}

Let $\A:\mathbb{M}^{2\times 2}\to \mthree_{\sym}$ be the operator given by
$$\A F:=\Bigg(\begin{array}{cc}\sym\,F&\hspace{-0.5 cm}\begin{array}{c}\lambda_1(F)\vspace{-0.1 cm}\\\lambda_2(F)\end{array}\vspace{-0.1 cm}\\\begin{array}{cc}\lambda_1(F)&\lambda_2(F)\end{array}&\hspace{-0.5 cm}\lambda_3(F)\end{array}\Bigg)\quad\text{for every }F\in\M^{2\times 2},$$
 where for every $F\in\M^{2\times 2}$ the triple $(\lambda_1(F),\lambda_2(F),\lambda_3(F))$ is the unique solution to the minimum problem
$$\min_{\lambda_i\in\R}Q\Bigg(\begin{array}{cc}\sym\,F&\hspace{-0.5 cm}\begin{array}{c}\lambda_1\vspace{-0.1 cm}\\\lambda_2\end{array}\vspace{-0.1 cm}\\\begin{array}{cc}\lambda_1&\lambda_2\end{array}&\hspace{-0.5 cm}\lambda_3\end{array}\Bigg).$$
We remark that for every $F\in\M^{2\times 2}$, $\A(F)$ is given by the unique solution to the linear equation
\begin{equation}
 \label{linearmin}
\C \A (F):\Bigg(\begin{array}{ccc}0&0&\lambda_1\\0&0&\lambda_2\\\lambda_1&\lambda_2&\lambda_3\end{array}\Bigg)=0\quad\text{for every }\lambda_1,\lambda_2,\lambda_3\in\R.
\end{equation}
This implies, in particular, that $\A$ is linear.

We define the quadratic form $Q_2:\mathbb{M}^{2\times 2}\to [0,+\infty)$ as
$$Q_2(F)=Q(\A (F))\quad\text{for every }F\in \mathbb{M}^{2\times 2}.$$
By properties of $Q$, we have that $Q_2$ is positive definite on symmetric matrices. We also define the tensor $\C_2:\mathbb{M}^{2\times 2}\to \mthree_{\sym}$, given by

\begin{equation}
\label{defc2}
\C_2 F:=\C\A (F)\quad\text{for every }F\in \mathbb{M}^{2\times 2}.
\end{equation}
We remark that by \eqref{linearmin} there holds
\begin{equation}
\label{nothird}
\C_2 F:G=\C_2 F:\Big(\begin{array}{cc}\sym\,G&0\\0&0\end{array}\Big)\quad\text{for every }F\in \mathbb{M}^{2\times 2},\,G\in\mthree
\end{equation}
and
$$Q_2(F)=\frac{1}{2}\C_2 F:\Big(\begin{array}{cc}\sym\,F&0\\0&0\end{array}\Big)\quad\text{for every }F\in \mathbb{M}^{2\times 2}.$$

Given a sequence of deformations $(y^{\ep})\subset W^{1,2}(\Omega;\R^3)$, we consider some associated quantities: the in-plane displacements
\begin{equation}
\label{inplane}
u^{\ep}(x'):=\frac{1}{\ep^{\alpha-1}}\intt{\big(\m{{y}^{\ep}}(x',x_3)-x'\big)}\quad\text{for a.e. }x'\in \omega,
\end{equation}
and the out-of-plane displacements
\begin{equation}
\label{outofplane}
v^{\ep}(x'):=\frac{1}{\ep^{\alpha-2}}\intt{y^{\ep}_3(x',x_3)}\quad\text{for a.e. }x'\in \omega.
\end{equation}
For every sequence $(y^{\ep})$ in $W^{1,2}(\Omega;\R^3)$ satisfying both \eqref{bddatum} and \eqref{elasticbd}, we introduce also the strains
\begin{equation}
\label{defgep}
G^{\ep}(x):=\frac{(R^{\ep}(x))^T\nep y^{\ep}(x)-Id}{\ep^{\alpha-1}}\quad\text{for a.e. }x\in\Omega,
\end{equation}
where the maps $R^{\ep}$ are the pointwise rotations provided by Theorem \ref{compactbd1}. 

Denote by $\cal{A}(u^0,v^0)$ the set of triples $(u,v,p)\in W^{1,2}(\Omega;\R^2)\times W^{2,2}(\Omega)\times L^2(\Omega;\md)$ such that 
$u=u^0,\,v=v^0,\text{ and }\nabla'v=\nabla'v^0$ $\cal{H}^1\text{ - a.e. on }\gamma_d$.
 The next theorem allows us to deduce some compactness properties for the displacements and strains and a liminf inequality for the scaled stored energy and plastic dissipation potential, introduced in \eqref{encov} and \eqref{disscov} (see \cite[Theorem 3.4]{D1}).
\begin{teo}
\label{liminfineq}
Assume that $\alpha\geq 3$. Let $(y^{\ep},P^{\ep})$ be a sequence of pairs in $\cal{A}_{\ep}(\phi^{\ep})$ satisfying 
\begin{equation}
\label{engest2}
\cal{I}(y^{\ep},P^{\ep})\leq C\ep^{2\alpha-2}
\end{equation}
for every $\ep>0$.
 Let $u^{\ep}$, $v^{\ep}$ and $G^{\ep}$ be defined as in \eqref{inplane}, \eqref{outofplane}, and \eqref{defgep}, respectively. Then, there exist $(u,v,p)\in \cal{A}(u^0,v^0)$ such that, up to subsequences, there hold
\begin{eqnarray}
&&\label{cptyep}y^{\ep}\to \Big(\begin{array}{c}x'\\0\end{array}\Big)\quad\text{strongly in }W^{1,2}(\Omega;\R^3),\\
&&\label{cptuep} u^{\ep}\deb u\quad\text{weakly in }W^{1,2}(\omega;\R^2),\\
&&\label{cptvep} v^{\ep}\to v\quad\text{strongly in }W^{1,2}(\omega),\\
&&\label{cptnep3}\frac{\nabla' y^{\ep}_3}{\ep^{\alpha-2}}\to \nabla' v\quad\text{strongly in }L^2(\Omega;\R^2),
\end{eqnarray}
and the following estimate holds true
\begin{equation}
\label{3comphest}
\big\|\frac{y^{\ep}_3}{\ep}-x_3-\ep^{\alpha-3}v^{\ep}\big\|_{L^2(\Omega)}\leq C\ep^{\alpha-2}.
\end{equation}
Moreover, there exists $G\in L^2(\Omega;\mthree)$ such that
\begin{equation}
\label{cptGep}
G^{\ep}\deb G\quad\text{weakly in }L^2(\Omega;\mthree),
\end{equation}
and the $2\times 2$ submatrix $G'$ satisfies
\begin{equation}
\label{gaff}
G'(x', x_3) = G_0(x') - x_3 (\nabla')^2 v(x')\quad\text{for a.e. }x\in\Omega,
\end{equation}
where
\begin{eqnarray}
&&\label{Ga3} \sym\, G_0 = \frac{(\nabla' u+(\nabla' u)^T +\nabla' v\otimes \nabla' v)}{2}\quad\text{if } \alpha=3,\\
&&\label{Ga>3} \sym\, G_0 = \sym \nabla' u\quad\text{if }\alpha> 3.
\end{eqnarray}
The sequence of plastic strains $(P^{\ep})$ fulfills 
\begin{equation}
\label{Pep1} P^{\ep}(x)\in K\quad\text{for a.e. }x\in\Omega,
\end{equation}
and
\begin{equation}
\label{Pep2} \|P^{\ep}-Id\|_{L^2(\Omega;\mthree)}\leq C\ep^{\alpha-1}
\end{equation}
for every $\ep$. Moreover, setting 
\begin{equation}
\label{defpep}
p^{\ep}:=\frac{P^{\ep}-Id}{\ep^{\alpha-1}},
\end{equation} up to subsequences
\begin{equation}
\label{wconvpep}
p^{\ep}\deb p\quad\text{weakly in }L^2(\Omega;\mthree).
\end{equation}
Finally,
\begin{eqnarray}
\label{liminftot}\intom{Q_2(\sym\, G'-p')}+\intom{B(p)}\leq \liminf_{\ep\to 0} \frac{1}{\ep^{2\alpha-2}}\cal{I}(y^{\ep},P^{\ep}).
\end{eqnarray}
If in addition 
\begin{equation}
\label{engest3}
\frac{1}{\ep^{\alpha-1}}\intom{D({P}^{\ep,0},{P}^{\ep})}\leq C\quad\text{for every }\ep>0
\end{equation}
and there exist a map $p^0\in L^2(\Omega;\mthree_D)$ and a sequence $({p}^{\ep,0})\subset L^2(\Omega;\mthree)$ such that $P^{\ep,0}=Id+\ep^{\alpha-1}p^{\ep,0}$, with ${p}^{\ep,0}\deb p^0$ weakly in $L^2(\Omega;\mthree)$, then
\begin{equation}
\label{liminfdiss}
\intom{H_{D}({p}-p^0)}\leq \liminf_{\ep\to 0} \frac{1}{\ep^{\alpha-1}}\intom{D({P}^{\ep,0},{P}^{\ep})}.
\end{equation}
\end{teo}
{\begin{proof}
The proof follows easily by adapting \cite[Proof of Theorem 3.4]{D1}. 
\end{proof}
{We conclude this section by providing an approximation result for triples $(u,v,p)\in \cal{A}(0,0)$ by means of smooth triples.}
More precisely, denoting by $C^{\infty}_c(\omega\cup\gamma_n)$ the sets of smooth maps having compact support in $\omega\cup\gamma_n$, the following lemma holds true.
\begin{lem}
\label{bdc}
{{
Let $(u,v,p)\in \cal{A}(0,0)$. Then there exists a sequence of triples $(u^{k},v^{k},p^{k})\in C^{\infty}_c(\omega\cup\gamma_n;\R^2)\times C^{\infty}_c(\omega\cup\gamma_n)\times C^{\infty}_c(\Omega;\md)$ such that
\begin{eqnarray*}
&&u^{k}\to u\quad\text{strongly in }W^{1,2}(\omega;\R^2),\\
&&v^{k}\to v\quad\text{strongly in }W^{2,2}(\omega),\\
&&p^{k}\to p\quad\text{strongly in }L^2(\Omega;\md).
\end{eqnarray*}
}} \end{lem}
\begin{proof}
The approximation of the plastic strain $p$ is obtained by standard arguments. The approximation of the in-plane displacements and out-of-plane displacements follows by adapting the arguments in \cite[Theorem 3.9 and Lemma 6.10]{DM}.
\end{proof}

\section{The quasistatic evolution problem}
\label{quas}
In this section we set the quasistatic evolution problem. 

For every $t\in [0,T]$ we prescribe a boundary datum $\pep(t)\in W^{1,\infty}(\Omega;\R^3)\cap C^{\infty}(\R^3;\R^3)$, defined as
\begin{equation}
\nonumber
\pep(t,x):=\Big(\begin{array}{c}x'\\x_3\end{array}\Big)+\ep^{\alpha-1}\Big(\begin{array}{c}u^0(t,x')\\0\end{array}\Big)+\ep^{\alpha-2}\Big(\begin{array}{c}-x_3\nabla' v^0(t,x')\\v^0(t,x')\end{array}\Big),
\end{equation} 
for every $x\in\R^3$, where the map $t\mapsto u^0(t)$ is assumed to be $C^{1}([0,T];C^1(\R^2;\R^2))$ and the map $t\mapsto v^0(t)$ is $C^{1}([0,T];C^2(\R^2))$. We consider deformations $t\mapsto y^{\ep}(t)$ from $[0,T]$ into $W^{1,2}(\Omega;\R^3)$ that satisfy 
\begin{equation}
\nonumber
y^{\ep}(t,x)=\pep(t,(x',\ep x_3))\quad\cal{H}^2\text{ -a.e. on }\Gamma_d,
\end{equation}
and plastic strains $t\mapsto P^{\ep}(t)$ from $[0,T]$ into $L^2(\Omega;SL(3))$. 

For technical reasons, it is convenient to modify the map $t\mapsto\pep(t)$ outside the set $\Omega$. We consider a truncation function $\tep\in W^{1,\infty}(\R)\cap {C^{1}}(\R)$ satisfying 
\begin{eqnarray}
 \label{treq1}&&\tep(s)=s\quad\text{ in }(-\ell_{\ep}, \ell_{\ep}),\\
 \label{treq2}&&|\tep(s)|\leq |s|\text{ for every }s\in \R,\\
\label{treq3}&&\|\tep\|_{L^{\infty}(\mathbb{R})}\leq 2 \ell_{\ep},\\
\label{treq4}&&\dot{\tep}(s)=0\quad\text{ if }|x_3|\geq \ell_{\ep}+1,\\
\label{treq5}&&\|\dot{\tep}(s)\|_{L^{\infty}(\R)}\leq 2,
\end{eqnarray}
where $\ell_{\ep}$ is such that 
\begin{eqnarray}
\label{lp1}&&\ep^{\alpha-1-\gamma}\ell_{\ep}\to 0,\\
\label{lp2}&& \ep \ell_{\ep}\to +\infty,\\
\label{lp3}&&\ep^{2\alpha-2}\ell_{\ep}^3\to 0,
\end{eqnarray}
for some $0<\gamma<\alpha-2$.
For $\alpha>3$ we also require
\begin{equation}
\label{lp4}
\ep^{\alpha-1}\ell_{\ep}^2\to 0.
\end{equation}
\begin{oss}
A possible choice of $\ell_{\ep}$ is $\ell_{\ep}=\frac{1}{\ep^{1+\lambda}}$, with $0<\lambda<\min\{\frac{\alpha-3}{2},\alpha-2-\gamma\}$ when $\alpha>3$ and $0<\lambda<\min\{\frac{1}{3},1-\gamma\}$ in the case $\alpha=3$.
\end{oss}

With a slight abuse of notation, for every $t\in [0,T]$ we still denote by $\pep(t)$ the map defined as
\begin{equation}
\label{defphiep}
\pep(t,x):=\Big(\begin{array}{c}x'\\x_3\end{array}\Big)+\ep^{\alpha-1}\Big(\begin{array}{c}u^0(t,x')-\tep\big(\frac{x_3}{\ep}\big)\nabla' v^0(t,x')\\0\end{array}\Big)+\ep^{\alpha-2}\Big(\begin{array}{c}0\\v^0(t,x')\end{array}\Big)
\end{equation}
for every $x\in\R^3$. 
\begin{oss}
\label{propinverse}
Conditions \eqref{treq1} and \eqref{lp2} guarantee that $\pep(t)$ is indeed an extension of the originally prescribed boundary datum, for $\ep$ small enough. Conditions \eqref{treq3} and \eqref{treq5} provide a uniform bound with respect to $t$ on the $W^{1,\infty}(\R^3;\R^3)$ norm of $\pep(t)-id$. By \eqref{treq3}, \eqref{treq5} and \eqref{lp1}, there exists $\ep_0>0$ such that, for every $t\in [0,T]$ and $\ep<\ep_0$, the map $\pep(t):\R^3\to \R^3$ is invertible with smooth inverse $\vep(t)$. Since $$\pep(t, \vep(t,x))=x\quad\text{for every }x\in\R^3,$$ by \eqref{defphiep} there holds
\begin{eqnarray}
\label{fcomp} && \bvep(t)-x'=-\ep^{\alpha-1}u^0(t,\bvep(t))+\ep^{\alpha-1}\tep\Big(\frac{\vep_3(t)}{\ep}\Big)\nabla' v^0(t,\bvep(t)),\\
\label{3comp}&& \vep_3(t)-x_3=-\ep^{\alpha-2}v^0(t,\bvep(t)),
\end{eqnarray}
 for every $t\in [0,T]$. Hence, by the smoothness of $u^0$ and $v^0$ and by \eqref{treq3}, we deduce the estimates
\begin{equation}
\label{distinv1}
\|\bvep(t)-x'\|_{L^{\infty}(\R^3;\R^2)}\leq C\ep^{\alpha-1}\ell_{\ep},
\end{equation}
and
\begin{equation}
\label{distinv13}
\|\vep_3(t)-x_3\|_{L^{\infty}(\R^3)}\leq C\ep^{\alpha-2},
\end{equation}
where both constants are independent of $t$. In particular, \eqref{fcomp} yields
\begin{eqnarray}
\nonumber \nabla \bvep(t)-\Big(\begin{array}{ccc}1&0&0\\0&1&0\end{array}\Big)&&=-\ep^{\alpha-1}\nabla' u^0(t,\bvep(t))\nabla \bvep(t)\\
\nonumber &&+\ep^{\alpha-1}\tep\Big(\frac{\vep_3(t)}{\ep}\Big)(\nabla')^2 v^0(t,\bvep(t))\nabla \bvep(t)\\
\label{expgr}&&+\ep^{\alpha-2}\dot{\tep}\Big(\frac{\vep_3(t)}{\ep}\Big)\nabla' v^0(t,\bvep(t))\otimes\nabla \vep_3(t),
\end{eqnarray}
and \eqref{3comp} implies
\begin{equation}
\label{gradinv0}
\nabla \vep_3(t)-e_3=-\ep^{\alpha-2} (\nabla \bvep(t))^T\nabla' v^0(t,\bvep(t)),
\end{equation}
for every $t\in [0,T]$.

A direct computation shows that
\begin{eqnarray}
\nonumber&&\nabla \pep(t,x)=Id+\ep^{\alpha-1}\Big(\begin{array}{cc}\nabla' u^0(t,x')&0\\0&0\end{array}\Big)-\ep^{\alpha-1}\Big(\begin{array}{cc}\tep\big(\frac{x_3}{\ep}\big)(\nabla')^2 v^0(t,x')&0\\0&0\end{array}\Big)\\
\label{gradpep}&&+\ep^{\alpha-2}\Big(\begin{array}{cc}0&-\dot{\tep}\big(\frac{x_3}{\ep}\big)\nabla' v^0(t,x')\\(\nabla' v^0(t,x'))^T&0\end{array}\Big)\quad\text{for every }x\in\R^3.
\end{eqnarray}
Hence by \eqref{treq3}, \eqref{treq5} and \eqref{lp1} there holds
\begin{equation}
\label{bddinverse}
\|\nabla \vep (t)\|_{L^{\infty}(\R^3;\mthree)}\leq \|(\nabla \pep (t))^{-1}\|_{L^{\infty}(\R^3;\mthree)}\leq C,\end{equation}
for every $t\in[0,T]$ and for every $\ep<\ep_0$. Therefore, \eqref{treq3}, \eqref{treq5}, \eqref{lp2}, \eqref{expgr} and \eqref{gradinv0} yield 
\begin{equation}
\label{gradinv}
\|\nabla \bvep(t)-(e_1|e_2|0)\|_{L^{\infty}(\R^3;\M^{3\times 2})}\leq C\ep^{\alpha-1}\ell_{\ep},
\end{equation}
and
\begin{equation}
\label{gradinv1}
\|\nabla \vep_3(t)-e_3\|_{L^{\infty}(\R^3;\R^3)}\leq C\ep^{\alpha-2}.
\end{equation}
\end{oss}
By Remark \ref{propinverse} for $\ep$ small enough the function $\pep(t)$ is a smooth diffeomorphism for every $t\in [0,T]$. This implies that we are allowed to define a map $t\mapsto \zep(t)$ from $[0,T]$ into $W^{1,2}(\Omega;\R^3)$ as the pointwise solution of 
\begin{equation}
\nonumber
y^{\ep}(t,x)=\pep(t,\zep(t,x))
\end{equation}
for every $t\in [0,T]$. {We note that 
\begin{equation}
\label{rsze}
\zep(t)=(x',\ep x_3)\quad\cal{H}^2\text{ - a.e. on }\Gamma_d
\end{equation}
for every $t\in [0,T]$.} According to this change of variable, the elastic energy at time $t$ associated to the deformation $y^{\ep}(t)$ can be written in terms of $\zep(t)$ as
$$\intom{W_{el}(\nep y^{\ep}(t) (P^{\ep})^{-1}(t))}=\intom{W_{el}\big(\nabla \pep (t,\zep(t))\nep \zep (t)(P^{\ep})^{-1}(t)\big)}.$$

For every $t\in [0,T]$ we define the three-dimensional stress as 
\begin{equation}
\nonumber
E^{\ep}(t):=\frac{1}{\ep^{\alpha-1}}DW_{el}\Big(\nabla \pep (t,\zep(t))\nep \zep (t)(P^{\ep})^{-1}(t)\Big)\Big(\nabla \pep (t,\zep(t))\nep \zep (t)(P^{\ep})^{-1}(t)\Big)^T.
\end{equation} 
Let $s_1,s_2\in [0,T]$, with $s_1\leq s_2$. For every function $t\mapsto P(t)$ from $[0,T]$ into $L^2(\Omega;SL(3))$, we define its dissipation as
$$\cal{D}(P;s_1,s_2):=\sup\Big\{\sum_{i=1}^N\intom{D(P(t_{i-1}), P(t_i))}: s_1=t_0<t_1<\cdots<t_N=s_2\Big\}.$$
Analogously, for every function $t\mapsto p(t)$ from $[0,T]$ into $L^2(\Omega;\md)$, we define its $H_D$-dissipation as
\begin{equation}
\label{NUMpag13}
\cal{D}_{H_D}(p;s_1,s_2):=\sup\Big\{\sum_{i=1}^N\intom{H_D( p(t_i)-p(t_{i-1}))}: s_1=t_0<t_1<\cdots<t_N=s_2\Big\}.
\end{equation}

Finally, we denote by $\cal{F}_{\ep}(t,z,P)$ the quantity
$$\cal{F}_{\ep}(t,z,P):=\intom{W_{el}\big(\nabla \pep (t,z)\nep z P^{-1}\big)}+\intom{W_{hard}(P)}$$
for every $t\in [0,T]$, $z\in W^{1,2}(\Omega;\R^3)$ and $P\in L^2(\Omega;SL(3))$.
We are now in a position to give the definition of quasistatic evolution associated to the boundary datum $t\mapsto\pep(t)$.
\begin{defin}\label{epquasevol}
Let $\ep>0$. 
An {\em$\ep$-quasistatic evolution} for the boundary datum $t\mapsto\pep(t)$ is a function $t\mapsto(\zep(t),P^{\ep}(t))$ from $[0,T]$ into $W^{1,2}(\Omega;\R^3)\times L^2(\Omega;SL(3))$ that satisfies the following conditions:
\begin{enumerate}
	\item[(qs1)] for every $t\in [0,T]$ we have {$\zep(t,x)=(x',\ep x_3)\quad\cal{H}^2\text{ - a.e. on }\Gamma_d$}, $P^{\ep}(t,x)\in K$ for a.e. $x\in\Omega$ and
\begin{eqnarray}
\nonumber\cal{F}_{\ep}(t,\zep(t),P^{\ep}(t))\leq \cal{F}_{\ep}(t,\tilde{z},\tilde{P})+{\ep^{\alpha-1}}\intom{D(P^{\ep}(t),\tilde{P})},
\end{eqnarray}
for every $(\tilde{z},\tilde{P})\in W^{1,2}(\Omega;\R^3)\times L^2(\Omega;SL(3))$ such that {$\tilde{z}(x)=(x',\ep x_3)\quad\cal{H}^2\text{ - a.e. on }\Gamma_d$} and $\tilde{P}(x)\in K$ for a.e. $x\in\Omega$;
	\item[(qs2)] the map 
	$$s\mapsto \intom{E^{\ep}(s):\Big(\nabla \dot{\pep}(s,\zep(s))(\nabla \pep)^{-1}(s,\zep(s))\Big)}$$ 
	is integrable in $[0,T]$ and for every $t\in[0,T]$
\begin{eqnarray}
\nonumber &&\cal{F}_{\ep}(t,\zep(t),P^{\ep}(t))+{\ep^{\alpha-1}}\cal{D}(P^{\ep};0,t)\\
\nonumber &&=\cal{F}_{\ep}(0,\zep(0),P^{\ep}(0))+{\ep^{\alpha-1}}\int_0^t{\intom{E^{\ep}(s):\Big(\nabla \dot{\pep}(s,\zep(s))(\nabla \pep)^{-1}(s,\zep(s))\Big)}\,ds}.
\end{eqnarray}
\end{enumerate}
\end{defin}
\begin{oss}
We remark that if the function $t\to (\zep(t),P^{\ep}(t))$ satisfies condition (qs1), then $E^{\ep}(t)\in L^1(\Omega;\mthree)$ for every $t\in [0,T]$. Indeed, by (qs1), taking $\tilde{z}(x)=(x',\ep x_3)$ for every $x\in\Omega$ and $\tilde{P}=P^{\ep}(t)$, we deduce
\begin{equation}
\label{preles}
\intom{W_{el}\big(\nabla \pep(t,\zep(t))\nep \zep(t)(P^{\ep})^{-1}(t)\big)}\leq \intom{W_{el}\big(\nabla \pep(t,(x',\ep x_3))(P^{\ep})^{-1}(t)\big)}.
\end{equation}
On the other hand, $P^{\ep}(t)\in K$ a.e. in $\Omega$ and for $\ep$ small enough there exists two constants $C_1$ and $C_2$ such that $\det (\nabla \pep(t,(x',\ep x_3)))\geq C_1$ and $\|\nabla \pep (t,(x',\ep x_3))\|_{L^{\infty}(\Omega;\mthree)}\leq C_2$. Therefore, by hypothesis (H1) the quantity in \eqref{preles} is finite and 
\begin{equation}
\label{detpo}
\det \big(\nabla \pep(t,\zep(t))\nep \zep(t)(P^{\ep})^{-1}(t)\big)>0\quad\text{a.e. in }\Omega
\end{equation}
for $\ep$ small enough. Finally, by \eqref{mandel2} we obtain
$$\intom{|E^{\ep}(t)|}\leq \frac{c_4}{\ep^{\alpha-1}}\Big(\intom{W_{el}(\nabla \pep(t,\zep(t))\nep \zep(t)(P^{\ep})^{-1}(t))}+1\Big)<+\infty.$$
\end{oss}
\begin{oss}
\label{Esym}
By the frame-indifference (H3) of $W_{el}$, there holds
\begin{equation}
\nonumber
DW_{el}(F)F^T=F(DW_{el}(F))^T\quad\text{for every }F\in\mthree_+.
\end{equation}
Therefore, by \eqref{detpo}, for $\ep$ small enough $E^{\ep}(t,x)\in\ms$ for every $t\in [0,T]$ and for a.e. $x\in\Omega$.
\end{oss}
Set
$$L_{\alpha}:=\begin{cases}0&\text{if }\alpha>3\\
1&\text{if }\alpha=3.
\end{cases}$$For every $\alpha\geq 3$ we define a reduced quasistatic evolution as follows.
\begin{defin}
\label{reda>3}
For $\alpha\geq 3$, a {\emph{reduced quasistatic evolution}} for the boundary data $t\mapsto u^0(t)$ and $t\mapsto v^0(t)$ is a map $t\mapsto (u(t),v(t),p(t))$ from $[0,T]$ into $W^{1,2}(\omega;\R^2)\times W^{2,2}(\omega)\times L^2(\Omega;\md)$, that satisfies the following conditions:
\begin{enumerate}
\item [(qs1$_{r\alpha}$)] for every $t\in [0,T]$ there holds $(u(t),v(t),p(t))\in\cal{A}(u^0(t),v^0(t))$, and setting 
\begin{equation}
\label{dsigmat}
e_{\alpha}(t):=\sym\nabla' u(t)+\tfrac{L_{\alpha}}{2}\nabla' v(t)\otimes\nabla' v(t)-x_3(\nabla')^2 v(t)-p'(t),
\end{equation} 
we have 
\begin{eqnarray*}
&&\intom{Q_2(e_{\alpha}(t))}+\intom{\B{p(t)}}\leq\intom{Q_2\big(\sym\,\nabla' \hat{u}+\tfrac{L_{\alpha}}{2}\nabla' \hat{v}\otimes\nabla' \hat{v}-x_3(\nabla')^2 \hat{v}-\hat{p}'\big)}\\&&+\intom{\B{\hat{p}}}+\intom{H_D(\hat{p}-p(t))},
\end{eqnarray*}
for every $(\hat{u},\hat{v},\hat{p})\in\cal{A}(u^0(t),v^0(t))$;
\item[(qs2$_{r\alpha}$)] the map 
$$s\to \intom{\C_2e_{\alpha}(s):\Big(\begin{array}{cc}\nabla' \dot{u}^0(s)+L_{\alpha}\nabla' \dot{v}^0(s)\otimes\nabla' v(s)-x_3(\nabla')^2 \dot{v}^0(s)&0\\0&0\end{array}\Big)}$$ is integrable in $[0,T]$. Moreover 
for every $t\in [0,T]$ there holds
\begin{eqnarray*}
&&\intom{Q_2(e_{\alpha}(t))}+\intom{\B{p(t)}}+\cal{D}_{H_D}(p;0,t)=\intom{Q_2(e_{\alpha}(0))}+\intom{\B{p(0)}}\\
&&+\int_0^t{\intom{\C_2e_{\alpha}(s):\Big(\begin{array}{cc}\nabla' \dot{u}^0(s)+L_{\alpha}\nabla' \dot{v}^0(s)\otimes\nabla' v(s)-x_3(\nabla')^2 \dot{v}^0(s)&0\\0&0\end{array}\Big)}\,ds}.
\end{eqnarray*}
\end{enumerate}
\end{defin}
\begin{oss}
\label{csd}
An adaptation of \cite[Theorem 4.5]{DDM} guarantees that, if $\alpha>3$, for every triple $(\overline{u},\overline{v},\overline{p})\in\cal{A}(u^0(0),v^0(0))$ satisfying
\begin{eqnarray*}
&&\intom{Q_2\big(\sym \nabla' \overline{u}-x_3(\nabla')^2 \overline{v}+\tfrac{L_{\alpha}}{2}\nabla' \overline{v}\otimes \nabla' \overline{v}-\overline{p}')}+\intom{\B{\overline{p}}}\\
&&\leq\intom{Q_2\big(\sym\nabla' \hat{u}-x_3(\nabla')^2 \hat{v}+\tfrac{L_{\alpha}}{2}\nabla' \hat{v}\otimes \nabla' \hat{v}-\hat{p}'\big)}+\intom{\B{\hat{p}}}+\intom{H_D(\hat{p}-\overline{p})},
\end{eqnarray*}
for every $(\hat{u},\hat{v},\hat{p})\in\cal{A}(u^0(0),v^0(0))$, there exists a reduced quasistatic evolution $t\mapsto (u(t),v(t),p(t))$ (according to Definition \ref{reda>3}) such that $u(0)=\overline{u}$, $v(0)=\overline{v}$ and $p(0)=\overline{p}$. Moreover, by adapting \cite[Theorem 5.2 and Remark 5.4]{DDM} one can show that the maps $t\mapsto u(t)$, $t\mapsto v(t)$ and $t\mapsto p(t)$ are Lipschitz continuous from $[0,T]$ into $W^{1,2}(\omega;\R^2)$, $W^{2,2}(\omega)$ and $L^2(\Omega;\md)$, respectively.

In the case $\alpha=3$, the existence of a reduced quasistatic evolution $t\mapsto (u(t),v(t),p(t))$ such that $(u(0),v(0),p(0))=(\overline{u},\overline{v},\overline{p})$ can still be proved by adapting \cite[Theorem 4.5]{DDM}. We remark that the proof of this result is more subtle than its counterpart in the case $\alpha>3$, due to the presence of the nonlinear term  $\frac{1}{2}\nabla'v\otimes \nabla'v$. {{In fact, for $\alpha=3$ one can not prove the analogous of \cite[Theorem 3.8]{DDM} and can not guarantee that the set of discontinuity points of the function $t\mapsto e_3(t)$ is at most countable. Hence, when trying to prove the analogous of \cite[Theorem 4.7]{DDM}, that is, to deduce the converse energy inequality by the minimality, some further difficulties arise to study convergence of the piecewise constant interpolants of $t\mapsto e_3(t)$. To cope with this problem}}, one can apply \cite[Lemma 4.12]{DFT}, which guarantees the existence of partitions of $[0,T]$ on which the Bochner integrals of some relevant quantities can be approximated by Riemann sums, and argue as in \cite[Lemma 5.7]{B}. 
\end{oss}
\begin{oss}
\label{eul3}
By taking $\hat{p}=p(t)$ in (qs1$_{r\alpha}$), it follows that a reduced quasistatic evolution $t\mapsto (u(t),v(t),p(t))$ satisfies
\begin{eqnarray*}
\intom{Q_2(e_{\alpha}(t))}\leq \intom{Q_2(\sym\,\nabla' \hat{u}+\tfrac{L_{\alpha}}{2}\nabla' \hat{v}\otimes\nabla' \hat{v}-x_3(\nabla')^2 \hat{v}-{p}'(t))}
\end{eqnarray*}
for every $(\hat{u},\hat{v})\in W^{1,2}(\omega;\R^2)\times W^{2,2}(\omega)$ such that 
{$$\hat{u}=u^0(t),\,\hat{v}=v^0(t)\text{ and }\nabla'\hat{v}=\nabla' v^0(t) \quad\cal{H}^1\text{ - a.e. on }\gamma_d.$$
 Hence, in particular, there holds
$$\intom{\C_2 e(t):\nabla' \zeta}=0$$
for every $\zeta \in W^{1,2}(\omega;\R^2)$ such that $\zeta=0$ $\cal{H}^1$ - a.e. on $\gamma_d$.}
\end{oss}
With the previous definitions at hand we are in a position to state the main result of the paper.
\begin{teo}
\label{cvstress}
Let $\alpha\geq 3$. Assume that $t\mapsto u^0(t)$ belongs to $C^1([0,T];W^{1,\infty}(\R^2;\R^2)\cap C^{1}(\R^2;\R^2))$ and $t\mapsto v^0(t)$ belongs to $C^1([0,T];W^{2,\infty}(\R^2)\cap C^{2}(\R^2))$, respectively. For every $t\in [0,T]$, let $\pep(t)$ be defined as in \eqref{defphiep}.
Let $(\mathring{u},\mathring{v},\mathring{p})\in \cal{A}(u^0(0),v^0(0))$ be such that
\begin{eqnarray}
\nonumber &&\intom{Q_2(\sym\nabla' \mathring{u}-x_3(\nabla')^2 \mathring{v}+\tfrac{L_{\alpha}}{2}\nabla' \mathring{v}\otimes \nabla' \mathring{v}-\mathring{p}')}+\intom{\B{\mathring{p}}}\\
\nonumber&&
\leq \intomm{Q_2(\nabla' \hat{u}-x_3(\nabla')^2 \hat{v}+\tfrac{L_{\alpha}}{2}\nabla' \hat{v}\otimes \nabla' \hat{v}-\hat{p}')}+\intom{\B{\hat{p}}}+\intom{H(\hat{p}-\mathring{p})},\\
\label{servedopo1}
\end{eqnarray} 
for every $(\hat{u},\hat{v},\hat{p})\in\cal{A}(u^0(0),v^0(0))$. 
Assume there exists a sequence of pairs $(y_0^{\ep},P_0^{\ep})\in \cal{A}_{\ep}(\pep(0))$ such that
\begin{equation}
\label{servedopo2}\cal{I}(y^{\ep}_0,P^{\ep}_0) \leq \cal{I}(\hat{y},\hat{P})+{\ep^{\alpha-1}}\intom{D(P^{\ep}_0,\hat{P})},
\end{equation}
for every $(\hat{y},\hat{P})\in\cal{A}_{\ep}(\pep(0))$, and
\begin{eqnarray}
\label{convu0} && u^{\ep}_0:=\frac{1}{\ep^{\alpha-1}}\intt{\big((y^{\ep}_0)'-x'\big)}\to \mathring{u}\quad\text{strongly in }W^{1,2}(\omega;\R^2),\\
\label{convv0} && v^{\ep}_0:=\frac{1}{\ep^{\alpha-2}}\intt{(y^{\ep}_0)_3}\to \mathring{v}\quad\text{strongly in }W^{1,2}(\omega),\\
\label{convP0} && p^{\ep}_0:=\frac{P^{\ep}_0-Id}{\ep^{\alpha-1}}\to \mathring{p}\quad\text{strongly in }L^2(\Omega;\md),\\
 \nonumber && \lim_{\ep\to 0}\,\frac{1}{\ep^{2\alpha-2}}\cal{I}(y^{\ep}_0, P^{\ep}_0) =\intom{Q_2(\sym \nabla' \mathring{u}-x_3(\nabla')^2 \mathring{v}+\tfrac{L_{\alpha}}{2}\nabla' \mathring{v}\otimes\nabla' \mathring{v}-{\mathring{p}}')}\\
\label{convE0} &&+\intom{\B{\mathring{p}}}.
\end{eqnarray}
Finally, for every $\ep>0$, let $t\mapsto (\zep(t),P^{\ep}(t))$ be an $\ep$-quasistatic evolution for the boundary datum $\pep(t)$ such that
$$\zep(0)=\vep(0,y^{\ep}_0)\quad\text{a.e. in }\Omega$$
and $$P^{\ep}(0)=P^{\ep}_0.$$
Then, there exists a reduced quasistatic evolution $t\mapsto (u(t),v(t),p(t))$ for the boundary data $(u^0(t),v^0(t))$ (according to Definition \ref{reda>3}), such that $u(0)=\mathring{u}$, $v(0)=\mathring{v}$, $p(0)=\mathring{p}$ and, up to subsequences,
\begin{equation}
\label{cvpt} p^{\ep}(t):=\frac{P^{\ep}(t)-Id}{\ep^{\alpha-1}}\deb p(t)\quad\text{weakly in }L^2(\Omega;\mthree)
\end{equation}
for every $t\in [0,T]$. Moreover, for $\alpha>3$ up to subsequences there holds
\begin{eqnarray}
&&\label{cvut}u^{\ep}(t):=\frac{1}{\ep^{\alpha-1}}\intt{\big(({\phi}^{\ep})'(t,\zep(t))-x'\big)}\deb u(t)\quad\text{weakly in }W^{1,2}(\omega;\R^2),\\
&&\label{cvvt}v^{\ep}(t):=\frac{1}{\ep^{\alpha-2}}\intt{\pep_3(t,\zep(t))}\to v(t)\quad\text{strongly in }W^{1,2}(\omega),
\end{eqnarray}
for every $t\in [0,T]$. For $\alpha=3$, for every $t\in [0,T]$ there exists a $t$-dependent subsequence $\ep_j\to 0$ such that
\begin{eqnarray}
&&\label{cvut3}u^{\epjt}(t):=\frac{1}{\epjt^{\alpha-1}}\intt{\big(({\phi}^{\epjt})'(t,\zepjt(t))-x'\big)}\deb u(t)\,\text{ weakly in }W^{1,2}(\omega;\R^2),\\
&&\label{cvvt3}v^{\epjt}(t):=\frac{1}{\epjt^{\alpha-2}}\intt{\pepjt_3(t,\zepjt(t))}\to v(t)\quad\text{strongly in }W^{1,2}(\omega).
\end{eqnarray}
\end{teo}
In the case $\alpha>3$ the convergence result is stronger than the analogous result for $\alpha=3$ as the convergence of $u^{\ep}(t)$ and $v^{\ep}(t)$ holds on a subsequence independent of $t$. This is related to the fact that, for $\alpha>3$,  once $t\mapsto p(t)$ is identified, both $t\mapsto u(t)$ and $t\mapsto v(t)$ are uniquely determined. In the case $\alpha=3$ this property is not true anymore because of the presence of the nonlinear term $\frac{1}{2}\nabla' v(t)\otimes \nabla' v(t)$. 

We shall prove the previous theorem in the next section. To conclude this section, we prove a technical lemma concerning some properties of the truncation maps $\tep$ and we provide the construction of the so-called ``joint recovery sequence", that will be used in the proof of Theorem \ref{cvstress}.
\begin{lem}
\label{cvproptep}
Let $\tep\in W^{1,\infty}(\R)\cap {C^{1}(\R)}$ be such that \eqref{treq1}--\eqref{lp2} hold and let $(\zeta^{\ep})$ be a sequence in $L^{2}(\Omega)$ such that 
\begin{equation}
\label{bdz}
\|\zeta^{\ep}\|_{L^2(\Omega)}\leq C\ep.
\end{equation}
Then,
\begin{equation}
\label{tep1h}\Big\|1-\dot{\tep}\Big(\frac{\zeta^{\ep}}{\ep}\Big)\Big\|_{L^2(\Omega)}\leq \frac{3}{ \ell_{\ep}}.
\end{equation}
 Moreover, if $\zeta^{\ep}$ satisfies
\begin{equation}
\label{3compus}
\big\|\frac{\zeta^{\ep}}{\ep}-x_3-\ep^{\alpha-3}v\big\|_{L^2(\Omega)}\to 0,
\end{equation}
for some $v\in L^{2}(\omega)$, then 
\begin{equation}
\label{claimvh} \tep\Big(\frac{\zeta^{\ep}}{\ep}\Big)\to\begin{cases}x_3&\text{if }\alpha>3\\
x_3+v&\text{if }\alpha=3\end{cases}\quad\text{strongly in }L^2(\Omega).\\
\end{equation}
\end{lem}
\begin{proof}
Denoting by $O_{\ep}$ the set
$$O_{\ep}:=\Big\{x\in \Omega: |\zeta^{\ep}(x)|\geq \ep\ell_{\ep}\Big\},$$
by \eqref{bdz} and by Chebychev inequality, there holds
$$\cal{L}^3(O_{\ep})\leq \frac{C}{\ell_{\ep}^2}.$$
Hence, by \eqref{treq1} and \eqref{treq5},
$$\Big\|1-\dot{\tep}\Big(\frac{\zeta^{\ep}}{\ep}\Big)\Big\|_{L^2(\Omega)}=\Big\|1-\dot{\tep}\Big(\frac{\zeta^{\ep}}{\ep}\Big)\Big\|_{L^2(O_{\ep})}\leq \frac{2}{ \ell_{\ep}}.$$

 To prove \eqref{claimvh}, we note that by \eqref{3compus} there holds
 $$\tep\Big(\frac{\zeta^{\ep}}{\ep}\Big)\to\begin{cases}x_3&\text{if }\alpha>3\\
x_3+v&\text{if }\alpha=3\end{cases}\quad\text{a.e. in }\Omega.$$
On the other hand, \eqref{treq2} yields $\big|\tep\big(\frac{\zeta^{\ep}}{\ep}\big)\big|\leq \big|\frac{\zeta^{\ep}}{\ep}\big|$ for every $\ep$ and for a.e. $x\in \Omega$.
 Therefore \eqref{claimvh} follows by the dominated convergence theorem. \end{proof}
For the sake of simplicity, in the next theorem we omit the time dependence of $u^0$ and $v^0$. With a slight abuse of notation, we denote by $\pep$ the map
\begin{equation}
\nonumber\pep(x):=\Big(\begin{array}{c}x'\\x_3\end{array}\Big)+\ep^{\alpha-1}\Big(\begin{array}{c}u^0(x')-\tep(\frac{x_3}{\ep})\nabla' v^0(x')\\0\end{array}\Big)+\ep^{\alpha-2}\Big(\begin{array}{c}0\\v^0(x')\end{array}\Big),
\end{equation} 
for a.e. $x\in\Omega$, where $u^0\in W^{1,\infty}(\R^2;\R^2)\cap C^{1}(\R^2;\R^2)$ and $v^0\in W^{2,\infty}(\R^2)\cap C^{2}(\R^2)$. 
We are now in a position to construct the joint recovery sequence.
\begin{teo}
\label{mutrecseq}
Let $(y^{\ep}, P^{\ep})\in \cal{A}_{\ep}(\pep)$ satisfy \eqref{engest2}
for every $\ep>0$. Let $u,v,G,p$ be defined as in Theorem \ref{liminfineq} and let
$\hat{u}:=u+\tilde{u}$, $\hat{v}:=v+\tilde{v}$, and $\hat{p}:=p+\tilde{p}$, where {$\tilde{u}\in C^{\infty}_c(\omega\cup\gamma_n;\R^2)$, $\tilde{v}\in C^{\infty}_c(\omega\cup\gamma_n)$} and $\tilde{p}\in C^{\infty}_c(\Omega;\md)$.
Then, there exists a sequence of pairs $(\hye,\hpe)\in \cal{A}_{\ep}(\pep)$, such that
\begin{eqnarray}
&&\label{cvidh} \hye\to \Big(\begin{array}{c}x'\\0\end{array}\Big)\quad\text{strongly in }W^{1,2}(\Omega;\R^3),\\
&&\label{4cptuepmrsh} \hat{u}^{\ep}:=\frac{1}{\ep^{\alpha-1}}\intt{\big((\hye)'-x'\big)}\deb \hat{u}\quad\text{weakly in }W^{1,2}(\omega;\R^2)\text{ for }\alpha>3,\\
&&\label{4cptuepmrsh2} \hat{u}^{\ep}\deb \hat{u}-v\nabla \tilde{v}\quad\text{weakly in }W^{1,2}(\omega;\R^2)\text{ for }\alpha=3,\\
&&\label{cptvepmrsh} \hat{v}^{\ep}:=\frac{1}{\ep^{\alpha-2}}\intt{\hye_3}\to \hat{v}\quad\text{strongly in }W^{1,2}(\omega),\\
\label{Pep1mrsh} &&\hpe(x)\in K\quad\text{for a.e. }x\in\Omega,\\
\label{wconvpepmrsh} &&
\hppe:=\frac{\hpe-Id}{\ep^{\alpha-1}}\deb \hat{p}\quad\text{weakly in }L^2(\Omega;\mthree).
\end{eqnarray}
Moreover, the following inequalities hold true:
\begin{equation}
 \label{limsuphardening}
 \limsup_{\ep\to 0}\frac{1}{\ep^{2\alpha-2}}\Big(\intom{W_{hard}(\hpe)}-\intom{W_{hard}(P^{\ep})}\Big)\leq \intom{\B{\hat{p}}}-\intom{\B{{p}}},
 \end{equation}
 \begin{equation}
 \label{limsupdiss}
 \limsup_{\ep\to 0}\frac{1}{\ep^{\alpha-1}}\intom{D(P^{\ep},\hpe)}\leq \intom{H_D(\hat{p}-p)},
 \end{equation}
and
\begin{eqnarray}
&&\nonumber\limsup_{\ep\to 0} \frac{1}{\ep^{2\alpha-2}}\Big(\intom{W_{el}(\nep \hye (\hpe)^{-1})}-\intom{W_{el}(\nep y^{\ep} (P^{\ep})^{-1})}\Big)\\
&&\nonumber\leq \intom{Q_2(\sym\,\hat{G}'-\hat{p}')}-\intom{Q_2(\sym\,{G}'-{p}')},\\
\label{limsupeng}
\end{eqnarray}
where the submatrix $\hat{G}'$ satisfies 
\begin{equation}
\nonumber
\hat{G}'(x', x_3) := \hat{G}_0(x') - x_3 (\nabla')^2 \hat{v}(x')\quad\text{for a.e. }x\in\Omega,
\end{equation}
and
\begin{eqnarray*}
&&\nonumber \sym\, \hat{G}_0 = \frac{(\nabla' \hat{u}+(\nabla' \hat{u})^T +\nabla' \hat{v}\otimes \nabla' \hat{v})}{2}\quad\text{for } \alpha=3,\\
&&\nonumber \sym\, \hat{G}_0 = \sym\, \nabla' \hat{u}\quad\text{for }\alpha> 3.
\end{eqnarray*}
\end{teo}

\begin{proof}
We divide the proof into four steps. In the first step we exhibit a sequence of deformations $(\hye)$ satisfying \eqref{cvidh}--\eqref{cptvepmrsh}. In the second step we construct a sequence $(\hpe)$ of plastic strains and we prove the limsup inequality for the hardening and the dissipation terms. In the third step we rewrite the elastic energy in terms of some auxiliary quantities and in the fourth step we prove the limsup inequality for the elastic energy.

We first remark that by \eqref{engest2} and the boundary condition
{\begin{equation}
\label{bddatum2}
 y^{\ep}(x)=\pep(x',\ep x_3)\quad\cal{H}^2\text{ - a.e. on }\Gamma_d,
\end{equation}}
 the sequence $(y^{\ep},P^{\ep})$ fulfills the hypotheses of Theorems \ref{compactbd1} and \ref{liminfineq}. Hence, there exists a sequence $(R^{\ep})\subset W^{1,\infty}(\omega;\mthree)$ such that \eqref{rt1}--\eqref{rt4} hold true, and $(y^{\ep})$ satisfies \eqref{cptyep}. Moreover, defining $u^{\ep}, v^{\ep},$ and $G^{\ep}$ according to \eqref{inplane}, \eqref{outofplane} and \eqref{defgep}, properties \eqref{cptuep}--\eqref{Ga>3} hold true. The sequence of plastic strains $(P^{\ep})$ satisfies 
\begin{equation}
\label{Pep1mrs} P^{\ep}(x)\in K\quad\text{for a.e. }x\in\Omega,
\end{equation}
and defining $p^{\ep}$ as in \eqref{defpep}, there holds
\begin{equation}
\label{wconvpepmrs}
p^{\ep}\deb p\quad\text{weakly in }L^2(\Omega;\mthree).
\end{equation}
Finally, by Theorem \ref{liminfineq}, $(u,v,p)\in\cal{A}(u^0,v^0)$ and, by \eqref{cptvep} and \eqref{3comphest}, the sequence $(y^{\ep}_3)$ fulfills the hypothesis of Lemma \ref{cvproptep}, hence 
\begin{equation}
\label{claimvhm} \tep\Big(\frac{y^{\ep}_3}{\ep}\Big)\to\begin{cases}x_3&\text{if }\alpha>3\\
x_3+v&\text{if }\alpha=3\end{cases}\quad\text{strongly in }L^2(\Omega),
\end{equation}
and by \eqref{lp2} and \eqref{tep1h}
\begin{equation}
\label{tep1hm} 
\frac{1}{\ep}-\frac{1}{\ep}\dot{\tep}\Big(\frac{y^{\ep}_3}{\ep}\Big)\to 0\quad\text{strongly in }L^2(\Omega).
\end{equation}
{\em Step 1: Construction of the deformations}\\
Let $d\in C^{\infty}_c(\R^3;\R^3)$ with $\supp\, d\subset \Omega$. Consider the map
$$\etp(x):=\int_{-\frac{1}{2}}^{\frac{x_3}{\ep}}{d(x',s)\,ds}\quad\text{for every }x\in \R^3.$$
Since $d$ has compact support in $\Omega$, there holds
$$|\etp(x)|\leq \int_{-\frac{1}{2}}^{\big|\frac{x_3}{\ep}\big|}{|d(x',s)|\,ds}\leq \int_{-\frac{1}{2}}^{\frac{1}{2}}{|d(x',s)|\,ds}\quad\text{for every }x\in\R^3$$
and analogously
\begin{equation}
\label{etaep0}
\|{\nabla}'{\etp}\|_{L^{\infty}(\R^3;\M^{3\times 2})}\leq \|{\nabla}' d\|_{L^{\infty}(\R^3;\M^{3\times 2})}.
\end{equation}
A straightforward computation yields
\begin{equation}
\label{etaep12}
\partial_3 \etp(x)=\frac{1}{\ep}d\Big(x',\frac{x_3}{\ep}\Big)\quad\text{for every }x\in\R^3.
\end{equation}
Hence,
\begin{equation}
\label{etaep1}
\|\etp\|_{W^{1,\infty}(\R^3;\R^3)}\leq \frac{C}{\ep}.
\end{equation}
In particular, the map $\etp\circ y^{\ep}$ satisfies
\begin{eqnarray}
&&\label{etaep2} \|\etp\circ y^{\ep}\|_{L^{\infty}(\Omega;\R^3)}\leq C,\\
&&\label{etaep3} \|{\nabla}'(\etp\circ y^{\ep})\|_{L^2(\Omega;\M^{3\times 2})}\leq C \|\nabla'(y^{\ep})'\|_{L^2(\Omega;\M^{2\times 2})}+\frac{C}{\ep}\|\nabla'y^{\ep}_3\|_{L^2(\Omega;\R^2)}.
\end{eqnarray}

We { extend $\tilde{u}$ and $\tilde{v}$ to zero outside their support, we} consider the functions
$$f^{\ep}(x):=x+\Big(\begin{array}{c}\ep^{\alpha-1}\tilde{u}(x')\\\ep^{\alpha-2}\tilde{v}(x')\end{array}\Big)-\Big(\begin{array}{c}\ep^{\alpha-1}\tep\big(\frac{x_3}{\ep}\big)\nabla' \tilde{v}(x')\\0\end{array}\Big)+\ep^{\alpha}\etp(x)$$
for every $x\in\R^3$, and we set
$$\hye:=f^{\ep}\circ y^{\ep}.$$

It is easy to see that $\hye\in W^{1,2}(\Omega;\R^3)$, we now check that 
{\begin{equation}
\label{bddat}
\hye=\phi^{\ep}(x',\ep x_3)\quad\cal{H}^2\text{ - a.e. on }\Gamma_d.
\end{equation}  
To prove it, we first remark that by \eqref{bddatum2}
\begin{equation}
\label{almbddat}
\hye=f^{\ep}(\pep(x',\ep x_3))\quad\cal{H}^2\text{ - a.e. on }\Gamma_d.\end{equation}
Hence, it remains only to show that
\begin{equation}
\label{almbddat2}
f^{\ep}(\pep(x',\ep x_3))=\pep(x',\ep x_3)\quad\cal{H}^2\text{ - a.e. on }\Gamma_d.
\end{equation}
Let $A\subset \R^2$ be an open set such that $\overline{\gamma_d}\subset (A\cap\partial\omega)$ and $\tilde{u},\tilde{v},\nabla'\tilde{v}=0$ in $A$.
Since $d$ has compact support in $\Omega$, without loss of generality we may assume that $\etp(x)=0$ for all $x\in A\times \R$ and every $\ep$. Therefore, we have $f^{\ep}(x)=x$ in $A\times\big(-\tfrac 12,\tfrac 12\big)$. Let now $O\subset \R^2$ be an open set such that $\overline{\gamma_d}\subset (O\cap\partial\omega)$ and $\overline{O}\subset A$, and let $0<\delta_0<\dist(O,\partial A)$. By \eqref{treq2}, there holds
$$|(\pep)'(x',\ep x_3)-x'|\leq \ep^{\alpha-1}\|u^0\|_{L^{\infty}(\R^2;\R^2)}+\frac{1}{2}\ep^{\alpha-2}\|\nabla' v^0\|_{L^{\infty}(\R^2;\M^{2\times 2})}< \frac{\delta_0}{2}$$
for every $x\in O\times\big(-\tfrac 12,\tfrac 12\big)$, for $\ep$ small enough. Hence, $\pep(x',\ep x_3)\in A$ for every $x\in O\times\big(-\tfrac 12,\tfrac 12\big)$, and $f^{\ep}(\phi^{\ep}(x',\ep x_3))=\phi^{\ep}(x',\ep x_3)$ for every $x\in O\times\big(-\tfrac 12,\tfrac 12\big)$. This implies \eqref{almbddat2} and \eqref{bddat}.}

To prove \eqref{cvidh}, we remark that by the smoothness of $\tilde{u}$ and $\tilde{v}$, estimates \eqref{treq3}, \eqref{treq5}, \eqref{lp2} and \eqref{etaep1} imply
\begin{equation}
\label{unifbdvarphi}
\|f^{\ep}-id\|_{W^{1,\infty}(\R^3;\R^3)}\leq C\ep^{\alpha-1}\ell_{\ep}.
\end{equation}
On the other hand, we have
\begin{eqnarray*}
&&\Big\|\hye-\Big(\begin{array}{c}x'\\0\end{array}\Big)\Big\|_{W^{1,2}(\Omega;\R^3)}\leq \|\hye-y^{\ep}\|_{W^{1,2}(\Omega;\R^3)}+\Big\|y^{\ep}-\Big(\begin{array}{c}x'\\0\end{array}\Big)\Big\|_{W^{1,2}(\Omega;\R^3)}\\
&&\leq C\|f^{\ep}-id\|_{W^{1,\infty}(\R^3;\R^3)}\|\nabla y^{\ep}\|_{L^2(\Omega;\mthree)}+\Big\|y^{\ep}-\Big(\begin{array}{c}x'\\0\end{array}\Big)\Big\|_{W^{1,2}(\Omega;\R^3)},
\end{eqnarray*}
so that \eqref{cvidh} follows by \eqref{cptyep}, \eqref{lp1} and \eqref{unifbdvarphi}.

We now prove convergence of the out-of-plane displacements associated to $(\hye)$.
To show \eqref{cptvepmrsh} we note that
\begin{equation}
\nonumber
\hat{v}^{\ep}=\frac{1}{\ep^{\alpha-2}}\intt{f^{\ep}_3(y^{\ep})}
 =v^{\ep}+\intt{\tilde{v}((y^{\ep})')}+\ep^2\intt{\etp_3(y^{\ep})}.
\end{equation}
By \eqref{cptyep}, up to subsequences, we can assume
\begin{equation}
(y^{\ep})'\to x'\quad\text{and }\nabla'(y^{\ep})'\to Id\quad\text{a.e. in }\Omega.
\label{ptcvid}
\end{equation}
Hence, by the dominated convergence theorem  and the smoothness of $\tilde{v}$ we obtain 
$$\tilde{v}((y^{\ep})')\to \tilde{v}\quad\text{strongly in }L^2(\Omega)$$
and
$$\nabla'\tilde{v}( (y^{\ep})')\to \nabla'\tilde{v}\quad\text{strongly in }L^2(\Omega;\R^2).$$
 By \eqref{cptyep}, \eqref{cptvep}, \eqref{etaep2} and \eqref{etaep3} we conclude
\begin{equation}
\nonumber
\hat{v}^{\ep}\to v+\tilde{v}=\hat{v}\quad\text{strongly in }W^{1,2}(\omega).
\end{equation}
To prove \eqref{4cptuepmrsh} and \eqref{4cptuepmrsh2} we note that
\begin{eqnarray}
\nonumber &&
\hat{u}^{\ep}=u^{\ep}+\intt{\tilde{u}((y^{\ep})')}-\intt{\tep\Big(\frac{y^{\ep}_3}{\ep}\Big)\nabla' \tilde{v}((y^{\ep})')}\\
\label{expuh} &&+\ep\intt{(\etp)'(y^{\ep})}.
\end{eqnarray}
By \eqref{claimvhm}, \eqref{ptcvid} and the dominated convergence theorem,
\begin{eqnarray}
&&\nonumber \tilde{u}((y^{\ep})')\to \tilde{u}\quad\text{strongly in }L^2(\Omega;\R^2),\\
&&\nonumber
\intt{\tep\Big(\frac{y^{\ep}_3}{\ep}\Big)\nabla' \tilde{v}((y^{\ep})')}\to 0\quad\text{strongly in }L^2(\omega;\R^2)\quad\text{for }\alpha>3,\\
&&\nonumber
\intt{\tep\Big(\frac{y^{\ep}_3}{\ep}\Big)\nabla' \tilde{v}((y^{\ep})')}\to v\nabla\tilde{v}\quad\text{strongly in }L^2(\omega;\R^2)\quad\text{for }\alpha=3.
\end{eqnarray}
Hence, by \eqref{etaep2}, we have
$$\hat{u}^{\ep}\to \hat{u}\quad\text{strongly in }L^2(\omega;\R^2)\quad\text{for }\alpha>3,$$
and
$$\hat{u}^{\ep}\to \hat{u}-v\nabla \tilde{v}\quad\text{strongly in }L^2(\omega;\R^2)\quad\text{for }\alpha=3.$$
To complete the proof of \eqref{4cptuepmrsh} and \eqref{4cptuepmrsh2}, it remains to show that 
 \begin{equation}
 \label{wcvgraduh}
 \frac{1}{\ep^{\alpha-1}}\nabla'\hat{u}^{\ep}\text{ is bounded in }L^2(\Omega;\M^{2\times 2}).
 \end{equation}
   By \eqref{expuh} there holds 
  \begin{eqnarray*}
&&\frac{1}{\ep^{\alpha-1}}\nabla'\hat{u}^{\ep}=\nabla'u^{\ep}+\intt{\nabla'\tilde{u}((y^{\ep})') \nabla'(y^{\ep})'}\\
&&-\intt{\tepp(\nabla')^2 \tilde{v}((y^{\ep})') \nabla'(y^{\ep})'}-\frac{1}{\ep}\intt{\teppp\nabla' \tilde{v}((y^{\ep})')\otimes \nabla'y^{\ep}_3}\\
&&+\ep\intt{\nabla'(\etp\circ y^{\ep})}.
\end{eqnarray*}
   By adding and subtracting the matrix $(R^{\ep})'$ we obtain
  \begin{eqnarray*}
 &&\intt{\tep\Big(\frac{y^{\ep}_3}{\ep}\Big)(\nabla')^2 \tilde{v}((y^{\ep})')\nabla' (y^{\ep})'}=\intt{\tep\Big(\frac{y^{\ep}_3}{\ep}\Big)(\nabla')^2 \tilde{v}((y^{\ep})')\big(\nabla' (y^{\ep})'-(R^{\ep})'\big)}\\
  &&+\intt{\tep\Big(\frac{y^{\ep}_3}{\ep}\Big)(\nabla')^2 \tilde{v}((y^{\ep})')(R^{\ep})'}.
  \end{eqnarray*}
  Combining \eqref{rt2} and \eqref{treq3}, we deduce
  \begin{equation}
  \nonumber
  \Big\|\tep\Big(\frac{y^{\ep}_3}{\ep}\Big)(\nabla')^2 \tilde{v}((y^{\ep})')\big(\nabla' (y^{\ep})'-(R^{\ep})'\big)\Big\|_{L^2(\Omega;\mathbb{M}^{2\times 2})}\leq C\ep^{\alpha-1}\ell_{\ep}.
  \end{equation}
  On the other hand, by \eqref{rt1} and \eqref{claimvhm}, the maps    
 $\tep\Big(\frac{y^{\ep}_3}{\ep}\Big)\nabla^2 \tilde{v}((y^{\ep})')(R^{\ep})'\text{ are bounded in }L^2(\Omega;\M^{2\times 2}).$ The $L^2$-boundedness of the quantity in \eqref{wcvgraduh} follows now by combining \eqref{cptyep}, \eqref{cptuep}, \eqref{cptnep3}, \eqref{treq5} and \eqref{etaep3}.\\

  {\em Step 2: Construction of the plastic strains}\\
 Arguing as in \cite[Proof of Lemma 3.6]{MS}, we introduce the sets
 $$\sep:=\{x\in \Omega: \exp(\ep^{\alpha-1}\tilde{p}(x))P^{\ep}(x)\in K\},$$
 we define
 $$\hppe:=\begin{cases}\frac{1}{\ep^{\alpha-1}}\big(\exp(\ep^{\alpha-1}\tilde{p})P^{\ep}-Id\big)&\text{in }\sep,\\
 p^{\ep}&\text{in }\Omega\setminus S_{\ep},\end{cases}$$
 and 
 \begin{equation}
\nonumber
 \hpe:=Id+\ep^{\alpha-1}\hppe,
 \end{equation}
 so that, by \eqref{Pep1mrs}, the sequence $(\hpe)$ satisfies \eqref{Pep1mrsh}. Since $\tr{\tilde{p}}=0$, 
 $$\det(\exp(\ep^{\alpha-1}\tilde{p}))=\exp(\ep^{\alpha-1}\tr{\tilde{p}})=1,$$
 therefore
 \begin{equation}
 \label{sl3ae}
\exp(\ep^{\alpha-1}\tilde{p}(x))P^{\ep}(x)\in SL(3)\quad\text{for a.e. }x\in \Omega.
 \end{equation}
 By \eqref{sl3ae} we can estimate $\cal{L}^3(\Omega\setminus \sep)$. Indeed by \eqref{prk2} and \eqref{wconvpepmrs} there holds
 \begin{eqnarray}
\nonumber &&
 \cal{L}^3(\Omega\setminus \sep)\leq c_k^2\intom{|\big(\exp(\ep^{\alpha-1}\tilde{p}(x))P^{\ep}(x)-Id|^2} \\
 \nonumber &&=c_k^2\intom{|\big(\exp(\ep^{\alpha-1}\tilde{p}(x))+\ep^{\alpha-1}\exp(\ep^{\alpha-1}\tilde{p}(x))p^{\ep}(x)-Id|^2}\\
  \label{measnotk}  &&
\leq C\ep^{2(\alpha-1)}\intom{(1+|p^{\ep}(x)|^2)}\leq C\ep^{2(\alpha-1)}.
 \end{eqnarray}
Now,
 \begin{equation}
 \label{diffhen}
 \hppe-p^{\ep}=\begin{cases} \frac{1}{\ep^{\alpha-1}}\big(\exp(\ep^{\alpha-1}\tilde{p})-Id\big)P^{\ep}&\text{in }\sep,\\
 0&\text{in }\Omega\setminus S_{\ep}.\end{cases}
 \end{equation}
 By \eqref{Pep1mrs}, \eqref{wconvpepmrs} and \eqref{measnotk} we deduce the following convergence properties:
 \begin{equation}
  \label{sumdiffph}
 \begin{cases}
   \|\hppe-p^{\ep}\|_{L^{\infty}(\Omega;\mthree)}\leq C,&\\
 \hppe-p^{\ep}\to \tilde{p}\quad\text{strongly in }L^2(\Omega;\mthree),&\\
 \hppe+p^{\ep}\deb \hat{p}+p\quad\text{weakly in }L^2(\Omega;\mthree),&
 \end{cases}
 \end{equation}
 hence in particular \eqref{wconvpepmrsh} holds true.
 Arguing exactly as in \cite[Proof of Lemma 3.6, Step 2 and Step 4]{MS}, we obtain \eqref{limsuphardening} and \eqref{limsupdiss}.\\
 {\em Step 3: Convergence properties of the elastic energy}\\
 To complete the proof of the theorem it remains to prove \eqref{limsupeng}.
 To this purpose, let $w^{\ep}$ be the map defined as
 \begin{equation}
\label{defwep}
w^{\ep}:=\frac{(P^{\ep})^{-1}-Id+\ep^{\alpha-1}p^{\ep}}{\ep^{\alpha-1}}=\ep^{\alpha-1}(P^{\ep})^{-1}(p^{\ep})^2.
\end{equation}
 By  \eqref{prk1} and \eqref{Pep1mrs}, there exists a constant $C$ such that 
 \begin{equation}
\nonumber
 \ep^{\alpha-1}\|p^{\ep}\|_{L^{\infty}(\Omega;\mthree)}\leq C
 \end{equation} 
 and 
 \begin{equation}
 \label{linftywephm}
 \ep^{\alpha-1}\|w^{\ep}\|_{L^{\infty}(\Omega;\mthree)}\leq C 
 \end{equation}  
 for every $\ep$. Furthermore, by \eqref{wconvpepmrs},
 $$\|w^{\ep}\|_{L^1(\Omega;\mthree)}\leq C\ep^{\alpha-1}\quad\text{for every }\ep.$$
 By the two previous estimates it follows that $(w^{\ep})$ is uniformly bounded in $L^2(\Omega;\mthree)$ and  
 \begin{equation}
 \label{l2wephm}
 w^{\ep}\deb 0\quad\text{weakly in }L^2(\Omega;\mthree).
 \end{equation} 
  Now, by \eqref{defgep} and the frame-indifference property (H3) of $W_{el}$ there holds
 \begin{eqnarray}
 \nonumber W_{el}(\nep y^{\ep}(P^{\ep})^{-1})&=&W_{el}\big((Id+\ep^{\alpha-1}G^{\ep})(Id+\ep^{\alpha-1}(w^{\ep}-p^{\ep}))\big)\\
 \label{id+fep}&=&W_{el}(Id+\ep^{\alpha-1}F^{\ep}),
 \end{eqnarray}
 for a.e. $x\in\Omega$, where 
 \begin{equation}
 \label{deffep1}
 F^{\ep}:=G^{\ep}+w^{\ep}-p^{\ep}+\ep^{\alpha-1}G^{\ep}(w^{\ep}-p^{\ep}).
 \end{equation}
 We note that $$\|G^{\ep}(w^{\ep}-p^{\ep})\|_{L^1(\Omega;\mthree)}\leq C$$
 by \eqref{cptGep}, \eqref{wconvpepmrs} and \eqref{l2wephm}. Moreover, by \eqref{cptGep}, \eqref{Pep1mrs} and \eqref{linftywephm},
 $$\ep^{\alpha-1}\|G^{\ep}(w^{\ep}-p^{\ep})\|_{L^2(\Omega;\mthree)}\leq\ep^{\alpha-1}\|G^{\ep}\|_{L^2(\Omega;\mthree)}\|(w^{\ep}-p^{\ep})\|_{L^{\infty}(\Omega;\mthree)}\leq C$$
 for every $\ep$. Hence
 $$\ep^{\alpha-1}G^{\ep}(w^{\ep}-p^{\ep})\deb 0\quad\text{weakly in }L^2(\Omega;\mthree),$$
 which in turn, by \eqref{cptGep}, \eqref{wconvpepmrs} and \eqref{l2wephm}, yields
 \begin{equation}
 \label{wconvFeh}
 F^{\ep}\deb G-p\quad\text{weakly in }L^2(\Omega;\mthree).
 \end{equation}
 
Analogously, we define
\begin{equation}
\label{defhwep}
\hat{w}^{\ep}:=\frac{(\hpe)^{-1}-Id+\ep^{\alpha-1}\hppe}{\ep^{\alpha-1}}=\ep^{\alpha-1}(\hpe)^{-1}(\hppe)^2.
\end{equation}
Then,
$$(\hat{P}^{\ep})^{-1}=Id+\ep^{\alpha-1}(\hat{w}^{\ep}-\hat{p}^{\ep}),$$
by \eqref{prk1} and \eqref{Pep1mrsh} we deduce
 \begin{equation}
\nonumber
 \ep^{\alpha-1}\|\hat{w}^{\ep}\|_{L^{\infty}(\Omega;\mthree)}\leq C,
 \end{equation}  
 and by \eqref{wconvpepmrsh},
  \begin{equation}
\nonumber
 \hat{w}^{\ep}\deb 0\quad\text{weakly in }L^2(\Omega;\mthree).
 \end{equation} 
 We define 
 \begin{equation}
 \label{defhgep}
 \hat{G}^{\ep}:=G^{\ep}+\hat{w}^{\ep}-\hppe+\ep^{\alpha-1}G^{\ep}(\hat{w}^{\ep}-\hppe).
 \end{equation}
Arguing as before, we can prove that 
\begin{equation}
\label{partconvfh}
\hat{G}^{\ep}\deb G-\hat{p}\quad\text{weakly in }L^2(\Omega;\mthree).
\end{equation}

We shall prove that there exists a sequence $(\hat{F}^{\ep})\subset L^2(\Omega;\mthree)$ satisfying $$W_{el}(\nep \hye(\hpe)^{-1})=W_{el}(Id+\ep^{\alpha-1}\hat{F}^{\ep})$$
and such that
\begin{equation}
\label{propfhep}
\hat{F}^{\ep}-\hat{G}^{\ep}\to N_{\alpha}\quad\text{strongly in }L^2(\Omega;\mthree),
\end{equation}
where 
\begin{equation}
\label{deffa}
{N}_{\alpha}:=\sym\Big(\begin{array}{c}\nabla' \tilde{u}-x_3(\nabla')^2 \tilde{v}\\0\end{array}\Big|d\Big)\quad\text{for }\alpha>3,
\end{equation}
and
\begin{eqnarray}
 N_{3}&:=&\sym\Big(\begin{array}{c}\nabla \tilde{u}-(x_3+v)(\nabla')^2 \tilde{v}+\frac{\nabla' \tilde{v}\otimes\nabla' \tilde{v}}{2}\\0\end{array}\Big|\begin{array}{c}d'(x',x_3+v)\\d_3(x',x_3+v)+\frac{1}{2}|\nabla'\tilde{v}|^2 \end{array}\Big)
\label{deffa3}
\end{eqnarray}
 a.e. in $\Omega$. 
To this purpose, we first observe that by \eqref{defgep}, \eqref{defhgep} and the frame-indifference hypothesis (H3) there holds
\begin{eqnarray}
\nonumber&&W_{el}(\nep \hye (\hpe)^{-1})=W_{el}\big(\nabla f^{\ep}(y^{\ep})\nep y^{\ep}(\hpe)^{-1}\big)\\
\label{eqwel}&&=W_{el}\Big((R^{\ep})^T\sqrt{(\nabla f^{\ep}(y^{\ep}))^T\nabla f^{\ep}(y^{\ep})}R^{\ep}(Id+\ep^{\alpha-1}\hat{G}^{\ep})\Big).
\end{eqnarray}
We set  
$$M^{\ep}(x):=\frac{\nabla f^{\ep}(x)-Id}{\ep^{\alpha-1}\ell_{\ep}}.$$
By \eqref{unifbdvarphi} there holds 
\begin{equation}
\label{unifbdmep}
\|M^{\ep}(y^{\ep})\|_{L^{\infty}(\Omega;\mthree)}\leq C
\end{equation}
for every $\ep$. 

We claim that, to prove \eqref{propfhep} it is enough to show that
\begin{equation}
\label{symmep}
\ell_{\ep}(R^{\ep})^T\sym \big(M^{\ep}(y^{\ep})\big)R^{\ep}\to\begin{cases} \sym\Big(\begin{array}{c}\nabla' \tilde{u}-x_3(\nabla')^2 \tilde{v}\\0\end{array}\Big|d\Big)&\text{if }\alpha>3\\
\sym\Big(\begin{array}{c}\nabla' \tilde{u}-(x_3+v)(\nabla')^2 \tilde{v}\\0\end{array}\Big|d(x',x_3+v)\Big)&\text{if }\alpha=3\end{cases}
\end{equation}
strongly in $L^2(\Omega;\mthree)$, and
\begin{equation}
\label{meptmep}
\ep^2\ell_{\ep}^2(R^{\ep})^T(M^{\ep}(y^{\ep}))^TM^{\ep}(y^{\ep})R^{\ep}\to\Big(\begin{array}{cc}\nabla' \tilde{v}\otimes\nabla' \tilde{v}&0\\0&|\nabla'\tilde{v}|^2\end{array}\Big)\quad\text{if }\alpha=3
\end{equation}
strongly in $L^2(\Omega;\mthree)$. Indeed, a Taylor expansion around the identity yields
$$\sqrt{(Id+F)^T(Id+F)}=Id+\sym\, F+\frac{F^TF}{2}-\frac{(\sym\, F)^2}{2}+O(|F|^3)$$
for every $F\in\mthree$.
Hence,
\begin{eqnarray*}
&&\sqrt{(\nabla f^{\ep}(y^{\ep}))^T\nabla f^{\ep}(y^{\ep})}=Id+\ep^{\alpha-1}\ell_{\ep}\,\sym\, M^{\ep}(y^{\ep})+\frac{\ep^{2\alpha-2}\ell_{\ep}^2}{2}(M^{\ep}(y^{\ep}))^TM^{\ep}(y^{\ep})\\
&&-\frac{\ep^{2\alpha-2}\ell_{\ep}^2}{2}\big(\sym\, M^{\ep}(y^{\ep})\big)^2+O(\ep^{3\alpha-3}\ell_{\ep}^3).
\end{eqnarray*}
Substituting the previous expression into \eqref{eqwel} we obtain
\begin{equation}
\label{id+hfep}
W_{el}(\nep \hye(\hpe)^{-1})=W_{el}(Id+\ep^{\alpha-1}\hat{F}^{\ep})
\end{equation}
where
\begin{eqnarray}
\nonumber\hat{F}^{\ep}&=&\hat{G}^{\ep}+\ell_{\ep}(R^{\ep})^T\sym\, M^{\ep}(y^{\ep})R^{\ep}+\frac{\ep^{\alpha-1}\ell_{\ep}^2}{2}(R^{\ep})^T(M^{\ep}(y^{\ep}))^TM^{\ep}(y^{\ep})R^{\ep}\\
\nonumber &&-\frac{\ep^{\alpha-1}\ell_{\ep}^2}{2}(R^{\ep})^T\big(\sym\, M^{\ep}(y^{\ep})\big)^2R^{\ep}+\ep^{\alpha-1}\ell_{\ep}(R^{\ep})^T\sym\, M^{\ep}(y^{\ep})R^{\ep}\hat{G}^{\ep}+O(\ep^{2\alpha-2}\ell_{\ep}^3)\\
\nonumber &&+O(\ep^{2\alpha-2}\ell_{\ep}^2)\hat{G}^{\ep}
\end{eqnarray}
Now, if $\alpha>3$, by \eqref{rt1} and \eqref{unifbdmep} there holds
\begin{eqnarray*}
\|\hat{F}^{\ep}-\hat{G}^{\ep}-\ell_{\ep}(R^{\ep})^T\sym (M^{\ep}(y^{\ep}))R^{\ep}\|_{L^2(\Omega;\mthree)}&\leq& C\ep^{\alpha-1}\ell_{\ep}^2+C\ep^{\alpha-1}\ell_{\ep}\|\hat{G}^{\ep}\|_{L^2(\Omega;\mthree)}\\
&+& C\ep^{2\alpha-2}\ell_{\ep}^3+C\ep^{2\alpha-2}\ell_{\ep}^2\|\hat{G}^{\ep}\|_{L^2(\Omega;\mthree)}.
\end{eqnarray*}
Hence, by combining \eqref{lp1}, \eqref{lp4}, \eqref{partconvfh} and \eqref{symmep} we deduce \eqref{propfhep}.

In the case $\alpha=3$, by \eqref{unifbdmep} and \eqref{symmep} there holds
$$\ep^4\ell_{\ep}^4\intom{|\sym (M^{\ep}(y^{\ep}))|^4}\leq C\ep^4\ell_{\ep}^4\intom{|\sym (M^{\ep}(y^{\ep}))|^2}\leq C\ep^4\ell_{\ep}^2.$$
Therefore, by \eqref{rt1} and \eqref{unifbdmep} we have
\begin{eqnarray*}
\|\hat{F}^{\ep}-\hat{G}^{\ep}-\ell_{\ep}(R^{\ep})^T\sym (M^{\ep}(y^{\ep}))R^{\ep}-\frac{\ep^2\ell_{\ep}^2}{2}(R^{\ep})^T (M^{\ep}(y^{\ep}))^TM^{\ep}(y^{\ep})R^{\ep}\|_{L^2(\Omega;\mthree)}\\
\leq C\ep^2\ell_{\ep}+C\ep^2\ell_{\ep}\|\hat{G}^{\ep}\|_{L^2(\Omega;\mthree)}+C\ep^4\ell_{\ep}^3\\
+C\ep^4\ell_{\ep}^2\|\hat{G}^{\ep}\|_{L^2(\Omega;\mthree)}.
\end{eqnarray*}
Therefore, once \eqref{symmep} and \eqref{meptmep} are proved, \eqref{propfhep} follows by \eqref{lp1}, \eqref{lp3} and \eqref{partconvfh}.

We now prove \eqref{symmep} and \eqref{meptmep}. By straightforward computations we have
\begin{eqnarray*}
&&\ell_{\ep}\sym \,(M^{\ep}(y^{\ep}))=\sym\Big(\begin{array}{cc}\nabla' \tilde{u}((y^{\ep})')-\teppm(\nabla')^2 \tilde{v}((y^{\ep})')&0\\0&0\end{array}\Big)\\
&&+\frac{1}{\ep}\sym\Big(\begin{array}{cc}0&\big(1-\tepppm\big)\nabla' \tilde{v}((y^{\ep})')\\0&0\end{array}\Big)+\ep\,\sym (\nabla'\etp(y^{\ep})|\partial_3 \etp(y^{\ep})).
\end{eqnarray*}
Now, $\ep \nabla'\etp(y^{\ep})\to 0$ strongly in $L^2(\Omega;\M^{3\times 2})$ by \eqref{etaep0}. Moreover, \eqref{etaep12} yields
$$\ep \partial_3 \etp(y^{\ep}(x))=d\Big((y^{\ep})'(x),\frac{y^{\ep}_3(x)}{\ep}\Big)\quad\text{for a.e. }x\in\Omega.$$
Hence, by \eqref{3comphest} and \eqref{ptcvid}, there holds
\begin{equation}
\label{star1}
\ep\,(\nabla'\etp(y^{\ep})|\partial_3 \etp(y^{\ep}))\to\begin{cases} (0|d)&\text{if }\alpha>3\\(0|d(x',x_3+v))&\text{if }\alpha=3\end{cases}
\end{equation}
strongly in $L^2(\Omega;\mthree).$
On the other hand, by \eqref{claimvhm}, \eqref{ptcvid}, and the dominated convergence theorem
\begin{equation}
\label{star2}
\nabla' \tilde{u}((y^{\ep})')-\tep\Big(\frac{y^{\ep}_3}{\ep}\Big)(\nabla')^2 \tilde{v}((y^{\ep})')\to\begin{cases}\nabla' \tilde{u}-x_3(\nabla')^2 \tilde{v}&\text{if }\alpha>3\\
\nabla' \tilde{u}-(x_3+v)(\nabla')^2 \tilde{v}&\text{if }\alpha=3\end{cases}
\end{equation}
strongly in $L^2(\Omega;\M^{2\times 2})$.
Claim \eqref{symmep} follows now by combining \eqref{rt1}, \eqref{rt4}, \eqref{tep1hm}, \eqref{star1} and \eqref{star2}. 

To prove \eqref{meptmep}, we observe that by \eqref{tep1hm}, \eqref{star1} and \eqref{star2}, if $\alpha=3$ there exists a constant $C$ such that
$$\Big\|\ell_{\ep}M^{\ep}(y^{\ep})-\frac{1}{\ep}\Big(\begin{array}{cc}0&-\nabla' \tilde{v}((y^{\ep})')\\(\nabla' \tilde{v}((y^{\ep})'))^T&0\end{array}\Big)\Big\|_{L^2(\Omega;\mthree)}\leq C$$
for every $\ep$. Hence, by \eqref{rt1} there holds
\begin{eqnarray}
\nonumber&&\Big\|\ep^2\ell_{\ep}^2(R^{\ep})^T(M^{\ep}(y^{\ep}))^TM^{\ep}(y^{\ep})R^{\ep}\\
\nonumber&&-(R^{\ep})^T\Big(\begin{array}{cc}0&-\nabla' \tilde{v}((y^{\ep})')\\(\nabla' \tilde{v}((y^{\ep})'))^T&0\end{array}\Big)^T\Big(\begin{array}{cc}0&-\nabla' \tilde{v}((y^{\ep})')\\(\nabla' \tilde{v}((y^{\ep})'))^T&0\end{array}\Big)R^{\ep}\Big\|_{L^2(\Omega;\mthree)}\\
\nonumber&&\leq C\ep^2\ell_{\ep}\|M^{\ep}(y^{\ep})\|_{L^{\infty}(\Omega;\mthree)}+C\ep\Big\|\Big(\begin{array}{cc}0&-\nabla' \tilde{v}((y^{\ep})')\\(\nabla' \tilde{v}((y^{\ep})'))^T&0\end{array}\Big)\Big\|_{L^{\infty}(\Omega;\mthree)},\\
\label{est1a3}
\end{eqnarray}
which converges to zero due to \eqref{lp1} and \eqref{unifbdmep}.
On the other hand, 
\begin{eqnarray*}
&&\Big(\begin{array}{cc}0&-\nabla' \tilde{v}((y^{\ep})')\\(\nabla' \tilde{v}((y^{\ep})'))^T&0\end{array}\Big)^T\Big(\begin{array}{cc}0&-\nabla' \tilde{v}((y^{\ep})')\\(\nabla' \tilde{v}((y^{\ep})'))^T&0\end{array}\Big)\\
&&=\Big(\begin{array}{cc}\nabla' \tilde{v}((y^{\ep})')\otimes \nabla' \tilde{v}((y^{\ep})')&0\\0&|\nabla' \tilde{v}((y^{\ep})')|^2\end{array}\Big).
\end{eqnarray*}
Moreover, by \eqref{ptcvid} and by the dominated convergence theorem there holds
\begin{equation}
\label{est2a3}\Big(\begin{array}{cc}\nabla' \tilde{v}((y^{\ep})')\otimes \nabla' \tilde{v}((y^{\ep})')&0\\0&|\nabla' \tilde{v}((y^{\ep})')|^2\end{array}\Big)\to \Big(\begin{array}{cc}\nabla' \tilde{v}\otimes\nabla' \tilde{v}&0\\0&|\nabla' \tilde{v}|^2\end{array}\Big)
\end{equation}
strongly in $L^2(\Omega;\mthree)$. Combining \eqref{est1a3} and \eqref{est2a3}, we deduce \eqref{meptmep} and therefore \eqref{propfhep}.\\
{\emph{Step 4: Limsup inequality for the elastic energy}}\\
We are now in a position to prove \eqref{limsupeng}. We argue as in \cite[Lemma 3.6]{MS}. We fix $\delta>0$ and we introduce the sets
$$U_{\ep}:=\{x\in \Omega: \ep^{\alpha-1}(|F^{\ep}|+|\hat{F}^{\ep}|)\leq c_{el}(\delta)\},$$
where $c_{el}(\delta)$ is the constant in \eqref{quadrwel}.
By \eqref{partconvfh} and \eqref{propfhep} it follows that
\begin{equation}
\label{propfhep2}
\hat{F}^{\ep}\deb \hat{F}_{\alpha}:=N_{\alpha}+G-\hat{p}\quad\text{weakly in }L^2(\Omega;\mthree),\text{ for }\alpha\geq 3.
\end{equation}
By \eqref{wconvFeh} and by Chebychev inequality we deduce
\begin{equation}
\label{measUeph}
\cal{L}^3(\Omega\setminus U_{\ep})\leq C\ep^{2\alpha-2}.
\end{equation}
Since $$\nep \hye (\hpe)^{-1}=\nabla f^{\ep}(y^{\ep})\big(\nep y^{\ep}(P^{\ep})^{-1}\big)P^{\ep}(\hpe)^{-1},$$
property \eqref{lemmams} yields
\begin{eqnarray}
\nonumber &&|W_{el}(\nep \hye (\hpe)^{-1})-W_{el}(\nep y^{\ep}(P^{\ep})^{-1})|\\
\label{diffwel} &&
\leq C(1+W_{el}(\nep y^{\ep}(P^{\ep})^{-1}))(|\nabla f^{\ep}(y^{\ep})-Id|+|P^{\ep}(\hpe)^{-1}-Id|)
\end{eqnarray}
a.e. in $\Omega$. By \eqref{prk1} and \eqref{Pep1mrsh} there holds 
\begin{eqnarray*}
\|P^{\ep}(\hpe)^{-1}-Id\|_{L^{\infty}(\Omega;\mthree)}\leq c_k\|P^{\ep}-\hpe\|_{L^{\infty}(\Omega;\mthree)}\\
\leq c_k\|(Id-\exp(\ep^{\alpha-1}\tilde{p}))P^{\ep}\|_{L^{\infty}(\Omega;\mthree)}
\leq C\ep^{\alpha-1},
\end{eqnarray*}
hence, by combining \eqref{unifbdvarphi}, \eqref{measUeph} and \eqref{diffwel} we deduce
\begin{eqnarray}
\nonumber &&\frac{1}{\ep^{2\alpha-2}}\Big|\int_{\Omega\setminus U_{\ep}}{W_{el}(\nep \hye (\hpe)^{-1})}-\int_{\Omega\setminus U_{\ep}}{W_{el}(\nep y^{\ep}(P^{\ep})^{-1})}\Big|\\
\label{badsetestimate}&&\leq C\ep^{\alpha-1}(1+\ell_{\ep})\Big(1+\frac{1}{\ep^{2\alpha-2}}\intom{W_{el}(\nep y^{\ep}(P^{\ep})^{-1})}\Big),
\end{eqnarray}
which tends to zero owing to \eqref{engest2} and \eqref{lp1}.

On the other hand, on the sets $U_{\ep}$ we can use the estimate \eqref{quadrwel}. Hence, by \eqref{id+fep}, \eqref{id+hfep} and the quadratic structure of $Q$ we obtain 
\begin{eqnarray}
\nonumber &&\frac{1}{\ep^{2\alpha-2}}\int_{U_{\ep}}{W_{el}(\nep \hye (\hpe)^{-1})\,dx}-\frac{1}{\ep^{2\alpha-2}}\int_{U^{\ep}}{W_{el}(\nep y^{\ep}(P^{\ep})^{-1})\,dx}\\
 \nonumber && \leq\delta \intom{(|F^{\ep}|^2+|\hat{F}^{\ep}|^2)}+\intom{(Q(\hat{F}^{\ep})-Q(F^{\ep}))}\\
\label{cit1}&&=  \delta \intom{(|F^{\ep}|^2+|\hat{F}^{\ep}|^2)}+\frac{1}{2}\intom{\C(\hat{F}^{\ep}-F^{\ep}):(\hat{F}^{\ep}+F^{\ep})}.
\end{eqnarray} 
Now, by \eqref{wconvFeh} and \eqref{propfhep2} there holds
\begin{equation}
\label{cit2}
\hat{F}^{\ep}+F^{\ep}\deb \hat{F}_{\alpha}+G-p\quad\text{weakly in }L^2(\Omega;\mthree).
\end{equation}
Moreover,
\begin{equation}
\label{diffefhe}
\hat{F}^{\ep}-F^{\ep}\to \hat{F}_{\alpha}-G+p\quad\text{strongly in }L^2(\Omega;\mthree).
\end{equation}

Indeed, by \eqref{propfhep} and \eqref{propfhep2} it is enough to show that
$$\hat{G}^{\ep}-F^{\ep}\to p-\hat{p}\quad\text{strongly in }L^2(\Omega;\mthree).$$
By \eqref{deffep1} and \eqref{defhgep} we have
$$\hat{G}^{\ep}-F^{\ep}=(Id+\ep^{\alpha-1}G^{\ep})(\hat{w}^{\ep}-\hppe-w^{\ep}+p^{\ep}).$$
Now, by \eqref{diffhen}, \eqref{defwep} and \eqref{defhwep}, $\hat{w}^{\ep}-w^{\ep}=0$ in $\Omega\setminus S^{\ep}$, whereas in the sets $S^{\ep}$ we have
\begin{eqnarray*}
&&\hat{w}^{\ep}-w^{\ep}=\ep^{\alpha-1}(\hpe)^{-1}(\hppe)^2-\ep^{\alpha-1}(P^{\ep})^{-1}(p^{\ep})^2\\
&&=\ep^{\alpha-1}(P^{\ep})^{-1}\big(\exp(-\ep^{\alpha-1}\tilde{p})(\hppe)^2-(p^{\ep})^2\big)\\
&&=\ep^{\alpha-1}(P^{\ep})^{-1}(\exp(-\ep^{\alpha-1}\tilde{p})-Id)(\hppe)^2+\ep^{\alpha-1}(P^{\ep})^{-1}((\hppe)^2-(p^{\ep})^2).
\end{eqnarray*}
Therefore, by \eqref{Pep1mrsh}, \eqref{wconvpepmrsh}, \eqref{Pep1mrs} and \eqref{wconvpepmrs}, we deduce
\begin{eqnarray*}
&&\|\hat{w}^{\ep}-w^{\ep}\|_{L^2(\Omega;\mthree)}\leq C(\ep^{\alpha-1}+\|\hppe-p^{\ep}\|_{L^2(\Omega;\mthree)})\leq C,\\
&&\|\hat{w}^{\ep}-w^{\ep}\|_{L^1(\Omega;\mthree)}\leq C\ep^{\alpha-1},\\
&&\|\hat{w}^{\ep}-w^{\ep}\|_{L^{\infty}(\Omega;\mthree)}\leq C(1+\|\hppe-p^{\ep}\|_{L^{\infty}(\Omega;\mthree)}).
\end{eqnarray*}
Combining these estimates with \eqref{cptGep} and \eqref{sumdiffph} we obtain \eqref{diffefhe}.

Consider now the case $\alpha>3$. By \eqref{badsetestimate}--\eqref{diffefhe} we have 
\begin{eqnarray*}
&&\nonumber\limsup_{\ep\to 0} \Big\{\frac{1}{\ep^{2\alpha-2}}\intom{W_{el}(\nep \hye (\hpe)^{-1})}-\frac{1}{\ep^{2\alpha-2}}\intom{W_{el}(\nep y^{\ep} (P^{\ep})^{-1})}\Big\}\\
&&\leq\frac{1}{2}\intom{\C(\hat{F}_{\alpha}-G+p):(\hat{F}_{\alpha}+G-p)}+C\delta.
\end{eqnarray*}
Since $\delta$ is arbitrary, we deduce
\begin{eqnarray}
&&\nonumber\limsup_{\ep\to 0} \Big\{\frac{1}{\ep^{2\alpha-2}}\intom{W_{el}(\nep \hye (\hpe)^{-1})}-\frac{1}{\ep^{2\alpha-2}}\intom{W_{el}(\nep y^{\ep} (P^{\ep})^{-1})}\Big\}\\
&&\nonumber \leq\frac{1}{2}\intom{\C(\hat{F}_{\alpha}-G+p):(\hat{F}_{\alpha}+G-p)}\\
&&\label{qtrick}=\intom{Q(\hat{F}_{\alpha})}-\intom{Q(G-p)}\leq \intom{Q(\hat{F}_{\alpha})}-\intom{Q_2(G'-p')}.
\end{eqnarray}
By \eqref{deffa} and \eqref{propfhep2}, up to an approximation argument, we may assume that $d$ is such that
$$Q(\hat{F}_{\alpha})=Q_2(\sym\,\nabla' \hat{u}-x_3(\nabla')^2\hat v-\hat{p}').$$
This, together with \eqref{qtrick}, implies \eqref{limsupeng}. 

In the case $\alpha=3$ a preliminary approximation argument is needed. {Let $(\tilde{u}^k)$ be a sequence in $C^{\infty}_c(\omega\cup\gamma_n;\R^2)$, such that
$$\tilde{u}^k\to \tilde{u}+v\nabla' \tilde{v}\quad\text{strongly in }W^{1,2}(\omega;\R^2)$$
(such a sequence exists by Lemma \ref{bdc} because $\tilde{u}\in C^{\infty}_c(\omega\cup\gamma_n;\R^2)$ and $\tilde{v}\in C^{\infty}_c(\omega\cup\gamma_n)$).
Let also ${v}^k\in C^{\infty}_c(\omega)$ be such that 
$${v}^k\to {v}\quad\text{strongly in }L^2(\omega)$$
and set}
 $${d}^k(x):=d(x',x_3-{v}^k(x'))\quad\text{for a.e. }x\in\Omega.$$
 Since $d\in C^{\infty}_c(\Omega;\R^3)$, there exists an open set $O\subset \R^2$ such that $\overline{O}\subset \omega$ and $d^k(x',x_3)=0$ for every $x\in (\omega\setminus \overline{O})\times \R$. Moreover, $d^k(x',x_3)=0$ for every $x\in\R^3$ such that $|x_3|>\frac{1}{2}+\|{v}^k\|_{L^{\infty}(\R^2)}$. Hence, $d^k\in C^{\infty}(\R^3;\R^3)$ and $$\supp\, {d^k}\subset \overline{O}\times \big(-\tfrac12-\|{v}^k\|_{L^{\infty}(\R^2)}, \tfrac 12+\|{v}^k\|_{L^{\infty}(\R^2)}\big).$$ 
 It is easy to see that \eqref{deffa3}, \eqref{propfhep2} and \eqref{badsetestimate}--\eqref{diffefhe} can still be deduced, and for every $k$ we can construct a sequence $(\hye_k, \hpe_k)$ that satisfies \eqref{cvidh}--\eqref{wconvpepmrsh} with $\hat{u}$ replaced by $u+\tilde{u}^k$, and
\begin{eqnarray*}
&&\nonumber\limsup_{\ep\to 0} \Big\{\frac{1}{\ep^{2\alpha-2}}\intom{W_{el}(\nep \hye_k (\hpe_k)^{-1})}-\frac{1}{\ep^{2\alpha-2}}\intom{W_{el}(\nep y^{\ep} (P^{\ep})^{-1})}\Big\}\\
&&\leq\frac{1}{2}\intom{\C(\hat{F}^k-G+p):(\hat{F}^k+G-p)},
\end{eqnarray*}
where
$$\hat{F}^k:=\sym\Big(\begin{array}{c}\nabla' \tilde{u}^k-(x_3+v)(\nabla')^2 \tilde{v}+\frac{\nabla' \tilde{v}\otimes\nabla' \tilde{v}}{2}\\0\end{array}\Big|\begin{array}{c}d'(x', x_3+{v}-{v}^k)\\d_3(x',x_3+{v}-{v}^k)+\frac{1}{2}|\nabla' \tilde{v}|^2\end{array}\Big)+G-\hat{p}.$$
 On the other hand,
$$\hat{F}^k\to \sym\Big(\begin{array}{c}\nabla' \tilde{u}-x_3(\nabla')^2 \tilde{v}+\nabla' {v}\otimes\nabla' \tilde{v}+\frac{\nabla' \tilde{v}\otimes\nabla' \tilde{v}}{2}\\0\end{array}\Big|\begin{array}{c}d'\\d_3+\frac{1}{2}|\nabla' \tilde{v}|^2\end{array}\Big)+G-\hat{p}=:\hat{F}$$
strongly in $L^2(\Omega;\mthree)$, as $k\to +\infty$. A diagonal argument leads then to the estimate
\begin{eqnarray}
&&\nonumber\limsup_{\ep\to 0} \frac{1}{\ep^{2\alpha-2}}\Big(\intom{W_{el}(\nep \hye (\hpe)^{-1})}-\intom{W_{el}(\nep y^{\ep} (P^{\ep})^{-1})}\Big)\\
&&\label{aend}\leq \frac{1}{2}\intom{\C(\hat{F}-G+p):(\hat{F}+G-p)}.
\end{eqnarray}
Up to a further approximation, we may assume that $d$ is such that
$$Q(\hat{F})=Q_2\Big(\sym\,\nabla' \hat{u}-x_3(\nabla')^2 \hat{v}+\frac{1}{2}\nabla' \hat{v}\otimes\nabla' \hat{v}-\hat{p}'\Big),$$
hence \eqref{limsupeng} follows by \eqref{aend}. 
\end{proof}
\section{Convergence of quasistatic evolutions}
\label{pquas}
The first part of this section is devoted to the proof of Theorem \ref{cvstress}. We first prove the theorem for $\alpha>3$ and then we show how the proof must be modified for $\alpha=3$.
\begin{proof}[Proof of Theorem \ref{cvstress} in the case $\alpha>3$]
The proof is divided into five steps.\\
{\em{Step 0: A priori estimates on the elasto-plastic energy}}\\
Set $y^{\ep}(t):=\pep(t,\zep(t))$ for every $t\in [0,T]$. Arguing as in the proof of \eqref{rsze}, it is immediate to see that 
{\begin{equation}
\label{bdtep}
y^{\ep}(t,x)=\pep(t,(x',\ep x_3))\quad\cal{H}^2\text{ - a.e. on }\Gamma_d.
\end{equation}} 
In this step we shall show that there exists a constant $C$ such that for every $t\in [0,T]$ and every $\ep$ there holds
\begin{equation}
\label{unifapest}
\frac{1}{\ep^{\alpha-1}}\|\dist(\nep y^{\ep}(t)(P^{\ep})^{-1}(t),SO(3))\|_{L^2(\Omega)}+\|p^{\ep}(t)\|_{L^2(\Omega;\mthree)}+\|\ep^{\alpha-1}p^{\ep}(t)\|_{L^{\infty}(\Omega;\mthree)}\leq C.
\end{equation}

To this purpose, we first remark that since $t\mapsto (\zep(t), P^{\ep}(t))$ is an $\ep$-quasistatic evolution, then 
\begin{equation}
\label{isink}
P^{\ep}(t,x)\in K\quad\text{for a.e. }x\in\Omega,\text{ for every }\ep\text{ and }t,
\end{equation}
  hence $\ep^{\alpha-1}p^{\ep}(t)\in K-Id$ for every $\ep$ and $t$ and by \eqref{prk1} there exists a constant $C$ such that
\begin{equation}
\label{unifbdpep}
\|\ep^{\alpha-1}p^{\ep}(t)\|_{L^{\infty}(\Omega;\mthree)}\leq C\quad\text{for every }\ep\text{ and }t.
\end{equation}

By the minimality condition (qs1), taking $\tilde{z}(x)=(x',\ep x_3)$ and $\tilde{P}(x)= Id$ for every $x\in \Omega$, and observing that $W_{hard}(Id)=0$ a.e. in $\Omega$ by \eqref{prh4}, we deduce
\begin{eqnarray}
\label{estiniz} \frac{1}{\ep^{2\alpha-2}}\cal{F}_{\ep}(t,\zep(t),P^{\ep}(t))\leq \frac{1}{\ep^{2\alpha-2}}\intom{W_{el}\big(\nabla \pep (t, (x',\ep x_3))\big)}
+\frac{1}{\ep^{\alpha-1}}\intom{D(P^{\ep}(t),Id)},
\end{eqnarray}
for every $t\in [0,T]$ and for all $\ep$. By \eqref{prd2} and \eqref{isink}, there holds 
$$D(P^{\ep}(t),Id)=D(Id, (P^{\ep})^{-1}(t))\leq c_7|(P^{\ep})^{-1}(t)-Id|\leq c_7c_K|Id-P^{\ep}(t)|,$$
where the last inequality follows by \eqref{prk1}. Hence, Holder inequality yields
\begin{equation}
\label{estiniz2}
\frac{1}{\ep^{\alpha-1}}\intom{D(P^{\ep}(t),Id)}\leq\frac{C}{\ep^{\alpha-1}}\|Id-P^{\ep}(t)\|_{L^2(\Omega;\mthree)}.
\end{equation}
On the other hand, by frame indifference (H3) of $W_{el}$ we obtain
\begin{eqnarray*}
{W_{el}\big(\nabla \pep(t, (x',\ep x_3))\big)}={W_{el}\Big(\sqrt{(\nabla \pep)^T (t, (x',\ep x_3))\nabla \pep (t,(x',\ep x_3)}\Big)}
\end{eqnarray*}
for every $x\in\Omega$ and for all $t\in [0,T]$. By \eqref{treq1}, \eqref{lp2} and \eqref{gradpep} there holds
\begin{eqnarray*}
\nonumber&&\nabla \pep(t, (x',\ep x_3))=Id+\ep^{\alpha-1}\Big(\begin{array}{cc}\nabla' {u}^0(t,x')-x_3(\nabla')^2 {v}^0(t,x')&0\\0&0\end{array}\Big)\\
&&+\ep^{\alpha-2}\Big(\begin{array}{cc}0&-\nabla' {v}^0(t,x')\\(\nabla' {v}^0(t,x'))^T&0\end{array}\Big),
\end{eqnarray*}
for every $x\in\Omega$. Since $\alpha>3$, we deduce
\begin{eqnarray*}
&&(\nabla \pep)^T (t,(x',\ep x_3))\nabla \pep (t, (x',\ep x_3))\\
&&=Id+2\ep^{\alpha-1}\sym\Big(\begin{array}{cc}\nabla' {u}^0(t,x')-x_3(\nabla')^2 {v}^0(t,x')&0\\0&0\end{array}\Big)+o(\ep^{\alpha-1}),
\end{eqnarray*}
and
\begin{equation}
\label{sqrtdeco}
\sqrt{(\nabla \pep)^T (t,(x',\ep x_3))\nabla \pep (t, (x',\ep x_3))}=Id+\ep^{\alpha-1}M(t,x)+o(\ep^{\alpha-1}),
\end{equation}
where
$$M(t,x)=\sym\Big(\begin{array}{cc}\nabla' {u}^0(t,x')-x_3(\nabla')^2 {v}^0(t,x')&0\\0&0\end{array}\Big)\quad\text{for every }x\in\Omega.$$
Therefore,
$$\frac{1}{\ep^{2\alpha-2}}W_{el}\big(\nabla \pep(t, (x',\ep x_3))\big)=\frac{1}{\ep^{2\alpha-2}}W_{el}\big(Id+\ep^{\alpha-1}M(t,x)+o(\ep^{\alpha-1})\big)$$
for every $x\in \Omega$. Now, by the smoothness of $u^0$ and $v^0$, there exists a constant $C$ such that 
\begin{equation}
\label{unifquad}
\sup_{t\in [0,T]}\|M(t)\|_{L^{\infty}(\Omega;\mthree)}\leq C
\end{equation} and there exist $\overline{\ep}$ such that, if $\ep<\overline{\ep}$, for every $t\in [0,T]$ 
$$|\ep^{\alpha-1}M(t)+o(\ep^{\alpha-1})|\leq c_{el}(1),$$
where $c_{el}$ is the constant in \eqref{quadrwel}.
 Therefore, by \eqref{quadrwel}, \eqref{growthcondQ}, and \eqref{unifquad} we have  
\begin{equation}
\label{ubnp}\frac{1}{\ep^{2\alpha-2}}\intom{W_{el}\big(\nabla \pep(t, (x',\ep x_3))\big)}\leq C\Big(\intom{|M(t)|^2}+1\Big)\leq C
\end{equation}
for every $\ep$ and for all $t\in [0,T]$.

By combining \eqref{estiniz}, \eqref{estiniz2} and \eqref{ubnp} we obtain
\begin{eqnarray}
\nonumber &&\frac{1}{\ep^{2\alpha-2}}\intom{W_{el}\big(\nep y^{\ep}(t)(P^{\ep})^{-1}(t)\big)}+\frac{1}{\ep^{2\alpha-2}}\intom{W_{hard}(P^{\ep}(t))}\\
\nonumber &&\leq C\Big(1+\frac{1}{\ep^{\alpha-1}}\|Id-P^{\ep}(t)\|_{L^2(\Omega;\mthree)}\Big).\\
\label{estiniz3}
\end{eqnarray}
Now, by \eqref{prh3} there holds 
$$\frac{c_6}{\ep^{2\alpha-2}}\intom{|Id-P^{\ep}(t)|^2}\leq C\Big(1+\frac{1}{\ep^{\alpha-1}}\|Id-P^{\ep}(t)\|_{L^2(\Omega;\mthree)}\Big),$$
which in turn, by Cauchy inequality implies
\begin{equation}
\label{part1ubd}
\|p^{\ep}(t)\|_{L^2(\Omega;\mthree)}=\frac{1}{\ep^{\alpha-1}}\|Id-P^{\ep}(t)\|_{L^2(\Omega;\mthree)}\leq C
\end{equation}
for every $\ep$ and for all $t\in [0,T]$. On the other hand, by \eqref{estiniz3} and \eqref{part1ubd}, we deduce
\begin{equation}
\label{unelt}
\frac{1}{\ep^{2\alpha-2}}\intom{W_{el}\big(\nep y^{\ep}(t)(P^{\ep})^{-1}(t)\big)}\leq C,
\end{equation}
for every $\ep$ and for all $t\in [0,T]$. Estimate \eqref{unifapest} follows now by \eqref{unifbdpep}, \eqref{part1ubd}, \eqref{unelt} and the growth condition (H4).\\
{\em{Step 1: A priori estimate on the dissipation functional}}.\\
In this step we shall show that there exists a constant $C$, such that
\begin{equation}
\label{bunifapest}
\frac{1}{\ep^{\alpha-1}}\cal{D}(P^{\ep};0,t)\leq C\quad\text{for every }\ep\text{ and for all }t\in [0,T].
\end{equation}
By (qs2), \eqref{convE0} and \eqref{estiniz3}--\eqref{unelt} it is enough to show that there exists a constant $C$ such that
\begin{equation}
\label{unifbdeep}
\Big|\frac{1}{\ep^{\alpha-1}}{\intom{E^{\ep}(t):\nabla \dot{\pep}(t,\zep(t))(\nabla \pep)^{-1}(t,\zep(t))}}\Big|\leq C
\end{equation}
for every $\ep$ and $t\in [0,T]$. To prove \eqref{unifbdeep}, we first deduce some properties of the map $t\mapsto E^{\ep}(t)$. 

Let $R\in SO(3)$. By \eqref{prk1} and \eqref{isink} there holds
\begin{eqnarray*}
&&|\nep y^{\ep}(t)-R|^2=|\nep y^{\ep}(t)-RP^{\ep}(t)+\ep^{\alpha-1}Rp^{\ep}(t)|^2\\
&&\leq 2|\nep y^{\ep}(t)(P^{\ep})^{-1}(t)-R|^2|P^{\ep}(t)|^2+2\ep^{2\alpha-2}|p^{\ep}(t)|^2\\
&&\leq 2\,c_K^2|\nep y^{\ep}(t)(P^{\ep})^{-1}(t)-R|^2+2\ep^{2\alpha-2}|p^{\ep}(t)|^2.
\end{eqnarray*}
Hence, the growth condition (H4) implies
$$\|\dist(\nep y^{\ep}(t),SO(3))\|^2_{L^2(\Omega;\mthree)}\leq C\Big(\intom{W_{el}(\nep y^{\ep}(t)(P^{\ep})^{-1}(t))}+\ep^{2\alpha-2}\|p^{\ep}(t)\|^2_{L^2(\Omega;\mthree)}\Big),$$
which in turn yields
\begin{equation}
\label{unifdist}
\|\dist(\nep y^{\ep}(t),SO(3))\|^2_{L^2(\Omega;\mthree)}\leq C\ep^{2\alpha-2}
\end{equation}
by \eqref{unifapest} and \eqref{unelt}.
By \eqref{bdtep} and \eqref{unifdist}, the sequence $y^{\ep}(t)$ fulfills the hypotheses of Theorem \ref{compactbd1}. Hence, for every $t\in [0,T]$ there exists a sequence of maps $(R^{\ep}(t))\subset W^{1,\infty}(\omega;\mthree)$ such that 
\begin{eqnarray}
&&\label{rt1t} R^{\ep}(t,x')\in SO(3)\quad\text{for every }x'\in \omega,\\
&&\label{rt2t} \|\nep y^{\ep}(t)-R^{\ep}(t)\|_{L^2(\Omega;\mthree)}\leq C\ep^{\alpha-1},\\
&&\label{rt3t} \|\partial_i R^{\ep}(t)\|_{L^2(\omega;\mthree)}\leq C\ep^{\alpha-2},\quad i=1,2,\\
&&\label{rt4t} \|R^{\ep}(t)-Id\|_{L^2(\omega;\mthree)}\leq C\ep^{\alpha-2},
\end{eqnarray} 
where the constant $C$ is independent of $\ep$ and $t$. 

We consider the auxiliary maps
$$w^{\ep}(t):=\frac{(Id+\ep^{\alpha-1}p^{\ep}(t))^{-1}-Id+\ep^{\alpha-1}p^{\ep}(t)}{\ep^{\alpha-1}},$$
the elastic strains
$$G^{\ep}(t):=\frac{(R^{\ep}(t))^T\nep y^{\ep}(t)-Id}{\ep^{\alpha-1}},$$
and the matrices
\begin{equation}
\label{deffept}
F^{\ep}(t):=G^{\ep}(t)+w^{\ep}(t)-p^{\ep}(t)+\ep^{\alpha-1}G^{\ep}(t)(w^{\ep}(t)-p^{\ep}(t)),
\end{equation}
for all $t\in [0,T]$. Clearly we have
\begin{equation}
\label{decompep}
(P^{\ep})^{-1}(t)=Id+\ep^{\alpha-1}(w^{\ep}(t)-p^{\ep}(t))\quad\text{and}\quad \nep y^{\ep}(t)=R^{\ep}(t)(Id+\ep^{\alpha-1}G^{\ep}(t)).
\end{equation}
Since 
\begin{equation}
\label{numerarepag35}
w^{\ep}(t)=\ep^{\alpha-1}(Id+\ep^{\alpha-1}p^{\ep}(t))^{-1}(p^{\ep}(t))^2
\end{equation}
for every $t\in [0,T]$, by \eqref{unifapest} and \eqref{isink} there holds
\begin{equation}
\label{linftybdw}
\|\ep^{\alpha-1}w^{\ep}(t)\|_{L^{\infty}(\Omega;\mthree)}\leq C\quad\text{for every }t\in [0,T],
\end{equation}
\begin{equation}
\label{numerarepag352}
\|w^{\ep}(t)\|_{L^1(\Omega;\mthree)}\leq C\ep^{\alpha-1}\quad\text{for every }t\in [0,T],
\end{equation}
and
\begin{equation}
\label{l2bdw}
\|w^{\ep}(t)\|_{L^2(\Omega;\mthree)}\leq C\quad\text{for every }t\in [0,T].
\end{equation}
Combining \eqref{numerarepag352} and \eqref{l2bdw} we deduce
\begin{equation}
\label{weakl2w}
w^{\ep}(t)\deb 0\quad\text{weakly in }L^2(\Omega;\mthree)\text{ for every }t\in [0,T].
\end{equation}
On the other hand, \eqref{rt1t} and \eqref{rt2t} yield
\begin{equation}
\label{unifbdrt2}
\|G^{\ep}(t)\|_{L^2(\Omega;\mthree)}\leq C
\end{equation}
for every $\ep$ and for all $t\in [0,T]$.
Collecting \eqref{unifapest}, \eqref{linftybdw}, \eqref{l2bdw} and \eqref{unifbdrt2}, we obtain  
\begin{eqnarray}
\nonumber&&\|F^{\ep}(t)\|_{L^2(\Omega;\mthree)}\leq \|G^{\ep}(t)\|_{L^2(\Omega;\mthree)}+\|w^{\ep}(t)\|_{L^2(\Omega;\mthree)}+\|p^{\ep}(t)\|_{L^2(\Omega;\mthree)}\\
\label{unifbdfep}&&+\|G^{\ep}(t)\|_{L^2(\Omega;\mthree)}\|\ep^{\alpha-1}(w^{\ep}(t)-p^{\ep}(t))\|_{L^{\infty}(\Omega;\mthree)}\leq C
\end{eqnarray}
for every $\ep$ and for all $t\in [0,T]$.

Now, by \eqref{deffept}, \eqref{decompep} and the frame-indifference (H3) of $W_{el}$ we deduce the decomposition
\begin{equation}
\label{stressrot}
\eep(t)=R^{\ep}(t)\eepp(t)(R^{\ep}(t))^T
\end{equation}
for every $t\in [0,T]$, where
\begin{equation}
\nonumber
\eepp(t):=\frac{1}{\ep^{\alpha-1}}DW_{el}(Id+\ep^{\alpha-1}F^{\ep}(t))(Id+\ep^{\alpha-1}F^{\ep}(t))^T.
\end{equation}
We argue as in \cite[Proof of Theorem 3.1, Steps 2--3]{M-S} and we first show that there exist two positive constants $k_1, k_2$, independent of $\ep$, such that
\begin{equation}
\label{claimunifinteg}
|\eepp(t)|\leq k_1\Big(\frac{W_{el}(Id+\ep^{\alpha-1}F^{\ep}(t))}{\ep^{\alpha-1}}+k_2|F^{\ep}(t)|\Big)
\end{equation}
for every $t\in [0,T]$ and for a.e. $x\in \Omega$. 

Indeed, let $c_{el_2}$ be the constant in \eqref{locquad}.
Suppose that $\ep^{\alpha-1}|F^{\ep}(t)|\geq c_{el_2}$. We remark that (H1), \eqref{isink} and \eqref{unelt} imply in particular that 
$$\det (\nep y^{\ep}(t))>0\quad\text{a.e. in }\Omega.$$
Therefore, by \eqref{mandel2} there holds
\begin{equation}
\label{toolclaimstress1}
|\eepp(t)|\leq \frac{c_3}{\ep^{\alpha-1}}\Big(W_{el}(Id+\ep^{\alpha-1}F^{\ep}(t))+1\Big)\leq {c_3}\Big(\frac{W_{el}(Id+\ep^{\alpha-1}F^{\ep}(t))}{\ep^{\alpha-1}}+\frac{1}{c_{el_2}}|F^{\ep}(t)|\Big).
\end{equation}
Consider now the case where $\ep^{\alpha-1}|F^{\ep}(t)|< c_{el_2}$. Then, by \eqref{locquad} there holds
$$DW_{el}(Id+\ep^{\alpha-1}F^{\ep}(t))\leq\ep^{\alpha-1}(2\bc+1)|F^{\ep}(t)|,$$
which in turn implies
\begin{equation}
\label{toolclaimstress2}
|\eepp(t)|\leq C|F^{\ep}(t)|(|Id|+|\ep^{\alpha-1}F^{\ep}(t)|)\leq C|F^{\ep}(t)|.
\end{equation}
Collecting \eqref{toolclaimstress1} and \eqref{toolclaimstress2}, we obtain \eqref{claimunifinteg}.

By \eqref{unelt}, \eqref{unifbdfep} and \eqref{claimunifinteg}, for every measurable $\Lambda\subset \Omega$, the following estimate holds true:
\begin{equation}
\label{unifinteg}
\int_{\Lambda}{|\eepp(t)|\,dx}\leq k_1\int_{\Lambda}{\Big(\frac{W_{el}(Id+\ep^{\alpha-1}F^{\ep}(t))}{\ep^{\alpha-1}}+k_2|F^{\ep}(t)|\Big)}\leq C(|\Lambda|^{\frac{1}{2}}+\ep^{\alpha-1}),
\end{equation}
for every $\ep$ and for all $t\in [0,T]$. By \eqref{rt1t} there holds also
\begin{equation}
\label{unifintegnt}
\int_{\Lambda}{|E^{\ep}(t)|\,dx}\leq C(|\Lambda|^{\frac{1}{2}}+\ep^{\alpha-1}),
\end{equation}
for every $\ep$ and for all $t\in [0,T]$.

Let now $\gamma\in (0,\alpha-2)$ be the positive constant in the definition of the maps $\tep$. Let $O_{\ep}(t)$ be the set given by
$$O_{\ep}(t):=\{x\in \Omega:\,{\ep}^{\alpha-1-\gamma}|F^{\ep}(t,x)|\leq c_{el_2}\},$$
and let $\chi_{\ep}(t):\Omega\to \{0,1\}$ be the map 
$$\chi_{\ep}(t,x)=\begin{cases}1&\text{if }x\in O_{\ep}(t),\\
 0&\text{otherwise}.
 \end{cases}$$ By Chebychev inequality and \eqref{unifbdfep} we deduce \begin{equation}
\label{measbep}
\cal{L}^3(\Omega \setminus O_{\ep}(t))\leq C {\ep}^{2(\alpha-1-\gamma)},
\end{equation}
for every $\ep$ and for all $t\in [0,T]$. By combining \eqref{unifinteg} and \eqref{measbep} we conclude that
\begin{equation}
\label{badsetstress}
\|(1-\chi_{\ep}(t)){\tilde{E}}^{\ep}(t)\|_{L^1(\Omega;\mthree)}\leq C\ep^{\alpha-1-\gamma}\quad\text{for every }t\in [0,T].
\end{equation}
By \eqref{unifintegnt} the previous estimate implies also
\begin{equation}
\label{badsetstressnt}
\|(1-\chi_{\ep}(t)){{E}}^{\ep}(t)\|_{L^1(\Omega;\mthree)}\leq C\ep^{\alpha-1-\gamma}\quad\text{for every }t\in [0,T].
\end{equation}
On the other hand \eqref{locquad} yields  the following estimate on the sets $O_{\ep}(t)$:
\begin{eqnarray*}
|\chi_{\ep}(t){\tilde{E}}^{\ep}(t)|\leq (2\bc+1)|F^{\ep}(t)||Id+\ep^{\alpha-1}F^{\ep}(t)|\leq C(1+c_{el_2}\ep^{\gamma})|F^{\ep}(t)|,
\end{eqnarray*}
which in turn, by \eqref{unifbdfep}, implies
\begin{equation}
\label{goodsetsufb}
\|\chi_{\ep}(t){\tilde{E}}^{\ep}(t)\|_{L^2(\Omega;\mthree)}\leq C
\end{equation}
for every $\ep$ and for all $t\in [0,T]$.

By \eqref{lp1}, \eqref{rt1t}, \eqref{badsetstressnt} and \eqref{goodsetsufb}, and since $E^{\ep}(t)$ is symmetric by Remark \ref{Esym}, to prove \eqref{unifbdeep} it is enough to show that there exists a constant $C$ such that
  \begin{equation}
 \label{linftybdwk}
 \Big\|\frac{1}{\ep^{\alpha-1}}\nabla \dot{\pep}(t,\zep(t))(\nabla \pep)^{-1}(t,\zep(t))\Big\|_{L^{\infty}(\Omega;\mthree)}\leq C\ell_{\ep}
 \end{equation}
 and
 \begin{equation}
 \label{l2bdwk}
 \Big\|\frac{1}{\ep^{\alpha-1}}\sym\Big(\nabla \dot{\pep}(t,\zep(t))(\nabla \pep)^{-1}(t,\zep(t))\Big)\Big\|_{L^{2}(\Omega;\mthree)}\leq C
 \end{equation}
for every $\ep$ and for all $t\in [0,T]$. By \eqref{gradpep}, there holds
\begin{eqnarray}
\nonumber&&\frac{1}{\ep^{\alpha-1}}\nabla \dot{\pep}(t,\zep(t))=\Big(\begin{array}{cc}\nabla' \dot{u}^0(t,\bzep(t))-\tep\big(\frac{\zep_3(t)}{\ep}\big)(\nabla')^2 \dot{v}^0(t,\bzep(t))&0\\0&0\end{array}\Big)\\
\label{gradpept}&&+\frac{1}{\ep}\Big(\begin{array}{cc}0&-\dot{\tep}\big(\frac{\zep_3(t)}{\ep}\big)\nabla' \dot{v}^0(t,\bzep(t))\\(\nabla' \dot{v}^0(t,\bzep(t)))^T&0\end{array}\Big).
\end{eqnarray}
Estimate \eqref{linftybdwk} follows directly by \eqref{treq3}, \eqref{treq5}, \eqref{lp2}, and \eqref{bddinverse}. To prove \eqref{l2bdwk}, we first provide an estimate for the $L^2$ norm of the maps $\frac{1}{\ep}\zep_3(t)$. To this purpose, let $v^{\ep}(t)$ be defined as in \eqref{cvvt}. It is easy to see that
$$v^{\ep}(t)=\frac{1}{\ep^{\alpha-2}}\intt{y^{\ep}_3(t)}\quad\text{and}\quad\nabla' v^{\ep}(t)=\frac{1}{\ep^{\alpha-2}}\intt{{\nabla'y^{\ep}_3(t)}}$$
for every $\ep$ and for all $t\in [0,T]$.
By \eqref{bdtep}, arguing as in the proof of Theorem \ref{compactbd1},
$$v^{\ep}(t)=v^0(t)\quad\cal{H}^1\text{ - a.e. on }\gamma_d.$$
 By \eqref{rt2t} and \eqref{rt4t}, we have
$$\|\nabla' v^{\ep}(t)\|_{L^2(\omega;\R^2)}\leq C$$
for every $\ep$ and $t\in [0,T]$. Hence, by Poincar\'e inequality we deduce
$$\|v^{\ep}(t)-v^0(t)\|_{L^2(\omega)}\leq C\|\nabla' v^{\ep}(t)-\nabla' v^0(t)\|_{L^2(\omega;\R^2)}\leq C,$$
which in turn, by the smoothness of $v^{0}$, yields
$$\|v^{\ep}(t)\|_{L^2(\omega)}\leq C\quad\text{for every }\ep\text{ and for all }t\in [0,T].$$
By \eqref{rt2t}, \eqref{rt4t} and Poincar\'e-Wirtinger inequality, we deduce
\begin{equation}
\label{estimate3}
\Big\|\frac{y^{\ep}_3(t)}{\ep}-x_3-\ep^{\alpha-3}v^{\ep}(t)\Big\|_{L^2(\Omega)}\leq C\Big\|\frac{\partial_3 y^{\ep}_3(t)}{\ep}-1\Big\|_{L^2(\Omega)}\leq C\ep^{\alpha-2}
\end{equation}
for every $t\in [0,T]$, which implies
\begin{equation}
\label{uf3}
 \Big\|\frac{y^{\ep}_3(t)}{\ep}\Big\|_{L^2(\Omega)}\leq C\quad\text{for every }\ep\text{ and }t\in [0,T].
\end{equation}
On the other hand,
\begin{equation}
\label{relzep}
z^{\ep}(t)=\vep(t,y^{\ep}(t))\quad\text{a.e. in }\Omega,
\end{equation}
hence by \eqref{3comp},
\begin{equation}
\label{zep3}
\frac{\zep_3(t)}{\ep}=\frac{y^{\ep}_3(t)}{\ep}-\ep^{\alpha-3}v^0(t,({\varphi}^{\ep})'(t,y^{\ep}(t))).
\end{equation}
Therefore \eqref{treq2} and \eqref{uf3} yield
\begin{equation}
\label{l2usf}
\Big\|\tep\Big(\frac{\zep_3(t)}{\ep}\Big)\Big\|_{L^2(\Omega)}\leq \Big\|\frac{\zep_3(t)}{\ep}\Big\|_{L^2(\Omega)}\leq C \quad\text{for every }\ep\text{ and }t\in [0,T].
\end{equation}
By Lemma \ref{cvproptep}, we deduce 
\begin{equation}
\label{l2usf2}
\Big\|1-\dot{\tep}\Big(\frac{\zep_3(t)}{\ep}\Big)\Big\|_{L^2(\Omega)}\leq \frac{2}{\ell_{\ep}}\quad\text{for every }\ep\text{ and }t\in [0,T].
\end{equation}
Collecting \eqref{lp2}, \eqref{gradpept}, \eqref{l2usf} and \eqref{l2usf2}, we deduce that there exists a constant $C$ such that
$$\Big\|\frac{1}{\ep^{\alpha-1}}\sym \nabla \dot{\pep}(t,\zep(t))\Big\|_{L^2(\Omega;\mthree)}\leq C$$
for every $\ep$ and for all $t\in [0,T]$. Therefore, to prove \eqref{l2bdwk}, it remains only to study the quantity
$$\frac{1}{\ep^{\alpha-1}}\sym\Big(\nabla \dot{\pep}(t,\zep(t))\big((\nabla \pep)^{-1}(t,\zep(t))-Id\big)\Big).$$
By \eqref{bddinverse}, 
$$\|(\nabla \pep(t))^{-1}-Id\|_{L^{\infty}(\Omega;\mthree)}\leq C\quad\text{for every }\ep\text{ and }t\in [0,T].$$
By \eqref{l2usf}, the first term in the right hand side of \eqref{gradpept} is uniformly bounded in $L^2(\Omega;\mthree)$. Therefore, it remains to show that
\begin{equation}
\label{newstar}
\frac{1}{\ep}\Big(\begin{array}{cc}0&-\dot{\tep}\big(\frac{\zep_3(t)}{\ep}\big)\nabla' \dot{v}^0(t,\bzep(t))\\(\nabla' \dot{v}^0(t,\bzep(t)))^T&0\end{array}\Big)\Big((\nabla \pep)^{-1}(t,\zep(t))-Id\Big)
\end{equation}
is uniformly bounded in $L^2(\Omega;\mthree)$.

By \eqref{treq5} and by the smoothness of $v^0$, there holds
\begin{equation}
\label{estt1}
\Big\|\frac{1}{\ep}\Big(\begin{array}{cc}0&-\dot{\tep}\big(\frac{\zep_3(t)}{\ep}\big)\nabla' \dot{v}^0(t,\bzep(t))\\(\nabla' \dot{v}^0(t,\bzep(t)))^T&0\end{array}\Big)\Big\|_{L^{\infty}(\Omega;\mthree)}\leq \frac{C}{\ep}
\end{equation}
for every $t\in [0,T]$.
On the other hand,
\begin{equation}
\label{estt2bis}
(\nabla \pep)^{-1}(t,\zep(t))=\nabla \vep(t,y^{\ep}(t))\quad\text{a.e. in }\Omega.
\end{equation}
Property \eqref{gradinv1} yields the estimate
\begin{equation}
\label{estt3}
\|\nabla \vep_3(t,y^{\ep}(t))-e_3\|_{L^{\infty}(\Omega;\R^3)}\leq C\ep^{\alpha-2}
\end{equation}
for every $t\in[0,T]$, whereas by \eqref{treq5}, \eqref{expgr} and \eqref{bddinverse} 
$$\|\nabla \vep_i(t,y^{\ep}(t))-e_i\|_{L^{2}(\Omega;\R^3)}\leq C\ep^{\alpha-1}\Big\|\tep\Big(\frac{\vep_3(t,y^{\ep}(t))}{\ep}\Big)\Big\|_{L^2(\omega)}+C\ep^{\alpha-2},$$
hence by \eqref{relzep} and \eqref{l2usf} we obtain
\begin{equation}
\label{estt4}
\|\nabla \vep_i(t,y^{\ep}(t))-e_i\|_{L^{2}(\Omega;\R^3)}\leq C\ep^{\alpha-2}.
\end{equation}
By combining \eqref{estt1}--\eqref{estt4}, we deduce
\begin{eqnarray}
\nonumber
&&\Big\|\frac{1}{\ep}\Big(\begin{array}{cc}0&-\dot{\tep}\big(\frac{\zep_3(t)}{\ep}\big)\nabla' \dot{v}^0(t,\bzep(t))\\(\nabla' \dot{v}^0(t,\bzep(t)))^T&0\end{array}\Big)\Big((\nabla \pep)^{-1}(t,\zep(t))-Id\Big)\Big\|_{L^2(\Omega;\mthree)}\\
\label{utiledopo}&&\leq C\ep^{\alpha-3}
\end{eqnarray}
for every $\ep$ and $t\in [0,T]$, therefore the quantity in \eqref{newstar} is uniformly bounded in $L^2(\Omega;\M^{3\times 3})$, and the proof of \eqref{l2bdwk} is complete.
By \eqref{badsetstress}--\eqref{l2bdwk}, since all estimates are uniform both in $\ep$ and $t$, we deduce \eqref{unifbdeep}, which in turn yields \eqref{bunifapest}.\\
{\em{Step 2: Reduced Stability}}\\
Owing to the a priori bounds \eqref{unifapest} and \eqref{bunifapest}, we can apply the generalized version of Helly's Selection Principle in \cite[Theorem A.1]{MRS}. To show it, take $\cal{Z}:=L^2(\Omega;\mthree)$ endowed with the weak topology of $L^2$, and set
$$\cal{D}_{\ep}(z_1,z_2):=\frac{1}{\ep^{\alpha-1}}\intom{D(Id+\ep^{\alpha-1}z_1, Id+\ep^{\alpha-1}z_2)}$$
and
$$\cal{D}_{\infty}(z_1,z_2):=\intom{H(z_2-z_1)}$$
for every $z_1,z_2\in L^2(\Omega;\mthree)$. Hypotheses (A.1) and (A.2) of \cite[Theorem A.1]{MRS} are satisfied by \eqref{growthh}--\eqref{liminfdiss}. Hypothesis (A.3) follows by adapting \cite[Lemmas 3.4 and 3.5]{MS}, whereas condition (A.4) follows directly by \eqref{unifapest} and \eqref{bunifapest}. 
Hence, by \cite[Theorem A.1]{MRS}, there holds
\begin{equation}
\label{convpept} 
\begin{array}{c}  p^{\ep}(t)\deb p(t)\quad\text{weakly in }L^2(\Omega;\mthree)\quad\text{for every }t\in [0,T],\vspace{0.2 cm}\\
\cal{D}_{H_D}(p;0,t)\leq \displaystyle{\liminf_{\ep\to 0}}\frac{1}{\ep^{\alpha-1}}\cal{D}(P^{\ep};0,t)\quad\text{for every }t\in [0,T].
 \end{array}
\end{equation}
Moreover, by $\eqref{convP0}$, $p(0)=\mathring{p}$.

Let now $t\in [0,T]$ be fixed. By \eqref{bdtep}, \eqref{rt2t}, \eqref{rt4t} and Poincar\'e inequality, up to subsequences there holds
\begin{equation}
\label{cvitsb}
y^{\ep}(t)\to\Big(\begin{array}{c}x'\\0\end{array}\Big)\quad\text{strongly in }W^{1,2}(\Omega;\R^3).
\end{equation}
Arguing as in the proof of Theorem \ref{liminfineq} and owing to \eqref{unifapest}, we deduce the existence of a pair $(u^*(t),v^*(t))\in W^{1,2}(\omega;\R^2)\times W^{2,2}(\omega)$ such that $(u^*(t),v^*(t),p(t))\in\cal{A}(u^0(t),v^0(t))$ and a sequence $\ep_j\to 0$ such that
\begin{eqnarray}
\label{cvu+}&&u^{\ep_j}(t)\deb u^*(t)\quad\text{weakly in }W^{1,2}(\omega;\R^2),\\
\label{cvv+}&&v^{\ep_j}(t)\to v^*(t)\quad\text{strongly in }W^{1,2}(\omega).
\end{eqnarray}
In particular, by \eqref{convu0} and \eqref{convv0} we have $u^*(0)=\mathring{u}$ and $v^*(0)=\mathring{v}$.
By \eqref{unifbdrt2} up to extracting a further subsequence, there exists a map $G^*(t)\in L^2(\Omega;\mthree)$ such that
\begin{equation}
\label{debsubsgep}
G^{\ep_j}(t)\deb G^*(t)\quad\text{weakly in }L^2(\Omega;\mthree)
\end{equation}
and the $2\times 2$ submatrix $(G^*)'(t)$ satisfies
\begin{equation}
\label{gaff+}
(G^*)'(t,x) = G^*_0(t,x') - x_3 (\nabla')^2 v^*(t,x')\quad\text{for a.e. }x\in\Omega,
\end{equation}
where
\begin{eqnarray}
&&\label{Ga3+} \sym\, G^*_0(t) = \sym\,\nabla' u^*(t).
\end{eqnarray}

We shall show that the triple $(u^*(t),v^*(t),p(t))$ satisfies the reduced stability condition (qs1$_{r\alpha}$). By {Lemma \ref{bdc}}, it is enough to prove the inequality for triples $(\hat{u},\hat{v},\hat{p})\in\cal{A}(u^0(t),v^0(t))$ such that 
{\begin{eqnarray*}
&&\tilde{u}:=\hat{u}-u^*(t)\in C^{\infty}_c(\omega\cup\gamma_n;\R^2),\\
&&\tilde{v}:=\hat{v}-v^*(t)\in C^{\infty}_c(\omega\cup\gamma_n),\\
&&\tilde{p}:=\hat{p}-p^*(t)\in C^{\infty}_c(\Omega;\md).
\end{eqnarray*}}
By Theorem \ref{mutrecseq} there exists a sequence $(\hat{y}^{\ep_j},\hat{P}^{\ep_j})\in\cal{A}_{\ep_j}(\pepjt(t))$ satisfying
\begin{eqnarray*}
&& \intom{Q_2(\sym\,\hat{G}'-\hat{p}')}+\intom{\B{\hat{p}}}\\
&&-\intom{Q_2(\sym\,({G}^*)'(t)-{p}'(t))}-\intom{\B{p(t)}}+\intom{H_D(\hat{p}-p(t))}\\
&&\geq \limsup_{\ep_j\to 0} \Big\{\frac{1}{{\ep_j}^{2\alpha-2}}\intom{W_{el}({\nep}_j {\hat{y}}^{\ep_j} ({\hat{P}}^{\ep_j})^{-1})}+\frac{1}{{\ep_j}^{2\alpha-2}}\intom{W_{hard}({\hat{P}}^{\ep_j})}\\
&&-\frac{1}{{\ep_j}^{2\alpha-2}}\intom{W_{el}({\nep}_j y^{\ep_j}(t) (P^{\ep_j})^{-1}(t))}-\frac{1}{{\ep_j}^{2\alpha-2}}\intom{W_{hard}(P^{\ep_j}(t))}\\
&&+\frac{1}{{\ep_j}^{\alpha-1}}\intom{D(P^{\ep_j}(t),\hat{P}^{\ep_j})}\Big\}\\
\end{eqnarray*}
 where 
\begin{equation}
\nonumber
\hat{G}'(x', x_3) := \hat{G}_0(x') - x_3 (\nabla')^2 \hat{v}(x')\quad\text{a.e. in }\Omega,
\end{equation}
and
\begin{eqnarray}
&&\nonumber \sym\, \hat{G}_0 = \sym\, \nabla' \hat{u}.
\end{eqnarray}
Inequality (qs1$_{r\alpha}$) follows now by the $\ep$-stability (qs1) of $(y^{\ep}(t),P^{\ep}(t))$. 

By strict convexity of the quadratic form $Q_2$, an adaptation of \cite[Theorem 3.8]{DDM} yields that, once $p(t)$ is identified, there exist unique $u(t)\in W^{1,2}(\omega;\R^2)$ and $v(t)\in W^{2,2}(\omega)$ such that (qs1$_{r\alpha}$) holds at time $t$. This implies that $u^*(t)=u(t)$, $v^*(t)=v(t)$ for every $t\in [0,T]$ and both \eqref{cvu+} and \eqref{cvv+} hold for the whole sequences $u^{\ep}(t)$ and $v^{\ep}(t)$ and for every $t\in [0,T]$. Moreover, by \eqref{debsubsgep}--\eqref{Ga3+} we have
\begin{equation}
\nonumber
\sym\,{(G^*)}'(t) = \sym\,\nabla'u(t) - x_3 (\nabla')^2 v(t)
\end{equation}
and 
$$\sym\,(G^{\ep})'(t)\deb  \sym\,\nabla'u(t) - x_3 (\nabla')^2 v(t)\quad\text{weakly in }L^2(\Omega;\mthree)\quad\text{for every }t\in [0,T].$$
{\em{Step 3: Convergence of the scaled stress}}\\
In this step we shall show that for every $t\in[0,T]$ there exists a subsequence $\ep_j$, possibly depending on $t$, such that 
\begin{equation}
\label{goodsetstress}
\chi_{\ep_j}(t)E^{\ep_j}(t)\deb E^*(t)\quad\text{weakly in }L^2(\Omega;\mthree),
\end{equation}
where 
\begin{equation}
\label{cestar}
E^*(t)=\C(G^*(t)-p(t)).
\end{equation}

To this purpose, for $t\in [0,T]$ fixed, let $\ep_j\to 0$ be such that \eqref{debsubsgep} holds and let $F^{\ep_j}(t)$ be the map defined in \eqref{deffept}.
By \eqref{unifapest}, \eqref{linftybdw} and \eqref{debsubsgep} we deduce
$$\|\ep^{\alpha-1}G^{\ep_j}(t)(w^{\ep_j}(t)-p^{\ep_j}(t))\|_{L^2(\Omega;\mthree)}\leq C\quad\text{for every }\ep_j.$$
On the other hand, by \eqref{unifapest}, \eqref{l2bdw}, and \eqref{debsubsgep}, there holds
\begin{equation}
\label{numerarepag39}
\ep^{\alpha-1}G^{\ep_j}(t)(w^{\ep_j}(t)-p^{\ep_j}(t))\to 0\quad\text{strongly in }L^1(\Omega;\mthree).
\end{equation}
Hence, we conclude that
\begin{equation}
\label{toolconvfep}
\ep^{\alpha-1}G^{\ep_j}(t)(w^{\ep_j}(t)-p^{\ep_j}(t))\deb 0\quad\text{weakly in }L^2(\Omega;\mthree).
\end{equation}
Collecting \eqref{deffept}, \eqref{weakl2w}, \eqref{convpept}, \eqref{debsubsgep} and \eqref{toolconvfep} we obtain
\begin{equation}
\label{convfep}
F^{\ep_j}(t)\deb G^*(t)-p(t)\quad\text{weakly in }L^2(\Omega;\mthree).
\end{equation}
By \eqref{measbep} we deduce that $\chi_{\ep_j}(t)\to 1$ boundedly in measure, therefore by \eqref{convfep} there holds
\begin{equation}
\nonumber
\chi_{\ep_j}(t)F^{\ep_j}(t)\deb G^*(t)-p(t)\quad \text{weakly in }L^2(\Omega;\mthree).
\end{equation}
Now, estimate \eqref{unifinteg} implies that the sequence $({\tilde{E}}^{\ep_j}(t))$ is uniformly bounded in $L^1(\Omega;\mthree)$ and is equiintegrable, hence by  Dunford-Pettis Theorem, up to extracting a further subsequence, there exists ${E}^*(t)\in L^1(\Omega;\ms)$ such that
\begin{equation}
\nonumber
{\tilde{E}}^{\ep_j}(t)\deb {E}^*(t)\quad\text{weakly in }L^1(\Omega;\mthree).
\end{equation}
Using a Taylor expansion argument in $O_{\ep}(t)$, and arguing as in \cite[Proof of Theorem 3.1, Step 3]{M-S} we deduce
\begin{equation}
\nonumber
\chi_{\ep_j}(t){\tilde{E}}^{\ep_j}(t)\deb \C (G^*(t)-p(t))\quad\text{weakly in }L^2(\Omega;\ms).
\end{equation}
By \eqref{rt1t} and \eqref{rt4t}, the sequence $(R^{\ep}(t))$ converges boundedly in measure to the identity,  hence the previous convergence implies in particular \eqref{goodsetstress} and \eqref{cestar}.\\
{\em{Step 4: Characterization of the limit stress}}\\
In this step we shall show that
\begin{equation}
\label{claimstress}
E^*(t)=\C_2 (\sym\,\nabla' u(t)-x_3(\nabla')^2v(t)-p'(t)):=E(t)\quad\text{for every }t\in [0,T].
\end{equation}
This, in turn, will imply that all convergence properties established in the previous step hold for the entire sequences and for every $t\in [0,T]$.

We first remark that, choosing $\tilde{P}=P^{\ep}(t)$ in (qs1) there holds
\begin{equation}
\label{minimel}
\intom{W_{el}(\nep y^{\ep}(t)(P^{\ep})^{-1}(t))}\leq \intom{W_{el}(\nep \tilde{y} (P^{\ep})^{-1}(t))},
\end{equation}
for every $\tilde{y}\in W^{1,2}(\Omega;\R^3)$ such that {$\tilde{y}=\phi^{\ep}(t,(x',\ep x_3))\quad\cal{H}^2\text{ - a.e. on }\Gamma_d$}.

{Let $\eta\in W^{1,\infty}(\R^3;\R^3)\cap C^{\infty}(\R^3;\R^3)$ be such that $\eta\circ \phi^{\ep}(t,(x',\ep x_3))=0\quad\cal{H}^2\text{ - a.e. on }\Gamma_d$.} Then, in particular, we can consider in \eqref{minimel} inner variations of the form $y^{\ep}+\lambda \eta \circ y^{\ep}$, where $\lambda \in \R$. By the growth hypothesis \eqref{mandel2} and by the minimality condition \eqref{minimel}, an adaptation of \cite[Theorem 2.4]{Ba} shows that $y^{\ep}(t)$ satisfies the following Euler-Lagrange equation:
\begin{equation}
\label{euler}
 \intom{DW_{el}(\nep y^{\ep}(t)(P^{\ep})^{-1}(t))(\nep y^{\ep}(t)(P^{\ep})^{-1}(t))^T:\nabla \eta (y^{\ep}(t))}=0
\end{equation}
for every $t\in [0,T]$ and for every $\eta\in W^{1,\infty}(\R^3;\R^3)\cap C^{\infty}(\R^3;\R^3)$ such that {$\eta \circ \phi^{\ep}(t,(x',\ep x_3))=0\quad\cal{H}^2\text{ - a.e. on }\Gamma_d$}. 
Hence,
\begin{equation}
\label{eulerstress}
\intom{E^{\ep}(t):\nabla \eta (y^{\ep}(t))}=0
\end{equation}
for every $t\in [0,T]$ and for every $\eta\in W^{1,\infty}(\R^3;\R^3)\cap C^{\infty}(\R^3;\R^3)$ such that {$\eta \circ \phi^{\ep}(t,(x',\ep x_3))=0\quad\cal{H}^2\text{ - a.e. on }\Gamma_d$}. 

Now, fix $t\in [0,T]$ and let $\ep_j$ be the sequence selected in the previous step.
Let {$\eta\in W^{1,\infty}(\R^3;\R^3)\cap C^{\infty}(\R^3;\R^3)$ be such that $\eta=0\quad\cal{H}^2\text{ - a.e. on }\Gamma_d$}. We consider the maps $\eta^{\ep_j}(t)\in W^{1,\infty}(\R^3,\R^3)\cap C^{\infty}(\R^3;\R^3)$ defined as
$$\eta^{\ep_j}(t):=\ep_j \eta\big(\vepj_1(t), \vepj_2(t), \tfrac{1}{\ep_j}\vepj_3(t)\big).$$
It is clear that {$\eta^{\ep_j}(t)\circ \phi^{\ep_j}(t,(x',\ep_j x_3))=0\quad\cal{H}^2\text{ - a.e. on }\Gamma_d$}, hence we can use $\eta^{\ep_j}(t)$ as a test function in \eqref{eulerstress} and we obtain
\begin{equation}
\label{testfctpart}
\intom{E^{\epjt}(t):\nabla\eta^{\ep_j} (y^{\ep_j}(t))}=0
\end{equation}
for every $j$.

Now, set $\xi^{\ep_j}(x)=\big(\vepj_1(t,x), \vepj_2(t,x), \tfrac{1}{\ep_j}\vepj_3(t,x)\big)$ for every $x\in \R^3$. We can rewrite \eqref{testfctpart} as
\begin{eqnarray}
\nonumber &&\sum_{i=1,2,3}\ep_j\intom{E^{\epjt}(t)e_i\cdot\sum_{k=1,2}\partial_k\eta (\xi^{\ep_j}(y^{\ep_j}(t)))\partial_i \xi^{\ep_j}_k(y^{\ep_j}(t))}\\
\nonumber &&+\ep_j\sum_{i=1,2}\intom{E^{\epjt}(t)e_i\cdot\partial_3\eta (\xi^{\ep_j}(y^{\ep_j}(t)))\partial_i \xi^{\ep_j}_3(y^{\ep_j}(t))}\\
\label{eulerconti}&&+\ep_j\intom{E^{\epjt}(t)e_3\cdot \partial_3\eta (\xi^{\ep_j}(y^{\ep_j}(t)))\partial_3 \xi^{\ep_j}_3(y^{\ep_j}(t))}=0.
\end{eqnarray}
Since $\eta\in W^{1,\infty}(\R^3,\R^3)$ and $E^{\epjt}(t)$ is uniformly bounded in $L^1(\Omega;\mthree)$ by \eqref{unifintegnt}, estimate \eqref{bddinverse} yields that the term in the first row of \eqref{eulerconti} converges to zero. By \eqref{gradinv1}, the term in the second row of \eqref{eulerconti} can be bounded as follows:
\begin{eqnarray*}
&&\Big|\ep_j\sum_{i=1,2}\intom{E^{\epjt}(t)e_i\cdot\partial_3\eta (\xi^{\ep_j}(y^{\ep_j}(t)))\partial_i \xi^{\ep_j}_3(y^{\ep_j}(t))}\Big|\leq C\ep_j^{\alpha-2}\|E^{\epjt}(t)e_i\|_{L^1(\Omega;\R^3)}
\end{eqnarray*}
and hence converges to zero due to \eqref{unifintegnt}.
By \eqref{gradinv1}, there holds
\begin{eqnarray}
\nonumber &&\Big|\ep_j\intom{E^{\epjt}(t)e_3\cdot \partial_3\eta (\xi^{\ep_j}(y^{\ep_j}(t)))\partial_3 \xi^{\ep_j}_3(y^{\ep_j}(t))}-\intom{E^{\epjt}(t)e_3\cdot \partial_3\eta (\xi^{\ep_j}(y^{\ep_j}(t)))}\Big|\\
\nonumber &&\leq C{\ep_j}^{\alpha-2}\|E^{\epjt}(t)e_3\|_{L^1(\Omega;\R^3)}.
\end{eqnarray}
which converges to zero, owing to \eqref{unifintegnt}. Therefore, \eqref{eulerconti} yields 
\begin{equation}
\label{remaineuler}
\lim_{\ep_j\to 0}\intom{E^{\epjt}(t)e_3\cdot \partial_3\eta (\xi^{\ep_j}(y^{\ep_j}(t)))}= 0.
\end{equation}
By \eqref{lp1}, \eqref{distinv1} and \eqref{cvitsb} we deduce 
\begin{equation}
\nonumber
\xi^{\ep_j}_k(y^{\ep_j}(t))\to x_k\quad\text{strongly in }L^2(\Omega)\quad\text{for }k=1,2.
\end{equation}
Since $\alpha>3$, by \eqref{distinv13} and \eqref{estimate3} we have $\xi^{\ep_j}_3(y^{\ep_j}(t))\to x_3$ strongly in $L^{2}(\Omega)$. Hence, by the regularity of $\eta$, 
\begin{eqnarray}
\nonumber && \partial_3\eta (\xi^{\ep_j}(y^{\ep_j}(t)))\to \partial_3 \eta (t,x)\quad\text{a.e. in }\Omega\text{ as }\ep_j\to 0.
\end{eqnarray}
By the dominated convergence theorem and by combining \eqref{lp1}, \eqref{badsetstressnt}, \eqref{goodsetstress} and \eqref{remaineuler}, we conclude that
$$\intom{E^*(t)e_3\cdot \partial_3 \eta (t)}=0,$$
for every {$\eta\in W^{1,\infty}(\R^3;\R^3)\cap C^{\infty}(\R^3;\R^3)$ such that $\eta=0\quad\cal{H}^2\text{ - a.e. on }\Gamma_d$}. Hence,
\begin{equation}
\label{3colstress}
E^*(t)e_3=0\quad\text{a.e. in } \Omega.
\end{equation}
By combining \eqref{linearmin}, \eqref{cestar} and \eqref{3colstress} we deduce \eqref{claimstress}. Moreover, by \eqref{linearmin} 
{\begin{equation}
\label{nnpserve}
\sym\,G^*(t)-p(t)=\A(\sym\, \nabla' u(t)-x_3(\nabla')^2 v(t)-p'(t)),\text{ for every }t\in [0,T].
\end{equation}}
{\em{Step 5: Reduced energy balance}}\\
To complete the proof of the theorem it remains to show that the triple $(u(t),v(t),p(t))$ satisfies
\begin{eqnarray}
\nonumber &&\intom{Q_2\big(\sym\nabla' u(t)-x_3(\nabla')^2 v(t)-p'(t)\big)}+\intom{\B{p(t)}}+\cal{D}(p;0,t)\\
\nonumber &&\leq\intom{Q_2\big(\sym\nabla' u(0)-x_3(\nabla')^2 v(0)-p'(0)\big)}+\intom{\B{p(0)}}\\
\nonumber &&+\int_0^t{\intom{\C_2(\sym\,\nabla' u(s)-x_3(\nabla')^2 v(s)-p'(s)):\Big(\begin{array}{cc}\nabla \dot{u}^0(s)-x_3(\nabla')^2 \dot{v}^0(s)&0\\0&0\end{array}\Big)}\,ds}.\\
\label{rin1}
\end{eqnarray}
Once \eqref{rin1} is proved, the opposite inequality in (qs2$_{r\alpha}$) follows by adapting of \cite[Theorem 4.7]{DDM}.

We claim that, to prove \eqref{rin1} it is enough to show that 
\begin{equation}
\label{claimtool21}
\frac{1}{\ep^{\alpha-1}}\sym\Big(\nabla \dot{\pep}(t,\zep(t))(\nabla \pep)^{-1}(t,\zep(t))\Big)\to  \sym\Big(\begin{array}{cc}\nabla' \dot{u^0}(t)-x_3(\nabla')^2 \dot{v}^0(t)&0\\0&0\end{array}\Big) 
\end{equation}
strongly in $L^2(\Omega;\mthree)$, for all $t\in [0,T]$.
Indeed, if \eqref{claimtool21} holds, by \eqref{lp1}, \eqref{badsetstressnt}, \eqref{linftybdwk}, \eqref{goodsetstress} and \eqref{claimstress}, one has
$$\frac{1}{\ep^{\alpha-1}}\intom{E^{\ep}(s):\nabla \dot{\pep}(s,\zep(s))(\nabla \pep)^{-1}(s,\zep(s))}\to \intom{E(s): \sym\Big(\begin{array}{cc}\nabla' \dot{u^0}(s)-x_3(\nabla')^2 \dot{v}^0(s)&0\\0&0\end{array}\Big)},
$$
for every $s\in [0,t]$. Hence, by \eqref{unifbdeep} and the dominated convergence theorem we deduce
\begin{eqnarray}
\nonumber &&
\frac{1}{\ep^{\alpha-1}}\int_0^t{\intom{E^{\ep}(s):\nabla \dot{\pep}(s,\zep(s))(\nabla \pep)^{-1}(s,\zep(s))}\,ds}\\
\label{claimwork1}&&\to\int_0^t{\intom{E(s): \sym\Big(\begin{array}{cc}\nabla \dot{u^0}(s)-x_3\nabla^2 \dot{v}^0(s)&0\\0&0\end{array}\Big)}\,ds}.
\end{eqnarray}
On the other hand, by Theorem \ref{liminfineq} there holds
\begin{eqnarray}
\nonumber &&\intom{Q_2\big(\sym\nabla' u(t)-x_3(\nabla')^2 v(t)-p'(t)\big)}+\intom{\B{p(t)}}\\
\nonumber &&\leq\liminf_{\ep\to 0}\frac{1}{\ep^{2\alpha-2}}\cal{F}_{\ep}(t,\zep(t),P^{\ep}(t)).
\end{eqnarray}
Therefore, once \eqref{claimtool21} is proved, by \eqref{convpept} and \eqref{claimwork1}, passing to the liminf in the $\ep$ energy balance (qs2), inequality \eqref{rin1} follows by \eqref{convE0}.

To prove \eqref{claimtool21}, we first study some properties of the maps $\zep(t)$. By \eqref{fcomp} and \eqref{relzep} there holds
$$\zep_i(t)=y^{\ep}_i(t)-\ep^{\alpha-1}u^0_i(t,\bvep(t,y^{\ep}(t)))+\ep^{\alpha-1}\tep\Big(\frac{\vep_3(t,y^{\ep}(t))}{\ep}\Big)\partial_i v^0(t,\bvep(t,y^{\ep}(t)))$$
for every $t\in [0,T]$, $i=1,2$. Hence, by \eqref{treq3}, \eqref{lp1} and \eqref{cvitsb} we deduce
\begin{equation}
\label{distzepi}
\zep_i(t)\to x_i\quad\text{strongly in }L^2(\Omega)\quad\text{for every }t\in[0,T],\,i=1,2. 
\end{equation}
Moreover, by \eqref{zep3} we have
\begin{eqnarray}
\nonumber
&&\Big\|\frac{\zep_3(t)}{\ep}-x_3-\ep^{\alpha-3}v(t)+\ep^{\alpha-3}v^0(t)\Big\|_{L^2(\Omega)}\leq \Big\|\frac{y^{\ep}_3(t)}{\ep}-x_3-\ep^{\alpha-3}v^{\ep}(t)\Big\|_{L^2(\Omega)}\\
\nonumber &&+\ep^{\alpha-3}\|v^{\ep}(t)-v(t)\|_{L^2(\Omega)}+\ep^{\alpha-3}\|v^0(t)-v^0(t,\bvep(t,y^{\ep}(t)))\|_{L^2(\Omega)}.
\end{eqnarray}
Hence, by \eqref{distinv1}, \eqref{cvvt}, \eqref{estimate3} and \eqref{cvitsb}, 
\begin{equation}
\label{distzep3}
\Big\|\frac{\zep_3(t)}{\ep}-x_3-\ep^{\alpha-3}v(t)+\ep^{\alpha-3}v^0(t)\Big\|_{L^2(\Omega)}\to 0
\end{equation}
 for every $t\in [0,T]$. In particular, by Lemma \ref{cvproptep},
 \begin{equation}
 \label{usfz}
 \tep\Big(\frac{\zep_3(t)}{\ep}\Big)\to x_3\quad\text{strongly in }L^2(\Omega).
 \end{equation}

 Arguing as in the proof of \eqref{l2bdwk}, we perform the decomposition 
 \begin{eqnarray}
\nonumber&& \frac{1}{\ep^{\alpha-1}}\sym\Big(\nabla \dot{\pep}(t,\zep(t))(\nabla \pep)^{-1}(t,\zep(t))\Big)=\frac{1}{\ep^{\alpha-1}}\sym\big(\nabla\dot{\pep}(t,\zep(t))\big)\\
 \label{decompst}&&+\frac{1}{\ep^{\alpha-1}}\sym\Big(\nabla\dot{\pep}(t,\zep(t))\Big((\nabla \pep)^{-1}(t,\zep(t))-Id\Big)\Big).
 \end{eqnarray}
 By \eqref{lp2}, \eqref{gradpept}, \eqref{l2usf2}, \eqref{distzepi} and \eqref{usfz}, we obtain
  \begin{equation}
 \label{term11}
 \frac{1}{\ep^{\alpha-1}}\sym\big(\nabla\dot{\pep}(t,\zep(t))\big)\to \sym\Big(\begin{array}{cc}\nabla \dot{u}^0(t)-x_3\nabla^2 \dot{v}^0(t)&0\\0&0\end{array}\Big)
 \end{equation}
 strongly in $L^2(\Omega;\mthree)$. To study the second term in the right-hand side of \eqref{decompst}, we remark that by \eqref{gradpept} and \eqref{utiledopo}, there holds
 \begin{eqnarray*}
 &&\Big\|\frac{1}{\ep^{\alpha-1}}\sym\Big(\nabla \dot{\pep}(t,\zep(t))(\nabla \pep)^{-1}(t)(\zep(t))-Id)\Big)\Big\|_{L^2(\Omega;\mthree)}\\
 &&\leq C \Big(1+\Big\|\tep\Big(\frac{\zep_3(t)}{\ep}\Big)\Big\|_{L^2(\Omega)}\Big)\|(\nabla \pep(t))^{-1}(\zep(t))-Id\|_{L^{\infty}(\Omega;\mthree)}+C\ep^{\alpha-3}.
 \end{eqnarray*}
 On the other hand, \eqref{lp2}, \eqref{gradinv}, \eqref{gradinv1} and \eqref{estt2bis} yield
 $$\|(\nabla \pep(t))^{-1}(\zep(t))-Id\|_{L^{\infty}(\Omega;\mthree)}\leq C\ep^{\alpha-1}\ell_{\ep}.$$
 Hence, by \eqref{lp1} and \eqref{usfz} we have
 \begin{eqnarray}
 \frac{1}{\ep^{\alpha-1}}\sym\Big(\nabla \dot{\pep}(t,\zep(t))(\nabla \pep)^{-1}(t,\zep(t))-Id)\Big)\to 0
 \label{term11bis}
 \end{eqnarray}
 strongly in $L^2(\Omega;\mthree)$. By combining \eqref{term11} and \eqref{term11bis} we obtain \eqref{claimtool21}. This completes the proof of the theorem.
 \end{proof}
 We give only a sketch of the proof of Theorem \ref{cvstress} in the case $\alpha=3$, as it follows closely that of Theorem \ref{cvstress} for $\alpha>3$. 
 \begin{proof}[Proof of Theorem \ref{cvstress} in the case $\alpha=3$]{\quad}\\
 \emph{Steps 0--3}\\
Steps 0--3 follow as a straightforward adaptation of the corresponding steps in the case $\alpha>3$, 
where now \eqref{sqrtdeco} holds with
\begin{eqnarray*}M(t,x)&:=&\sym\Big(\begin{array}{cc}\nabla' {u}^0(t,x')-x_3(\nabla')^2 {v}^0(t,x')&0\\0&0\end{array}\Big)\\
&+&\frac{1}{2}\Big(\begin{array}{cc}\nabla' v^0(t,x')\otimes \nabla'v^0(t,x')&0\\0&|\nabla' v^0(t,x')|^2 \end{array}\Big)
\end{eqnarray*}
for every $x\in\Omega$ and for all $t\in [0,T]$. The only relevant difference is that we can not conclude that $u(t)$ and $v(t)$ are uniquely determined once $p(t)$ is identified. Hence, now all convergence properties hold on $t$-dependent subsequences.\\
{\em{Step 4: Characterization of the limit stress}}\\
Arguing exactly as in Step 4 of the proof of Theorem \ref{cvstress} for $\alpha>3$, we obtain
 \begin{equation}
\label{almeuler}
\intom{E(t,x)e_3\cdot \partial_3 \eta (t,(x',x_3+v(t,x')-v^0(t,x')))}=0
\end{equation}
for every {$\eta\in W^{1,\infty}(\R^3;\R^3)\cap C^{\infty}(\R^3;\R^3)$ such that $\eta=0\quad\cal{H}^2\text{ - a.e. on }\Gamma_d$}. We consider a sequence $(w_k)\subset C^\infty_c(\omega)$ that converges to $v(t)-v^0(t)$ strongly in $L^2(\omega)$. We take as test functions in \eqref{almeuler} the maps $\eta_k(x):=\eta(x',x_3-w_k(x'))$, where {$\eta\in W^{1,\infty}(\R^3,\R^3)\cap C^{\infty}(\R^3;\R^3)$ and $\eta=0\quad\cal{H}^2\text{ - a.e. on }\Gamma_d$}. We have
$$\intom{E(t,x)e_3\cdot \partial_3 \eta (t,(x',x_3+v(t,x')-v^0(t,x')-w_k(x')))}=0\quad\text{for every }k.$$
Passing to the limit as $k\to +\infty$ in the previous equation, by the dominated convergence theorem we deduce
$$\intom{E(t)e_3\cdot \partial_3 \eta}=0$$
for every {$\eta\in W^{1,\infty}(\R^3,\R^3)\cap C^{\infty}(\R^3;\R^3)$ such that $\eta=0\quad\cal{H}^2\text{ - a.e. on }\Gamma_d$}, which implies $E(t)e_3=0$ a.e. in $\Omega$. Hence, \eqref{linearmin} yields
\begin{equation}
\nonumber
E(t)=\C_2(e(t)),
\end{equation}
and
{\begin{equation}
\label{strlima3}
\sym\,G(t)-p(t)=\A (\sym\nabla' u(t)+\tfrac{1}{2}\nabla' v(t)\otimes\nabla' v(t)-x_3(\nabla')^2 v(t)-p'(t)).
\end{equation}}
{\em{Step 5: Reduced energy balance}}\\
Arguing as in Step 5 of the case $\alpha>3$, to prove (qs2$_{r3}$) it is enough to show that
 \begin{eqnarray}
\nonumber &&\intom{Q_2(e_3(t))}+\intom{\B{p(t)}}+\cal{D}_{H_D}(p;0,t)\\
\nonumber &&\leq\intom{Q_2(e_3(0))}+\intom{\B{p(0)}}\\
\nonumber &&+\int_0^t{\intom{\C_2(e_3(s)):\Big(\begin{array}{cc}\nabla \dot{u}^0(s)+\nabla' v(s)\otimes \nabla' \dot{v}^0(s)-x_3(\nabla')^2 \dot{v}^0(s)&0\\0&0\end{array}\Big)}\,ds},\\
\label{rin13}
\end{eqnarray}
where $t\mapsto e_3(t)$ is the map defined in \eqref{dsigmat}.
Indeed, once \eqref{rin13} is proved, (qs2$_{r3}$) follows by adapting \cite[Theorem 4.7]{DDM} according to Remark \ref{csd}.
To prove \eqref{rin13}, we argue as in \cite[Lemma 5.1]{B} and we set
$$\Theta^{\ep}(t):=\frac{1}{\ep^2}\intom{E^{\ep}(t):\nabla \dot{\pep}(t,\zep(t))(\nabla \pep)^{-1}(t,\zep(t))},$$
$$\Theta(t):=\limsup_{\ep\to 0}\Theta^{\ep}(t)$$
for every $t\in [0,T]$. By \eqref{unifbdeep} (which is still true for $\alpha=3$), $\Theta(t)\in L^1([0,T])$ and by Fatou's lemma there holds
\begin{equation}
\label{limsuptau}
\limsup_{\ep\to 0}\int_0^t{\Theta^{\ep}(s)\,ds}\leq \int_0^t{\Theta(s)\,ds}.
\end{equation}
Now, by Theorem \ref{liminfineq} we know that
\begin{eqnarray*}
\nonumber &&\intom{Q_2(e(t))}+\intom{\B{p(t)}}\leq \liminf_{\ep\to 0}\frac{1}{\ep^{2\alpha-2}}\cal{F}^{\ep}(t,\zep(t), P^{\ep}(t)).
\end{eqnarray*}
By (qs2), \eqref{convE0}, \eqref{convpept} and \eqref{limsuptau} we deduce
\begin{eqnarray*}
\nonumber &&\intom{Q_2(e(t))}+\intom{\B{p(t)}}+\cal{D}_{H_D}(p;0,t)\leq\intom{Q_2(e(0))}+\intom{\B{p(0)}}+\int_0^t{\Theta(s)\,ds}.
\end{eqnarray*}
Hence, to prove \eqref{rin13} it is enough to show that 
\begin{equation}
\label{identtaue}
\Theta(t)=\intom{E(t):\Big(\begin{array}{cc}\nabla \dot{u}^0(t)+\nabla' v(t)\otimes \nabla' \dot{v}^0(t)-x_3(\nabla')^2 \dot{v}^0(t)&0\\0&0\end{array}\Big)}
\end{equation}
for a.e. $t\in [0,T]$.

To this purpose, fix $t\in [0,T]$ and let $\epjt\to 0$ be such that 
$$\Theta(t)=\lim_{\epjt\to 0}\Theta^{\epjt}(t).$$
Up to extracting a further subsequence, we may assume that $\epjt$ is the same subsequence we selected in the previous steps. We claim that
\begin{eqnarray}
\nonumber
&&\frac{1}{\epjt^{2}}\sym\Big(\nabla \dot{\pepjt}(t,\zepjt(t))(\nabla \pepjt)^{-1}(t,\zepjt(t))\Big)\\
\nonumber&&\to  \sym\Big(\begin{array}{cc}\nabla' \dot{u}^0(t)+\nabla' \dot{v}^0(t)\otimes \nabla' {v}^0(t)-(x_3+v(t)-v^0(t))(\nabla')^2 \dot{v}^0(t)&0\\0&\frac{d}{dt}\frac{|\nabla' v^{0}(t)|^2}{2}\end{array}\Big)\\
\label{lhtg} 
\end{eqnarray}
strongly in $L^2(\Omega;\mthree)$. To prove the claim, we perform the decomposition \eqref{decompst}. Now, arguing as in the proof of \eqref{term11}, and using \eqref{distzep3} and Lemma \ref{cvproptep} we obtain
\begin{equation}
\label{term12}
\frac{1}{\epjt^{2}}\sym(\nabla \dot{\pepjt}(t,\zepjt(t)))\to  \sym\Big(\begin{array}{cc}\nabla' \dot{u}^0(t)-(x_3+v(t)-v^0(t))(\nabla')^2 \dot{v}^0(t)&0\\0&0\end{array}\Big) 
\end{equation}
strongly in $L^2(\Omega;\mthree)$. To study the second term in the right-hand side of \eqref{decompst}, we remark that by
\eqref{lp2}, \eqref{gradinv}, \eqref{gradinv1} and \eqref{estt2bis}, one has
\begin{equation}
\nonumber
\Big\|(\nabla \pepjt)^{-1}(t,\zepjt(t))-Id\Big\|_{L^{\infty}(\Omega;\mthree)}\leq C\epjt^{2}\ell_{\epjt}.
\end{equation}
By \eqref{l2usf}, there holds
\begin{eqnarray}
\nonumber
&&\Big\|\Big(\begin{array}{cc}\nabla' \dot{u}^0(t,\bzepjt(t))-\tepjt\big(\frac{\zepjt_3(t)}{\epjt}\big)(\nabla')^2 \dot{v}^0(t,\bzepjt(t))&0\\0&0\end{array}\Big)\Big((\nabla \pepjt)^{-1}(t,\zepjt(t))-Id\Big)\Big\|_{L^2(\Omega;\mthree)}\\
\label{term13}
&&\leq C\epjt^{2}\ell_{\epjt},
\end{eqnarray}
which tends to zero due to \eqref{lp1}. 

By \eqref{gradpept}, it remains only to study the asymptotic behaviour of 
$$\frac{1}{\epjt}\Big(\begin{array}{cc}0&-\dot{\theta}^{\epjt}\big(\frac{\zepjt_3(t)}{\epjt}\big)\nabla' \dot{v}^0(t,\bzepjt(t))\\(\nabla' \dot{v}^0(t,\bzepjt(t)))^T&0\end{array}\Big)\Big((\nabla \pepjt)^{-1}(t,\zepjt(t))-Id\Big).$$
By \eqref{estt2bis}, this is the same as studying the quantity
 $$\frac{1}{\epjt}\Big(\begin{array}{cc}0&-\dot{\theta}^{\epjt}\big(\frac{\zepjt_3(t)}{\epjt}\big)\nabla' \dot{v}^0(t,\bzepjt(t))\\(\nabla' \dot{v}^0(t,\bzepjt(t)))^T&0\end{array}\Big)\Big(\nabla \vepjt(t,y^{\epjt}(t))-Id\Big).$$
 
We claim that
\begin{equation}
\label{fclhp}\frac{1}{\epjt}\Big(\nabla \vepjt(t,y^{\epjt}(t))-Id\Big)\to \Big(\begin{array}{cc}0&\nabla' v^0(t)\\-(\nabla' v^0(t))^T&0\end{array}\Big)
\end{equation}
strongly in $L^2(\Omega;\mthree)$. Indeed, by \eqref{expgr} and \eqref{bddinverse} and the smoothness of $u^0$ and $v^0$,
\begin{eqnarray*}
&&\Big\|\frac{1}{{\epjt}}\Big(\nabla (\vepjt)'(t,y^{\epjt}(t))-\Big(\begin{array}{ccc}1&0&0\\0&1&0\end{array}\Big)\Big)-(0|\nabla' v^0(t))\Big\|_{L^2(\Omega;\M^{2\times 3})}\leq C\epjt\Big\|\tepjt\Big(\frac{\vepjt_3(t,y^{\epjt}(t))}{\epjt}\Big)\Big\|_{L^2(\Omega)}\\
&&+\Big\|\dot{\theta}^{\epjt}\Big(\frac{\vepjt_3(t,y^{\epjt}(t))}{\epjt}\Big)\nabla' v^0(t, (\vep)'(t, y^{\epjt}(t)))\otimes (\nabla \vepjt_3(t, y^{\epjt}(t))-e_3)\Big\|_{L^2(\Omega;\M^{2\times 3})}\\
&&+\Big\|\dot{\theta}^{\epjt}\Big(\frac{\vepjt_3(t,y^{\epjt}(t))}{\epjt}\Big)\nabla' v^0(t, (\vep)'(t, y^{\epjt}(t)))-\nabla' v^0(t)\Big\|_{L^2(\Omega;\R^2)}+C\ep_j.
\end{eqnarray*}
By \eqref{treq2}, \eqref{3comp}, and \eqref{uf3}(which can be proved arguing exactly as in Step 1 of the case $\alpha>3$), we deduce
\begin{equation}
\label{nnpp}
\Big\|\tepjt \Big(\frac{\vepjt_3(t,y^{\epjt}(t))}{\epjt}\Big)\Big\|_{L^2(\Omega)}\leq\Big\|\frac{\vepjt_3(t,y^{\epjt}(t))}{\epjt}\Big\|_{L^2(\Omega)}\leq C\Big(\Big\|\frac{y^{\epjt}_3(t)}{\epjt}\Big\|_{L^2(\Omega)}+\|v^0\|_{L^{\infty}(\omega;\R^2)}\Big)\leq C.
\end{equation}
On the other hand, by \eqref{treq5} and \eqref{gradinv1}
\begin{eqnarray*}
&&\Big\|\dot{\theta}^{\epjt}\Big(\frac{\vepjt_3(t,y^{\epjt}(t))}{\epjt}\Big)\nabla' v^0(t, (\vep)'(t, y^{\epjt}(t)))\otimes (\nabla \vepjt_3(t, y^{\epjt}(t))-e_3)\Big\|_{L^2(\Omega;\M^{2\times 3})}\\
&&\leq C\|\nabla \vepjt_3(t,y^{\epjt}(t))-e_3\|_{L^{\infty}(\Omega;\R^3)}\leq C\epjt.
\end{eqnarray*}
Finally, by \eqref{nnpp} and Lemma \ref{cvproptep}
\begin{eqnarray*}
&&\Big\|\dot{\theta}^{\epjt}\Big(\frac{\vepjt_3(t,y^{\epjt}(t))}{\epjt}\Big)\nabla' v^0(t, (\vep)'(t, y^{\epjt}(t)))-\nabla' v^0(t)\Big\|_{L^2(\Omega;\R^2)}\\
&&\leq C \Big\|\dot{\theta}^{\epjt}\Big(\frac{\vepjt_3(t,y^{\epjt}(t))}{\epjt}\Big)-1\Big\|_{L^2(\Omega)}+\|\nabla' v^0(t, (\vep)'(t, y^{\epjt}(t)))-\nabla' v^0(t)\|_{L^2(\omega;\R^2)}\\
&&\leq \frac{C}{\ell_{\epjt}}+\|\nabla' v^0(t, (\vep)'(t, y^{\epjt}(t)))-\nabla' v^0(t)\|_{L^2(\omega;\R^2)}
\end{eqnarray*}
which converges to zero owing to \eqref{lp1}, \eqref{distinv1}, \eqref{distinv13}, \eqref{cvitsb} (which can be proved arguing exactly as in Step 2 of the case $\alpha>3$) and the dominated convergence theorem. By collecting the previous remarks, we obtain
$$\Big\|\frac{1}{{\epjt}}\Big(\nabla (\vepjt)'(t,y^{\epjt}(t))-\Big(\begin{array}{ccc}1&0&0\\0&1&0\end{array}\Big)\Big)-(0|\nabla' v^0(t))\Big\|_{L^2(\Omega;\mthree)}\to 0.$$
On the other hand, by \eqref{gradinv0} there holds
\begin{eqnarray*}
&&\Big\|\frac{\nabla \vepjt_3(t,y^{\epjt}(t))-e_3}{\epjt}+\Big(\begin{array}{c}\nabla'v^0\\0\end{array}\Big)\Big\|_{L^2(\Omega;\R^3)}\leq C\Big\|\nabla (\vepjt)'(t)-\Big(\begin{array}{ccc}1&0&0\\0&1&0\end{array}\Big)\Big\|_{L^{\infty}(\Omega;\M^{2\times 3})}\\
&&+\|\nabla' v^0(t, (\vepjt)'(t, y^{\epjt}(t)))-\nabla' v^0(t)\|_{L^2(\Omega;\R^2)}
\end{eqnarray*}
which tends to zero owing to \eqref{lp1}, \eqref{distinv1}, \eqref{distinv13}, \eqref{gradinv}, \eqref{cvitsb} and the dominated convergence theorem. Therefore, the proof of claim \eqref{fclhp} is completed.

Now, by \eqref{cvitsb}, \eqref{fclhp} and the dominated convergence theorem we conclude that
\begin{eqnarray}
\nonumber &&\frac{1}{\epjt}\Big(\begin{array}{cc}0&-\dot{\theta}^{\epjt}\big(\frac{\zepjt_3(t)}{\epjt}\big)\nabla' \dot{v}^0(t,\bzepjt(t))\\(\nabla' \dot{v}^0(t,\bzepjt(t)))^T&0\end{array}\Big)\Big(\nabla \vepjt(t,y^{\epjt}(t))-Id\Big)\\
&&\label{term15} \to \Big(\begin{array}{cc}\nabla' \dot{v}^0(t)\otimes \nabla' {v}^0(t)&0\\0&\frac{d}{dt}\frac{|\nabla' v^0(t)|^2}{2}\end{array}\Big)
\end{eqnarray}
strongly in $L^2(\Omega;\mthree)$. By combining \eqref{term12}, \eqref{term13} and \eqref{term15} we deduce \eqref{lhtg}. Now, by \eqref{nothird}, \eqref{lp1}, \eqref{badsetstressnt}, \eqref{linftybdwk}, \eqref{goodsetstress} (which still hold true for $\alpha=3$), \eqref{strlima3} and \eqref{lhtg} we obtain
\begin{equation}
\label{A1}
\Theta(t)=\intom{E(t):\Big(\begin{array}{cc}\nabla' \dot{u}^0(t)+\nabla' \dot{v}^0(t)\otimes \nabla' {v}^0(t)-(x_3+v(t)-v^0(t))(\nabla')^2 \dot{v}^0(t)&0\\0&0\end{array}\Big)}.
\end{equation}
On the other hand,
\begin{eqnarray*}
&&\sym(\nabla' \dot{v}^0(t)\otimes \nabla' {v}^0(t)-(v(t)-v^0(t))(\nabla')^2 \dot{v}^0(t))\\
&&=-\sym\,\nabla' \big((v(t)-v^0(t))\nabla' \dot{v}^0(t)\big)+\sym \big(\nabla' v(t)\otimes \nabla'\dot{v}^0(t)\big)
\end{eqnarray*}
and
\begin{equation}
\label{A2}
\intom{\C_2 E(t):\nabla' \big((v(t)-v^0(t))\nabla' \dot{v}^0(t)\big)}=0
\end{equation}
by Remark \ref{eul3}. By combining \eqref{A1} and \eqref{A2}, the proof of \eqref{identtaue} and of the theorem is complete.
\end{proof}

To conclude this section we show some corollaries of Theorem \ref{cvstress}. We first prove that under the hypotheses of the theorem we can deduce convergence of the elastic energy and of the hardening functional. More precisely, the following result holds true.

\begin{cor}
Under the assumptions of Theorem \ref{cvstress}, for $\alpha>3$ for every $t\in [0,T]$, setting $y^{\ep}(t):=\pep(t,\zep(t))$ there holds
\begin{equation}
\nonumber
\lim_{\ep\to 0}\frac{1}{\ep^{2\alpha-2}}\intom{W_{el}(\nep y^{\ep}(t)(P^{\ep})^{-1}(t))}= \intom{Q_2(\sym\,\nabla' u(t)-x_3(\nabla')^2 v(t)-p'(t))},
\end{equation}
and
\begin{equation}
\label{cvs2}
\lim_{\ep\to 0}\frac{1}{\ep^{2\alpha-2}}\intom{W_{hard}(P^{\ep}(t))}= \intom{\B{p(t)}}.
\end{equation}
The analogous result holds true for $\alpha=3$ on the $t$-dependent subsequence $\epjt\to 0$ selected in Theorem \ref{cvstress}.
\end{cor}
 \begin{proof}
 The result follows by combining the liminf inequalities in Theorem \ref{liminfineq}, the $\ep$-energy balance (qs2) and the reduced energy balance (qs1$_{r\alpha}$). 
 \end{proof}
 In particular, we can deduce strong convergence of the sequence of scaled plastic strains by the convergence of the energies.
 \begin{cor}
 \label{sp}
 Under the hypotheses of Theorem \ref{cvstress}, {for $\alpha>3$} there holds
 \begin{equation}
 \label{scpt}
 p^{\ep}(t)\to p(t)\quad\text{strongly in }L^2(\Omega;\mthree)
 \end{equation}
 for every $t\in [0,T]$. { The analogous result holds true for $\alpha=3$ on the $t$-dependent subsequence $\epjt\to 0$ selected in Theorem \ref{cvstress}.}
 \end{cor}
 \begin{proof}
 Fix $\delta>0$ and let $c_h(\delta)$ be the constant in \eqref{prh4}. By \eqref{prh4} there holds
 \begin{equation}
 \label{tay}
 W_{hard}(Id+F)\geq \B{F}-C\delta|F|^2\quad\text{for every }F\in\mthree,\,|F|<c_h(\delta).
 \end{equation}
  
 Fix $t\in [0,T]$ and for every $\ep$ consider the set
 $$S_{\ep}(t):=\Big\{x\in\Omega: |p^{\ep}(t,x)|<\frac{c_h(\delta)}{\ep}\Big\}.$$
 Denoting by $\mu_{\ep}(t)$ the characteristic function of the set $S_{\ep}(t)$, by \eqref{cvpt} and Chebychev inequality,
 \begin{equation}
 \label{iceq0}
 {\mu_{\ep}(t)}\to 1\quad\text{boundedly in measure as }\ep\to 0.
 \end{equation}
 and thus
 \begin{equation}
 \label{wcrset}
 \mu_{\ep}(t)p^{\ep}(t)\deb p(t)\quad\text{weakly in }L^2(\Omega;\mthree).
 \end{equation}
 We remark that in the set $S_{\ep}(t)$ we have $\ep^{\alpha-1}|p^{\ep}(t)|<\ep^{\alpha-2}c_h(\delta).$
 Hence, by \eqref{tay} for $\ep$ small enough there holds
 \begin{eqnarray*}
 \frac{1}{\ep^{2\alpha-2}}W_{hard}(P^{\ep}(t))\geq  \frac{1}{\ep^{2\alpha-2}}\mu_{\ep}(t)W_{hard}(P^{\ep}(t))
 \geq \mu_{\ep}(t)\big(\B{p^{\ep}(t)}-C\delta|p^{\ep}(t)|^2\big).
 \end{eqnarray*}
  In particular, by \eqref{cvpt}, \eqref{cvs2} and the lower semicontinuity of $B$ with respect to weak $L^2$ convergence, we have
 \begin{eqnarray*}
 &&\intom{\B{p(t)}}= \lim_{\ep\to 0}\frac{1}{\ep^{2\alpha-2}}\intom{W_{hard}(P^{\ep}(t))}\geq\limsup_{\ep\to 0}\frac{1}{\ep^{2\alpha-2}}\intom{\mu_{\ep}(t)W_{hard}(P^{\ep}(t))}\\ 
 &&\geq  \limsup_{\ep\to 0}\intom{\mu_{\ep}(t)\B{p^{\ep}(t)}}-C\delta
\geq  \liminf_{\ep\to 0}\intom{\mu_{\ep}(t)\B{p^{\ep}(t)}}-C\delta\geq \intom{\B{p(t)}}-C\delta.
 \end{eqnarray*}
Since $\delta$ is arbitrary, we obtain
 \begin{eqnarray}
&& \label{iceq1}\lim_{\ep\to 0}\intom{\mu_{\ep}(t)\B{p^{\ep}(t)}}= \intom{\B{p(t)}}
\end{eqnarray}
and  by \eqref{cvs2}
\begin{equation}
\label{iceq1bis}
\lim_{\ep\to 0}\frac{1}{\ep^{2\alpha-2}}\intom{(1-\mu_{\ep}(t))W_{hard}(P^{\ep}(t))}=0.
\end{equation}
By \eqref{prh3} and \eqref{iceq1bis} we deduce
\begin{equation}
\label{iceq2}
\lim_{\ep\to 0}\intom{(1-\mu_{\ep}(t))|p^{\ep}(t)|^2}\leq \frac{2}{c_6}\lim_{\ep\to 0}\frac{1}{\ep^{2\alpha-2}}\intom{(1-\mu_{\ep}(t))W_{hard}(P^{\ep}(t))}=0.
\end{equation}
Hence, by \eqref{grbelowh} there holds
\begin{eqnarray}
\nonumber \intom{|p^{\ep}(t)-p(t)|^2}&=&\intom{\mu_{\ep}(t)|p^{\ep}(t)-p(t)|^2}+\intom{(1-\mu_{\ep}(t))|p^{\ep}(t)-p(t)|^2}\\
\nonumber&\leq& \frac{2}{c_6}\intom{\mu_{\ep}(t)\B{p^{\ep}(t)-p(t)}}+ 2\intom{(1-\mu_{\ep}(t))(|p^{\ep}(t)|^2+|p(t)|^2)}.\\
\label{iceq3}
\end{eqnarray}
Recalling the quadratic structure of $B$, the first term in the second row of \eqref{iceq3} can be decomposed as
\begin{eqnarray*}
\frac{2}{c_6}\intom{\mu_{\ep}(t)\B{p^{\ep}(t)-p(t)}}&=&\frac{2}{c_6}\intom{\mu_{\ep}(t)\B{p^{\ep}(t)}}+\frac{2}{c_6}\intom{\mu_{\ep}(t)\B{p(t)}}\\
&-&\frac{4}{c_6}\intom{\mu_{\ep}(t)\mathbb{B}p^{\ep}(t):p(t)}
\end{eqnarray*}
and tends to zero due to \eqref{iceq0}--\eqref{iceq1}.
On the other hand, by \eqref{iceq0} and \eqref{iceq2} 
$$\intom{(1-\mu_{\ep}(t))(|p^{\ep}(t)|^2+|p(t)|^2)}\to 0.$$
By combining the previous results, we deduce \eqref{scpt}.
\end{proof}
Convergence of the energy implies also strong convergence of the in-plane displacements. More precisely, the following result holds true.
\begin{cor}
\label{su}
Under the assumptions of Theorem \ref{cvstress}, for $\alpha>3$, for every $t\in [0,T]$ there holds
\begin{equation}
\label{scut}
u^{\ep}(t)\to u(t)\quad\text{strongly in }W^{1,2}(\omega;\R^2).
\end{equation} 
The same result holds true for $\alpha=3$, on the $t$-dependent subsequence $\epjt\to 0$ selected in Theorem \ref{cvstress}.
\end{cor}
\begin{proof}
We prove the corollary for $\alpha>3$. The case where $\alpha=3$ follows by simple adaptations. Fix $t\in [0,T]$ and let $F^{\ep}(t)$ be the map defined in \eqref{deffept}. Fix $\delta>0$ and consider the set
$$U_{\ep}(t):=\Big\{x\in\Omega: |F^{\ep}(t,x)|<\frac{c_{el}(\delta)}{\ep}\Big\},$$
where $c_{el}(\delta)$ is the constant in \eqref{quadrwel}.
In particular, in the set $U_{\ep}(t)$ there holds $\ep^{\alpha-1}|F^{\ep}(t)|\leq \ep^{\alpha-2}c_{el}(\delta)$. Hence, denoting by $\mu_{\ep}(t)$ the characteristic function of $U_{\ep}(t)$, by (H3), \eqref{quadrwel} and \eqref{decompep}, we have
\begin{eqnarray*}
\frac{1}{\ep^{2\alpha-2}}W_{el}(\nep y^{\ep}(t)(P^{\ep})^{-1}(t))=\frac{1}{\ep^{2\alpha-2}}W_{el}(Id+\ep^{\alpha-1}F^{\ep}(t))\geq \mu_{\ep}(t){Q(F^{\ep}(t))}-\mu_{\ep}(t)C\delta|F^{\ep}(t)|^2.
\end{eqnarray*} 
By Chebychev inequality and \eqref{unifbdfep},
\begin{equation}
\label{nnpmeas}
\mu_{\ep}(t)\to 1\quad\text{boundedly in measure,}
\end{equation}
whereas by \eqref{convfep} and \eqref{nnpserve},
\begin{equation}
\label{nnpfep}
\mu_{\ep}(t)\sym\,F^{\ep}(t)\deb \A(\sym \,\nabla' u(t)-x_3(\nabla')^2 v(t)-p'(t))\quad\text{weakly in }L^2(\Omega;\mthree).
\end{equation}
Arguing as in the proof of \eqref{iceq1} we obtain
\begin{eqnarray}
 \label{iceq11}&& \lim_{\ep\to 0}\intom{\mu_{\ep}(t)Q(F^{\ep}(t))}=\intom{Q_2(\sym\,\nabla' u(t)-x_3(\nabla')^2 v(t)-p'(t))}
\end{eqnarray}
 and
 \begin{eqnarray}
\nonumber\lim_{\ep \to 0} \frac{1}{\ep^{2\alpha-2}}\intom{(1-\mu_{\ep}(t))W_{el}(Id+\ep^{\alpha-1}F^{\ep}(t))}= 0.
\end{eqnarray}
By (H4), this implies that
\begin{equation}
\label{iceq12bis}
\lim_{\ep \to 0} \frac{1}{\ep^{2\alpha-2}}\intom{(1-\mu_{\ep}(t))\dist^2(Id+\ep^{\alpha-1}F^{\ep}(t), SO(3))}\to 0.
\end{equation}
On the other hand, \eqref{growthcondQ} and \eqref{nothird} yield
\begin{eqnarray*}
&&\intom{\big|\mu_{\ep}(t)\sym\, F^{\ep}(t)-  \A\big(\nabla' u(t)-x_3(\nabla')^2 v(t)-p'(t)\big)\big|^2}\\
&&\leq \frac{1}{r_{\C}}\intom{Q\big(\mu_{\ep}(t)\sym\, F^{\ep}(t)-  \A\big(\nabla' u(t)-x_3(\nabla')^2 v(t)-p'(t)\big)\big)}\\
&&=\frac{1}{r_{\C}}\intom{Q(\mu_{\ep}(t) F^{\ep}(t))}+\frac{1}{r_{\C}}\intom{Q_2(\nabla' u(t)-x_3(\nabla')^2 v(t)-p'(t))}\\
&&-\frac{2}{r_{\C}}\intom{\mu_{\ep}(t)\C_2F^{\ep}(t):(\nabla' u(t)-x_3(\nabla')^2 v(t)-p'(t))}.
\end{eqnarray*}
Hence, by \eqref{nnpfep} and \eqref{iceq11}
\begin{equation}
\label{iceq14}
\mu_{\ep}(t)\sym\, F^{\ep}(t)\to  \A(\nabla' u(t)-x_3(\nabla')^2 v(t)-p'(t))\quad\text{strongly in }L^2(\Omega;\mthree).
\end{equation}
 Moreover,
\begin{eqnarray}
&&\nonumber\frac{1}{\ep^{\alpha-1}}\mu_{\ep}(t) \dist(Id+\ep^{\alpha-1}F^{\ep}(t), SO(3))\\
&&\nonumber=\mu_{\ep}(t)|\sym\, F^{\ep}(t)|+\mu_{\ep}(t)O(\ep^{\alpha-1}|F^{\ep}(t)|^2)\to | \A(\nabla' u(t)-x_3(\nabla')^2 v(t)-p'(t))|\\
\label{iceq13}
\end{eqnarray}
strongly in $L^2(\Omega)$. By combining \eqref{iceq12bis} and \eqref{iceq13} we deduce
$$\frac{1}{\ep^{\alpha-1}}\dist(Id+\ep^{\alpha-1}F^{\ep}(t), SO(3))\to | \A(\nabla' u(t)-x_3(\nabla')^2 v(t)-p'(t))|$$
strongly in $L^2(\Omega)$. In particular, the sequence $\frac{1}{\ep^{2\alpha-2}}\dist^2(Id+\ep^{\alpha-1}F^{\ep}(t), SO(3))$ is equi-integrable.

Now, recalling that by \eqref{deffept} there holds
$$Id+\ep^{\alpha-1}F^{\ep}(t)=(Id+\ep^{\alpha-1}G^{\ep}(t))(Id+\ep^{\alpha-1}p^{\ep}(t))^{-1},$$
by \eqref{prk1} and \eqref{isink} for every $R\in SO(3)$ we deduce
\begin{eqnarray*}
&&\frac{1}{\ep^{2\alpha-2}}|Id+\ep^{\alpha-1}G^{\ep}(t)-R|^2=\frac{1}{\ep^{2\alpha-2}}|(Id+\ep^{\alpha-1}F^{\ep}(t))(Id+\ep^{\alpha-1}p^{\ep}(t))-R|^2\\
&&\leq \frac{c_k^2}{\ep^{2\alpha-2}}|Id+\ep^{\alpha-1}F^{\ep}(t)-R|^2+|p^{\ep}(t)|^2,
\end{eqnarray*}
which in turn implies
\begin{eqnarray*}
\frac{1}{\ep^{2\alpha-2}}\dist^2(Id+\ep^{\alpha-1}G^{\ep}(t), SO(3))\leq \frac{c_k^2}{\ep^{2\alpha-2}}\dist^2(Id+\ep^{\alpha-1}F^{\ep}(t), SO(3))+|p^{\ep}(t)|^2.
\end{eqnarray*}
Hence, by \eqref{scpt} $\frac{1}{\ep^{2\alpha-2}}\dist^2(Id+\ep^{\alpha-1}G^{\ep}(t), SO(3))$ is equi-integrable. Arguing as in \cite[Section 7.2, Proof of Theorem 2]{FJM2} we obtain the equi-integrability of $|G^{\ep}(t)|^2$. 

We claim that also $|F^{\ep}(t)|^2$ is equi-integrable. Indeed, by \eqref{deffept}, there holds
$$|F^{\ep}(t)|^2\leq C(|G^{\ep}(t)|^2+|w^{\ep}(t)|^2+|p^{\ep}(t)|^2+\ep^{2\alpha-2}|G^{\ep}(t)w^{\ep}(t)|^2+\ep^{2\alpha-2}|G^{\ep}(t)p^{\ep}(t)|^2).$$
Now, by \eqref{unifapest}, \eqref{isink} and \eqref{numerarepag35}, we have 
$$|w^{\ep}(t)|^2\leq c_K^2\ep^{2\alpha-2}|p^{\ep}(t)|^4\leq C|p^{\ep}(t)|^2.$$
Hence, by \eqref{scpt} the maps $|w^{\ep}(t)|^2$ are equi-integrable. Moreover, by \eqref{unifapest} there holds
$$\ep^{2\alpha-2}|G^{\ep}(t)p^{\ep}(t)|^2\leq C|G^{\ep}(t)|^2$$
and by \eqref{linftybdw}
$$\ep^{2\alpha-2}|G^{\ep}(t)w^{\ep}(t)|^2\leq C|G^{\ep}(t)|^2.$$
Therefore, the equi-integrability of $|F^{\ep}(t)|^2$ follows from the equi-integrability of $|G^{\ep}(t)|^2$. 

By \eqref{iceq14}, this implies that
$$\sym\, F^{\ep}(t)\to \A(\nabla' u(t)-x_3(\nabla')^2 v(t)-p'(t))$$
strongly in $L^2(\Omega;\mthree)$. On the other hand, by \eqref{numerarepag352} and \eqref{numerarepag39},
$$w^{\ep}(t)-\ep^{\alpha-1}G^{\ep}(t)(p^{\ep}(t)-w^{\ep}(t))\to 0$$
strongly in $L^1(\Omega;\mthree)$. Therefore, by \eqref{deffept} and \eqref{scpt} we obtain
{\begin{eqnarray*}
\sym\, G^{\ep}(t)\to \A(\nabla' u(t)-x_3(\nabla')^2 v(t)-p'(t))+p(t)\quad\text{strongly in }L^1(\Omega;\mthree).
\end{eqnarray*}}
By the equi-integrability of $|G^{\ep}(t)|^2$, it follows that {$$\sym\, G^{\ep}(t)\to \A(\nabla' u(t)-x_3(\nabla')^2 v(t)-p'(t))+p(t)\quad\text{strongly in }L^2(\Omega;\mthree).$$}
The conclusion follows then arguing as in \cite[Section 7.2, Proof of Theorem 2]{FJM2}.
\end{proof}
\section{Convergence of approximate minimizers}
\label{appr}
Theorems \ref{cvstress} is actually only  a convergence result. Indeed, under our assumptions the existence of an $\ep$-quasistatic evolution according to Definition \ref{epquasevol} is not guaranteed. Howewer, following the same approach as in \cite[Theorem 2.3]{MS}, we can extend our convergence result to sequences of approximate discrete-time ${\ep}$-quasistatic evolutions. More precisely, setting 
{\begin{eqnarray*}
\cal{A}_{\ep}&:=&\{(z,P)\in W^{1,2}(\Omega;\R^3)\times L^2(\Omega;SL(3)):\\
&&z=(x',\ep x_3)\quad\cal{H}^2\text{ - a.e. on }\Gamma_d\quad\text{ and }P(x)\in K\quad\text{a.e. in }\Omega\},
\end{eqnarray*}}  
we give the following definition.
\begin{defin}
\label{apmin}
Given a sequence of time-partitions
$$\{0=t^0_{\ep}<t^1_{\ep}<\cdots t^{N^{\ep}}_{\ep}=T\},$$
with time-steps
\begin{equation}
\label{ss49}\tau_{\ep}:=\max_{i=1,\cdots N^{\ep}}(t^{i}_{\ep}-t^{i-1}_{\ep})\to 0\quad\text{as }\ep\to 0,
\end{equation}
and a sequence of positive parameters $\delta_{\ep}\to 0$, we call $\{(z^i_{\ep}, P^i_{\ep})\}$ a sequence of \emph{approximate minimizers} if, for every $\ep>0$, $(z^0_{\ep}, P^0_{\ep})\in\cal{A}_{\ep}$, and $(z^i_{\ep},P^i_{\ep})\in \cal{A}_{\ep}$ satisfies
\begin{eqnarray}
\nonumber&&\cal{F}_{\ep}(t^i_{\ep},z^i_{\ep},P^i_{\ep})+{\ep^{\alpha-1}}\intom{D(P^{i-1}_{\ep},P^i_{\ep})}\\
\label{mind}
&&\leq \ep^{2\alpha-2}\delta_{\ep}(t^i_{\ep}-t^{i-1}_{\ep})+\inf_{(z,P)\in\cal{A}_{\ep}}\Big\{\cal{F}_{\ep}(t^i_{\ep},z,P)+{\ep^{\alpha-1}}\intom{D(P^{i-1}_{\ep},P)}\Big\}
\end{eqnarray}
for every $i=1,\cdots, N^{\ep}$.

\end{defin}
Our final result is to show that every sequence of approximate minimizers converges, as $\ep\to 0$, to a reduced quasistatic evolution.
{
\begin{teo}
\label{cvapp}
Let $\alpha\geq 3$. Assume that $t\mapsto u^0(t)$ belongs to $C^1([0,T];W^{1,\infty}(\R^2;\R^2)\cap C^{1}(\R^2;\R^2))$ and $t\mapsto v^0(t)$ belongs to $C^1([0,T];W^{2,\infty}(\R^2)\cap C^{2}(\R^2))$, respectively. For every $t\in [0,T]$, let $\pep(t)$ be defined as in \eqref{defphiep} and let $(\mathring{u},\mathring{v},\mathring{p})\in \cal{A}(u^0(0),v^0(0))$ be such that 
\begin{eqnarray}
\nonumber &&\intom{Q_2(\sym\nabla' \mathring{u}-x_3(\nabla')^2 \mathring{v}+\tfrac{L_{\alpha}}{2}\nabla' \mathring{v}\otimes \nabla' \mathring{v}-\mathring{p}')}+\intom{\B{\mathring{p}}}\\
\nonumber&&
\leq \intomm{Q_2(\nabla' \hat{u}-x_3(\nabla')^2 \hat{v}+\tfrac{L_{\alpha}}{2}\nabla' \hat{v}\otimes \nabla' \hat{v}-\hat{p}')}+\intom{\B{\hat{p}}}+\intom{H_D(\hat{p}-\mathring{p})},\\
\label{servedopo1}
\end{eqnarray} 
for every $(\hat{u},\hat{v},\hat{p})\in\cal{A}(u^0(0),v^0(0))$. 
Given a sequence of time-partitions
$$\{0=t^0_{\ep}<t^1_{\ep}<\cdots t^{N^{\ep}}_{\ep}=T\},$$
with time-steps
\begin{equation}
\nonumber \tau_{\ep}:=\max_{i=1,\cdots N^{\ep}}(t^{i}_{\ep}-t^{i-1}_{\ep})\to 0\quad\text{as }\ep\to 0,
\end{equation}
and a sequence of positive parameters $\delta_{\ep}\to 0$, assume there exists a sequence of pairs $(y_0^{\ep},P_0^{\ep})\in \cal{A}_{\ep}(\pep(0))$ such that
\begin{equation}
\label{appminimg}\cal{I}(y^{\ep}_0,P^{\ep}_0) \leq \cal{I}(\hat{y},\hat{P})+{\ep^{\alpha-1}}\intom{D(P^{\ep}_0,\hat{P})}+\delta_{\ep}\tau_{\ep}\ep^{2\alpha-2},
\end{equation}
for every $(\hat{y},\hat{P})\in\cal{A}_{\ep}(\pep(0))$, and
\begin{eqnarray}
\label{convu0d} && u^{\ep}_0:=\frac{1}{\ep^{\alpha-1}}\intt{\big((y^{\ep}_0)'-x'\big)}\to \mathring{u}\quad\text{strongly in }W^{1,2}(\omega;\R^2),\\
\label{convv0d} && v^{\ep}_0:=\frac{1}{\ep^{\alpha-2}}\intt{(y^{\ep}_0)_3}\to \mathring{v}\quad\text{strongly in }W^{1,2}(\omega),\\
\label{convP0d} && p^{\ep}_0:=\frac{P^{\ep}_0-Id}{\ep^{\alpha-1}}\to \mathring{p}\quad\text{strongly in }L^2(\Omega;\md),\\
 \nonumber && \lim_{\ep\to 0}\,\frac{1}{\ep^{2\alpha-2}}\cal{I}(y^{\ep}_0, P^{\ep}_0) =\intom{Q_2(\sym \nabla' \mathring{u}-x_3(\nabla')^2 \mathring{v}+\tfrac{L_{\alpha}}{2}\nabla' \mathring{v}\otimes\nabla' \mathring{v}-{\mathring{p}}')}\\
\label{convE0d} &&+\intom{\B{\mathring{p}}}.
\end{eqnarray}

Let $(z^i_{\ep}, P^{i}_{\ep})$ be a sequence of approximate minimizers and let $(\overline{z}^{\ep}(t), \overline{P}^{\ep}(t))$ be the corresponding right-continuous, piecewise constant interpolants on the time partitions. Let $\overline{\phi}^{\ep}(t)$ be the associated interpolant of $t\mapsto\pep(t)$. Then, for every $t\in [0,T]$ 
\begin{equation}
\nonumber
\overline{p}^{\ep}(t):=\frac{\overline{P}^{\ep}(t)-Id}{\ep^{\alpha-1}}\deb p(t)\quad\text{weakly in }L^2(\Omega;\mthree).
\end{equation}
Moreover, for $\alpha>3$, for every $t\in [0,T]$ the following convergence properties hold true:
\begin{eqnarray*}
&&\overline{u}^{\ep}(t):=\frac{1}{\ep^{\alpha-1}}\intt{\big((\overline{\phi}^{\ep})'(t,\overline{z}^{\ep}(t))-x'\big)}\deb u(t)\quad\text{weakly in }W^{1,2}(\omega;\R^2),\\
&&\overline{v}^{\ep}(t):=\frac{1}{\ep^{\alpha-2}}\intt{\overline{\phi}^{\ep}_3(t, \overline{z}^{\ep}(t))}\to v(t)\quad\text{strongly in }W^{1,2}(\omega),
\end{eqnarray*}
where $t\mapsto (u(t),v(t),p(t))$ is a reduced quasistatic evolution.

For $\alpha=3$, up to extracting a $t$-dependent subsequence $\epjt\to 0$, there holds
  \begin{eqnarray*}
&&\overline{u}^{\epjt}(t):=\frac{1}{\epjt^{\alpha-1}}\intt{\big((\overline{\phi}^{\epjt})'(t,\overline{z}^{\epjt}(t))-x'\big)}\deb u(t)\quad\text{weakly in }W^{1,2}(\omega;\R^2),\\
&&\overline{v}^{\epjt}(t):=\frac{1}{\epjt^{\alpha-2}}\intt{\overline{\phi}^{\epjt}_3(t, \overline{z}^{\epjt}(t))}\to v(t)\quad\text{strongly in }W^{1,2}(\omega),
\end{eqnarray*}
where $t\mapsto (u(t),v(t),p(t))$ is a reduced quasistatic evolution.
\end{teo}
\begin{oss}
The set of admissible data $(\mathring{u},\mathring{v},\mathring{p})$ for Theorem \ref{cvapp} is nonempty. 

Indeed, for every $\ep>0$ let $(y^{\ep}_0,P^{\ep}_0)\in\cal{A}_{\ep}(\pep(0))$ be such that
$$\cal{I}(y^{\ep}_0,P^{\ep}_0)+\ep^{\alpha-1}\intom{D(Id,P^{\ep}_0)}\leq \inf_{(\hat{y},\hat{P})\in\cal{A}_{\ep}(\pep(0))}\Big\{\cal{I}(\hat{y},\hat{P})+\ep^{\alpha-1}\intom{D(Id,\hat{P})}\Big\}+\delta_{\ep}\tau_{\ep}\ep^{2\alpha-2}.$$
Since by \eqref{triang}
$$D(Id,\hat{P})\leq D(Id,P^{\ep}_0)+D(P^{\ep}_0,\hat{P}),$$
we deduce that $(y^{\ep}_0,P^{\ep}_0)$ fulfills \eqref{appminimg}. By the regularity of $\partial\omega$, the set $\gamma_d$ coincides $\cal{H}^1$ - a.e. with its closure in the relative topology of $\partial\omega$, which in turn is a closed (nontrivial) interval in $\partial\omega$.
Hence, by \cite[Theorem 5.1]{D1}, choosing $p^{\ep,0}=p^0=0$ for every $\ep>0$, and $s_{\ep}=\delta_{\ep}\tau_{\ep}\ep^{2\alpha-2}$, we infer the existence of a triple $(\mathring{u},\mathring{v},\mathring{p})\in\cal{A}(u^0(0),v^0(0))$ such that \eqref{servedopo1} is satisfied and \eqref{convu0d}--\eqref{convE0d} hold true.
\end{oss}}
\begin{proof}[Proof of Theorem \ref{cvapp}]
The proof follows along the general lines of the proof of Theorems \ref{cvstress}. We sketch the main steps in the case $\alpha>3$. The case $\alpha=3$ follows by straightforward adaptations.\\
\emph{Quasi-stability condition}\\
By \eqref{triang} the piecewise constant interpolants fullfill
\begin{equation}
\label{quasst}
\cal{F}_{\ep}(t, \overline{z}^{\ep}(t),\overline{P}^{\ep}(t))\leq \cal{F}_{\ep}(t,\hat{z},\hat{P})+{\ep^{\alpha-1}}\intom{D(\overline{P}^{\ep}(t),\hat{P})}+\delta_{\ep}\tau_{\ep}\ep^{2\alpha-2}
\end{equation}
for every $(\hat{z},\hat{P})\in\cal{A}_{\ep}$. The previous inequality will play the role of the $\ep$-stability condition (qs1). \\
\emph{Discrete energy inequality}\\
To adapt the proof of Theorem \ref{cvstress} we shall need an analogous of condition (qs2). To this purpose, we notice that, by \eqref{mind} the following discrete energy inequality holds true
\begin{eqnarray*}
&&\cal{F}_{\ep}(t^i_{\ep},z^i_{\ep},P^i_{\ep})+{\ep^{\alpha-1}}\intom{D(P^{i-1}_{\ep},P^i_{\ep})}\leq \ep^{2\alpha-2}\delta_{\ep}(t^i_{\ep}-t^{i-1}_\ep)+\cal{F}_{\ep}(t^i_{\ep}, z^{i-1}_{\ep},P^{i-1}_{\ep})\\
&&=\ep^{2\alpha-2}\delta_{\ep}(t^i_{\ep}-t^{i-1}_\ep)+\cal{F}_{\ep}(t^{i-1}_{\ep}, z^{i-1}_{\ep},P^{i-1}_{\ep})+\int_{t^{i-1}_{\ep}}^{t^i_{\ep}}{\partial_s \cal{F}_{\ep}(s, z^{i-1}_{\ep}, P^{i-1}_{\ep})\,ds}\\
&&=\ep^{2\alpha-2}\delta_{\ep}(t^i_{\ep}-t^{i-1}_\ep)+\cal{F}_{\ep}(t^{i-1}_{\ep}, z^{i-1}_{\ep},P^{i-1}_{\ep})\\
&&+{\ep^{2\alpha-2}}\int_{t^{i-1}_{\ep}}^{t^i_{\ep}}{\intom{DW_{el}\big(\nabla \pep(s,z^{i-1}_{\ep})\nep z^{i-1}_{\ep}(P^{i-1}_{\ep})^{-1}\big):\nabla \dot{\pep}(s,z^{i-1}_{\ep})\nep z^{i-1}_{\ep}(P^{i-1}_{\ep})^{-1}}\,ds}\\
&&=\ep^{2\alpha-2}\delta_{\ep}(t^i_{\ep}-t^{i-1}_{\ep})+\cal{F}_{\ep}(t^{i-1}_{\ep}, z^{i-1}_{\ep},P^{i-1}_{\ep})\\
&&+{\ep^{\alpha-1}}\int_{t^{i-1}_{\ep}}^{t^i_{\ep}}{\intom{E_{\ep}^{i-1}(s):\nabla \dot{\pep}(s,z^{i-1}_{\ep})(\nabla \pep)^{-1}(s,z^{i-1}_{\ep})}\,ds},
\end{eqnarray*}
where
$$E_{\ep}^{i-1}(s):=\frac{1}{\ep^{\alpha-1}}DW_{el}\big(\nabla \pep(s,z^{i-1}_{\ep})\nep z^{i-1}_{\ep}(P^{i-1}_{\ep})^{-1}\big)\big(\nabla \pep(s,z^{i-1}_{\ep})\nep z^{i-1}_{\ep}(P^{i-1}_{\ep})^{-1}\big)^T$$
for every $s\in [t^{i-1}_{\ep}, t^i_{\ep}]$. 

By iterating the discrete energy inequality, recalling that $\overline{P}^{\ep}(t)$ is locally constant, we obtain
\begin{eqnarray}
&&\nonumber\cal{F}_{\ep}(t,\overline{z}^{\ep}(t),\overline{P}^{\ep}(t))+{\ep^{\alpha-1}}\cal{D}(\overline{P}^{\ep};0,t)\\
&&\nonumber\leq \ep^{2\alpha-2}\delta_{\ep}T +\cal{F}_{\ep}(0, \zep_0, P^{\ep}_0)+{\ep^{\alpha-1}}\int_0^t{\intom{\overline{E}^{\ep}(s):\nabla \dot{\pep}(s,\overline{z}^{\ep}(s))(\nabla \pep)^{-1}(s,\overline{z}^{\ep}(s))}\,ds},\\
\label{notenough}
\end{eqnarray} 
where $\zep_0:=\vep(0,y^{\ep}_0)$ and
$$\overline{E}^{\ep}(s):=\frac{1}{\ep^{\alpha-1}}DW_{el}\big(\nabla \pep(s,\overline{z}^{\ep}(s))\nep\overline{z}^{\ep}(s)(\overline{P}^{\ep})^{-1}(s)\big)\big(\nabla \pep(s,\overline{z}^{\ep}(s))\nep \overline{z}^{\ep}(s)(\overline{P}^{\ep})^{-1}(s)\big)^T$$
for every $s\in [0,t]$.\\
\emph{Proof of the reduced stability condition and energy balance}\\
The reduced stability condition can be deduced as in Step 2 of the proof of Theorem \ref{cvstress}. Moreover, arguing as in the proof of Theorem \ref{cvstress} one can show that $\overline{E}^{\ep}(t)$ converges in the sense of \eqref{badsetstressnt} and \eqref{goodsetstress} to a limit stress $E(t)$ such that
$$E(t)=\C(G(t)-p(t)).$$
The crucial step to deduce the reduced energy balance is to show that $E(t)e_3=0$ a.e. in $\Omega$, that is,
\begin{equation}
\label{minimality}
E(t)=\C_2(G'(t)-p'(t)).
\end{equation}
The main difference with respect to Theorem \ref{cvstress} is that in this case we can not deduce this condition starting from the three-dimensional Euler-Lagrange equations because \eqref{notenough} does not imply \eqref{euler}. 

To cope with this problem, set $\overline{y}^{\ep}(t)=\overline{\phi}^{\ep}(t,\overline{z}^{\ep}(t))$ for every $t\in [0,T]$. Let {$\eta\in W^{1,\infty}(\R^3;\R^3)\cap C^{\infty}(\R^3;\R^3)$ be such that $\eta=0\quad\cal{H}^2\text{ - a.e. on }\Gamma_d$}. We consider variations of the form 
$$\hat{y}=\overline{y}^{\ep}(t)+\tau_{\ep}\ep^{\alpha-1}\eta^{\ep}\circ \overline{y^{\ep}},$$
where $\eta^{\ep}$ is the test function considered in Step 4 of the proof of Theorem \ref{cvstress}. By \eqref{quasst}, taking $\hat{P}=\overline{P}^{\ep}(t)$, we deduce
\begin{eqnarray*}
-\delta_{\ep}&&\leq\frac{1}{\ep^{\alpha-1}}\intom{\frac{W_{el}\Big(\Big(Id+\tau_{\ep}\ep^{\alpha-1}\nabla \eta^{\ep}(\overline{y}^{\ep}(t))\Big)\nep \overline{y}^{\ep}(t)(\overline{P}^{\ep})^{-1}(t)\Big)-W_{el}(\nep \overline{y}^{\ep}(t)(\overline{P}^{\ep})^{-1}(t))}{\tau_{\ep}\ep^{\alpha-1}}}\\
&&=\frac{1}{\ep^{\alpha-1}}\intom{\int_0^1{\frac{d}{ds}\frac{W_{el}\Big(\Big(Id+s\tau_{\ep}\ep^{\alpha-1}\nabla \eta^{\ep}(\overline{y}^{\ep}(t))\Big)\nep\overline{y}^{\ep}(t)(\overline{P}^{\ep})^{-1}(t)\Big)}{ \tau_{\ep}\ep^{\alpha-1}}}\,ds}\\
&&=\intom{\Phi^{\ep}(t):\nabla \eta^{\ep}(\overline{y}^{\ep}(t))},
\end{eqnarray*}
where
$$\Phi^{\ep}(t):=\frac{1}{\ep^{\alpha-1}}\int_0^1{DW_{el}\Big(\Big(Id+s\tau_{\ep}\ep^{\alpha-1}\nabla \eta^{\ep}(\overline{y}^{\ep}(t))\Big)\nep\overline{y}^{\ep}(t)(\overline{P}^{\ep})^{-1}(t)\Big)(\nep\overline{y}^{\ep}(t)(\overline{P}^{\ep})^{-1}(t))^T\,ds}.$$

Since $\overline{P}^{\ep}(t)\in L^2(\Omega;SL(3))$, $\det\,\overline{P}^{\ep}(t)=1$ a.e. in $\Omega$. Moreover, by (H1) and \eqref{quasst} we deduce that $\det\, \nep\overline{y}^{\ep}(t)>0$ a.e. in $\Omega$. On the other hand, since $\|\nabla \eta^{\ep}\|_{L^{\infty}(\Omega;\mthree)}\leq C$ for every $\ep$ (see Step 4 of the proof of Theorem \ref{cvstress} and \eqref{bddinverse}), by \eqref{ss49},
$$\det\,(Id+s\tau_{\ep}\ep^{\alpha-1}\nabla \eta^{\ep}(\overline{y}^{\ep}(t)))>0\quad\text{ for every }s\in [0,1],$$
for $\ep$ small enough. Hence, by combining \eqref{mandel2} and \eqref{lemmams} we deduce that $\Phi^{\ep}(t)$ is well defined for $\ep$ small enough.

Now, there holds
\begin{equation}
\label{star5}
\liminf_{\ep\to 0}\Big\{\intom{\Phi^{\ep}(t):\nabla \eta^{\ep}(\overline{y}^{\ep}(t))}\Big\}\geq 0.
\end{equation}

We claim that
\begin{equation}
\label{claimdisc}
\lim_{\ep\to 0}\intom{\Phi^{\ep}(t):\nabla \eta^{\ep}(\overline{y}^{\ep}(t))}=\intom{E(t)e_3:\partial_3\eta}.
\end{equation}
We note that, once \eqref{claimdisc} is proved, from \eqref{star5} it follows that
$$\intom{E(t)e_3:\partial_3\eta}\geq 0$$
for every {$\eta\in W^{1,\infty}(\R^3;\R^3)\cap C^{\infty}(\R^3;\R^3)$ such that $\eta=0\quad\cal{H}^2\text{ - a.e. on }\Gamma_d$}, hence the proof of \eqref{minimality} is complete. 

To prove \eqref{claimdisc}, it is enough to consider the sets
$$O_{\ep}(t):=\{x:\ep^{\alpha-1-\gamma}|\overline{F}^{\ep}(t)|<1\},$$
where the maps $\overline{F}^{\ep}(t)$ are the piecewise constant interpolants of the maps $F^{\ep}(t)$ defined in \eqref{deffept}. Arguing as in the proof of \eqref{badsetstress} and \eqref{goodsetstress}, one can show that, denoting by $\chi_{\ep}(t)$ the characteristic function of the set $O_{\ep}(t)$, there holds
$$||(1-\chi_{\ep}(t))\Phi^{\ep}(t)||_{L^1(\Omega;\mthree)}\leq C\ep^{\alpha-1}$$
and
$$\chi_{\ep}(t)\Phi^{\ep}(t)\deb E(t)\quad\text{weakly in }L^2(\Omega;\mthree).$$
 Claim \eqref{claimdisc} follows now arguing as in Step 4 of the proof of Theorem \ref{cvstress}.
\end{proof}
\noindent
\textbf{Acknowledgements.}
I warmly thank Maria Giovanna Mora for having proposed to me the study of this problem and for many helpful and stimulating discussions and suggestions.\\
This work was partially supported by MIUR under PRIN 2008.

\end{document}